\documentclass[reqno,11pt]{amsart}
\usepackage[margin=1in]{geometry}
\newcommand{\vvNumberWithin}{subsection}%'Theorem' (and variants) and equations are numbered by subsection

\RequirePackage{amsthm}
\RequirePackage{amsmath}
\RequirePackage{amsfonts}
\RequirePackage{amssymb}
\usepackage[all]{xy}
\RequirePackage{tikz-cd}
\RequirePackage{txfonts}%Use txfonts, after ams-env's to avoid errors.
\RequirePackage{lmodern}%lmodern is the default settings for fonts. Using txfonts only for symbols.
\RequirePackage{enumerate}%Allows us to get away with custom lists quickly.
\RequirePackage{hyperref}%Allows better (clickable) internal references
\RequirePackage{xstring}%Useful for processing of \tocnonum below
\newcommand*{\FIXME}[1]{}

\newcommand*{\Q}{\mathbb{Q}}

\newcommand*{\Z}{\mathbb{Z}}
\newcommand*{\C}{\mathbb{C}}
\DeclareMathOperator{\id}{Id} % Not exactly an operator, usually a functor.
\DeclareMathOperator{\Spec}{Spec}
\DeclareMathOperator{\Hom}{Hom}
\let\hom\relax% Set equal to \relax so that LaTeX thinks it's not defined
\DeclareMathOperator{\hom}{Hom}

\renewcommand{\L}{\mathbb{L}}

\newtheoremstyle{misc}%
     {\topsep}% space above
     {\topsep}% space below
     {}% body font
     {}% Indent amount (empty = no indent, \parindent = para indent)
     {\itshape}% Thm head font
     {}% Punctuation after thm head
     { }% Space after thm head (\newline = linebreak, { } = Normal word space. Space IS NEEDED)
     {}% Thm head spec

\newtheoremstyle{newdef}{\topsep}{\topsep}{}{}{\bfseries}{.}{ }{\thmnumber{#2}\thmnote{ #3}}

\newtheorem{master}{Master}[\vvNumberWithin]

\theoremstyle{newdef}
\newtheorem{npara}[master]{NamedParagraph}

\swapnumbers
\theoremstyle{plain}
\newtheorem{theorem}[master]{Theorem}%[\vvNumberWithin]
\newtheorem*{theorem*}{Theorem}

\newtheorem*{result*}{Result}
\newtheorem{lemma}[master]{Lemma}
\newtheorem*{lemma*}{Lemma}
\newtheorem{corollary}[master]{Corollary}
\newtheorem*{corollary*}{Corollary}
\newtheorem{proposition}[master]{Proposition}
\newtheorem*{proposition*}{Proposition}
\newtheorem{assumption-proposition}[master]{Assumption/Proposition}

\theoremstyle{definition}
\newtheorem{example}[master]{Example}
\newtheorem*{example*}{Example}

\newtheorem*{application*}{Application}
\newtheorem{definition}[master]{Definition}
\newtheorem*{definition*}{Definition}
\newtheorem{remark}[master]{Remark}
\newtheorem*{remark*}{Remark}
\newtheorem{para}[master]{}
\newtheorem*{para*}{}
\newtheorem{notation}[master]{Notation}
\newtheorem*{notation*}{Notation}

\newtheorem*{question*}{Question}

\newtheorem*{problem*}{Problem}

\theoremstyle{remark}

\newtheorem*{exercise*}{Exercise}

\makeatletter
\@addtoreset{master}{section}
\makeatother

\makeatletter
\@addtoreset{equation}{section}
\makeatother
\numberwithin{equation}{subsection}
 \renewcommand{\theequation}{%
      \ifnum\value{subsection} > 0
      \ifnum\value{subsubsection} > 0
      \thesubsubsection.\Alph{equation}%
      \else%
      \thesubsection.\Alph{equation}%
      \fi%
      \else% 
      \thesection.\Alph{equation}%
      \fi%
    }
\newcommand{\nocontentsline}[3]{}
\newcommand{\tocless}[2]{\let\tempcontentsline=\addcontentsline\let\addcontentsline=\nocontentsline#1{#2}\hspace{-1em}\let\addcontentsline=\tempcontentsline}
\newcommand{\tocnonum}[2]{{\tocless#1{#2} \addcontentsline{toc}{subsection}{#2}}}%Hard coding subsection, what is a better way?

\newcommand*{\vvspan}[1]{{\langle #1 \rangle}}
\newcommand*{\vvtspan}[1]{{\langle #1 \rangle}^\infty}
\newcommand*{\iso}{\cong}
\title{On weight truncations in the motivic setting}
\author{Vaibhav~Vaish}
\address{Stat-Math Unit, Indian Statistical Institute, 8th Mile, Mysore Road, Bangalore 560059}
\email{vaibhav\_if@isibang.ac.in}

\begin{document}
\begin{abstract}
	Using punctual gluing of $t$-structures, we construct an analogue of S.\@ Morel's weight truncation functors (for certain weight profiles) in the setting of motivic sheaves. As an application we construct a canonical motivic analogue of the intersection complex for an arbitrary threefold over any field $k$. We are also able to recover certain invariants of singularities motivically.
\end{abstract}
\maketitle
\tableofcontents
\section{Introduction}
\tocless\subsection{} In the conjectural picture of Beilinson \cite{beilinson1987height} one should be able to construct an abelian category of mixed motivic sheaves over any scheme $X$, related through the formalism of Grothendieck's six operations, and such that the relation with algebraic cycle homology theories is as expected. 

	By work of F. Morel and Voevodsky \cite{voevodsky_morel_99}, Jardine \cite{jardine_2000}, Ayoub \cite{ayoub_thesis_1, ayoub_thesis_2}, Cisinski--Deglise \cite{cisinski2012triangulated} and others, one has a tensor triangulated category which acts as the derived category of the abelian category of motivic sheaves, is equipped with the  Grothendieck's six functors (we work with $\Q$-coefficients throughout), and has the correct relationship with the algebraic cycle homology theories. Given a scheme $X$, we denote the corresponding category of motivic sheaves over $X$ as $DM(X)$ (for our purposes we can take the stable $2$-functor which attaches to any scheme $X$ the category of etale motivic sheaves without transfer, $DA(X,\Q)$ of Ayoub \cite[2.1]{ayoub2012relative} or, alternatively, the Beilinson motives, $DM_{B, c}(X,\Q)$ of Cisinski-Deglise \cite{cisinski2012triangulated} -- since we are working with rational coefficients, the two categories are known to be equivalent \cite[16.2.18]{cisinski2012triangulated}). Then the Beilinson's conjectures reduce to the problem of constructing an appropriate $t$-structure on this triangulated category, whose heart is the abelian category of mixed motivic sheaves. 
	
	While this appears beyond reach at the moment, various other structures on the realizations have been lifted to this setting of triangulated categories. For example, Bondarko formalizes the notion of weight structures in \cite{bondarko_weights} and they lift to the relative setting due to \cite{hebert2011structure} or \cite{bondarko2014weights}. In the setting of any mixed realization (in the sense of, say, \cite{saito_formalisme_2006}), the (perverse) $t$-structure and the weight structure can be used to construct a novel $t$-structure due to S.~Morel \cite[\S 3]{morelThesis} (degenerate, with trivial core, also a weight structure). Given that the weight structures on $DM(X)$ have already been constructed, a full construction of the analogue of Morel's $t$-structures on $DM(X)$ would be a significant progress towards construction of the motivic $t$-structure. 

	This article is motivated by the pursuit to construct analogue of certain specific Morel's $t$-structures in the setting of triangulated category of motivic sheaves. More precisely, we would work with the mild generalizations of Morel's $t$-structures constructed in \cite{arvindVaish} which are actually more useful in practice. While the construction cannot be made on all of $DM(X)$ or for  arbitrary choice of weights, the limited construction is good enough to have several interesting applications. 
	
\tocless\subsection{} Let us fix $k$ to be a finite field and $l\ne \text{char}(k)$. For any scheme $X$ of finite type over $k$, we work with the triangulated category of mixed complexes of $l$-adic sheaves over $X$, $D^b(X) = D^b_m(X, \Q_l)$ \cite{BBD} as the realization category of the triangulated category of motivic sheaves. Similar constructions would hold for $k=\C$ with $D^b(X)$ the derived category of mixed Hodge modules of Saito \cite{saito_introduction_1989, saito_mhm_1990} or $k$ a number field with $D^b(X)$ being the mixed category of S.~Morel \cite{morelOverQ}. 

 Let $X$ be a variety over $k$. In \cite[\S 3.1]{arvindVaish} one defines mild generalizations of Morel's $t$-structures \cite[\S 3]{morelThesis} -- for any monotone step function  $D:\{0,\ldots,\dim X\}\rightarrow \Z$ (that is $D$ is monotone, and $|D(r)-D(r-1)|\le 1$), we get a $t$-structure $({^wD^{\le D}}(X), {^wD^{> D}}(X))$ on $D^b(X)$. Here $D$ is to be thought of as a ``weight profile'', and $A\in {^wD^{\le D}}(X)$ is characterized by the fact that for all sufficiently fine stratifications $X=\sqcup_i S_i$, $A|_{S_i}$ is lisse whose cohomology sheaves, $\mathcal H^j(A|_{S_i})$, are local systems and $($weight of $\mathcal H^j(A|_{S_i})) \le (D(\dim S_i) - \dim S_i)$ for all $j$ and all $i$.

Let $\id$ denote the identity function. Of these $D$, those satisfying $\id\le D\le \dim X$ are the most interesting ones. For any such $D$ we get a distinguished complex of sheaves in $D^b(X)$, denoted as $EC^D_X := w_{\le D}IC_X$, where $w_{\le D}$ is the truncation for the negative part of the $t$-structure above and $IC_X$ is the intersection complex of $X$. Then these complexes can be thought of as approximations to $IC_X$ -- for $D=\dim X$ we have a natural identification $EC^{\dim X}_X \cong IC_X$ (see \cite[3.1.4]{morelThesis}).

Our main result then is to construct a motivic analogue for this $t$-structure for specific $D$ on appropriate subcategories of $DM(X)$. These subcategories will be large enough to recover the motivic analogue of $EC^D_X$, denoted here as $EM^D_X$. More specifically:% we have $D\in\{\id, \id+1\}$ for arbitrary $X$, and $D=F:=\{3\mapsto 3, 2\mapsto 3, 1\mapsto 2, 0\mapsto 2\}$ for $\dim X\le 3$.
	
%Let $DM^{coh}(X)\subset DM(X)$, denote the triangulated category of cohomological motives -- this is the dual of effective motives. Also if $\dim X\le 3$, we have a 
%category $DM^{coh}_{3, dom}(X)\subset DM(X)$ which we refrain from describing in the introduction. See \ref{FIXME} for detailed definitions of both the categories. Then:
\begin{theorem*}Let $k$ be an arbitrary field and $X$ be a scheme of finite type over $k$.
\begin{itemize} 
	\item (See \ref{tStructure:01}, \ref{application:EMX:01}) Let $D\in \{\id, \id+1\}$. Then there is a $t$-structure on $DM^{coh}(X)\subset DM(X)$ corresponding to $D$, denoted as $({^wDM^{\le D}}(X), {^wDM^{> D}}(X))$. This allows us to construct the motive $EM^D_X\in {^wDM^{\le D}}(X)\subset DM^{coh}(X)$.
	
	\item (See \ref{base:tStructureF}, \ref{application:EMX:F}) Let $\dim X \le 3$. Then there is a $t$-structure on $DM^{coh}_{3, dom}(X)\subset DM^{coh}(X)$ corresponding to $F=\{3\mapsto 3, 2\mapsto 3, 1\mapsto 2, 0\mapsto 2\}$, denoted as $({^wDM^{\le F}}(X), {^wDM^{> F}}(X))$, on $DM^{coh}_{3, dom}(X)$. This allows us to construct the motive $EM^F_X\in {^wDM^{\le F}}(X)\subset DM_{3, dom}^{coh}(X)$.
	
	\item (See \ref{application:weightTruncationsArePreserved}) These constructions play well with realizations -- assume that $k$ is a finite field, $k=\C$ or $k$ is a number field. Assume that for all schemes $X$ we have triangulated realization functors $real_X:DM(X) \rightarrow D^b(X)$ which commute with the four functors of Grothendieck and such that $real_X(\mathbf 1_X)=\Q_X$ ($D^b(X)$  as above, $\Q_X\in D^b(X)$ corresponds to the constant sheaf). Then:
	\[
		real_X(EM^D_X)\cong EC^D_X
	\]
	where either (a) $X$ is arbitrary and $D\in \{\id,\id+1\}$ or (b) $\dim X\le 3$ and $D=F$.
	\end{itemize}
\end{theorem*}
	Here $DM^{coh}(X)$ is the category of cohomological motives, dual of effective motives, and the subcategory $DM^{coh}_{3, dom}(X)\subset DM^{coh}(X)$ is obtained by imposing additional restrictions on dimensions (see \ref{definieCohMotives} for definitions).
	
	Also note that for the purpose of the last result, it is important to have realization functors to a mixed category (that is with a notion of weights) as discussed above. For example, Betti realization of Ayoub--Zucker \cite{ayoub2010note} is not good enough for our purpose. Such realizations are known for $k$ finite due to \cite{ivorra2007realisation}. Even though realizations for $k\hookrightarrow \C$ are also known due to \cite{ivorra2016perverse} but in that case the compatibility with the four functors of Grothendieck is not known. 

\tocless\subsection{}	While $EM^D_X$ are interesting objects by themselves, in special cases they would also recover the motivic analogue of $IC_X$. For example this happens for $\dim X=3$ and $D=F$, and leads to our main application:
\begin{theorem*}[See \ref{application:ICfor3folds}]
	Let $X$ be any variety with $\dim X=3$. and assume the situation in the previous proposition (that is, assume the existence of good realization functors). Then the object $EM^F_X\in DM(X)$ satisfies:
		\[
			real_X(EM^F_X) \cong EC^F_X \cong IC_X.
		\] 
	In particular, $IM_X:=EM_X^F$ is the motivic analogue of the intersection complex.
\end{theorem*}
	As mentioned before, due to Ivorra \cite{ivorra2007realisation} such realization functors are known for $k$ finite.
	
	Note that for $k=\C$ this construction also follows from \cite{muller2012relative}, however unlike their construction, even over $\C$ our construction is canonical enough to admit actions from correspondences on any open $U\subset X$. This is useful, for example, in the case of Shimura threefolds, where it allows one to lift Hecke actions to $IC_X$.
	\-\\

\tocless\subsection{}	Wildeshaus gives an internal characterization of intersection complex in the motivic settings in \cite[Definition 2.10]{wildeshaus_shimura_2012} (which is a slightly weaker notion than the proposed definition in \cite{wildeshaus_ic} but exists unconditionally in more cases, for example, the Baily-Borel compactification of an arbitrary Shimura variety). However under this definition, $IM_X$, even if it exists, is characterized only up to an isomorphism. On the other hand, $EM^D_X$ as constructed here are unique up to a unique isomorphism. Nevertheless, they bear the correct relationship with motivic intersection complex for $D=\id, \id+1$:
\begin{theorem*}[See \ref{application:invariants}]
	Let $k$ be a field, and let $X$ be an algebraic variety over $k$. Let $D$ be $\id$ or  $\id+1$. Then
	\[
		EM^D_X \cong w_{\le D} IM_X.
	\]
	where $IM_X$ is the intersection complex $j_{!*}1_U$ in the sense of \cite[Definition 2.10]{wildeshaus_shimura_2012}. The relationship holds whenever $IM_X$ exists (e.g. for $X$ the Baily-Borel compactification of a locally symmetric Hermitian space by \cite{wildeshaus_shimura_2012}).
\end{theorem*}

	Presently it is a limitation of our approach that this comparison is not applicable for $D=F$. In particular, presently we cannot show that the motivic intersection complex as constructed here satisfies the internal characterization of Wildeshaus. At the very least, it should be possible to show that the motivic intersection complex here gives rise to a Chow motive (that is, is of weight $0$ in the sense of Bondarko \cite{bondarko_weights}), however this would require a more careful calibration of our approach and is left for a future work.
		
\tocless\subsection{} The construction of $EM_X^D$ for $D=\id$ is not new and already appears in \cite{vaish2016motivic} (where it was also denoted as $EM_X$) and in \cite{ayoub2012relative} (where it was denoted as $\mathbb E_X$). This already had interesting applications in the context of Shimura varieties, and we expect the same to hold for $D=\id+1$. 

It is also possible to think of $EM_X^D$ for $D=\id,\id+1$ as invariants of singularities of $X$. For example, we have the following:
\begin{proposition*}[See \ref{application:EMX:01}] Let $\pi:Y\rightarrow X$ be any resolution of singularity. Let $D=\id$ or $D=\id+1$. Then:
	\[
		EM_X^D \cong w_{\le D}\pi_*1_Y.
	\]
\end{proposition*}
	In particular the truncation on the right -- which should be thought of as a piece of the motive of the resolution $Y$ of $X$ -- is independent of the chosen resolution.

	One can also show that $w_{\le D}$ for $D=\id, \id+1$ factors through ``birational motives'' (see \ref{application:wIsBirational}). Then the methods here can be used to construct motivic analogues of several (cohomological) invariants of singularities occurring in literature. For example, one can construct the motivic analogue of boundary complexes of \cite{payne2013boundary}, or show that the motivic $H^1(-)$ of any fiber is independent of resolutions, see \ref{example:boundary}, \ref{example:h1} for details.

\tocless\subsection{} At the heart of this article is the notion of punctual gluing of $t$-structures \cite{vaish2017punctual} which, under certain conditions, allows one to construct a $t$-structure on (the triangulated category of motives on) $X$, given a $t$-structure at (the triangulated category of motives on) each Zariski point $x\in X$. Thus, in presence of punctual gluing, the construction of Morel's $t$-structure on a base $X$ reduces to analogous constructions over $\Spec k$, for $k$ any field. This can be further reduced to the situation when $k$ is perfect using separatedness and continuity. We briefly motivate the constructions in this setting below. 

Let us work in the settings of realizations as before. Assume $\pi:Y\rightarrow \Spec k$ is smooth, proper. Then $H^i(Y)=R^i\pi_*\Q_Y$ is pure of weight $i$. Therefore we have:
\[
	w_{\le i}R\pi_*\Q_Y \cong \tau_{\le i}R\pi_*\Q_Y,
\]
	that is Morel's characterization in terms of weights is reduced to a characterization in terms of degrees. We also have an identification, using decomposition theorem (since $Y$ is smooth, proper):
	\[
		\tau_{\le i}R\pi_*\Q_Y = \bigoplus_{0\le j\le i} H^j(Y)[-j].
	\]
	
	Motivically, one expects a Chow-K\"unneth decomposition (see Murre's conjecture $A$, \cite[1.4]{murre1993conjectural}) on the motive of $Y$ in the category of Chow motives $CHM(k)$, which corresponds to the piece $H^j(Y)[-j]$. While the full Chow-K\"unneth decomposition is not known, one knows motivic $H^0(X)$ and $H^1(X)$ for arbitrary varieties and motivic $H^2(X)$ for surfaces. Therefore it is possible to construct motivic analogue of $w_{\le i}R\pi_*\Q_Y = \tau_{\le i}R\pi_*\Q_Y$ for $i=0, 1$ and $Y$ arbitrary smooth proper, or even $i=2$ and $Y$ smooth proper with $\dim Y\le 2$. There is an embedding of Chow motives $CHM(k)$ into triangulated category of motives $DM(k)$ for $k$ perfect, and hence this truncation can be realized in $DM(k)$.
	
	Let us assume $i=0$ or $i=1$. Then the motive $w_{\le i}\pi_*1_Y$ can be defined by what we saw above. By an explicit version of Brown representability \ref{tFromGenerators}, we will extend this truncation to the full subcategory finitely generated by motives of the form $\pi_*1_Y$, the category of cohomological motives $DM^{coh}(k)$. The technical input for this extension is the explicit computations of motivic cohomology $H^{p,q}_{\mathcal M}(X)$ for $q=0, 1$ which are well known. While $DM^{coh}(k)$ is not the full category $DM(k)$, it does satisfy the formalism of gluing. Therefore using punctual gluing we get a $t$-structure on $DM^{coh}(X)$ corresponding to the weight profile $D=\id+i$.
	
	If $i=2$, we have to additionally restrict the dimensions of $Y$ that occur. Furthermore, to accommodate formalism of gluing (Grothendieck's four functors in particular), one needs to allow arbitrary negative Tate twists of $\pi_*1_Y$ as well. Then we have:
	\begin{align*}
		w_{\le 2}(R\pi_*\Q_Y(-1)) = (w_{\le 0}R\pi_*\Q_Y)(-1)& &w_{\le 2}(R\pi_*\Q_Y(-r)) = 0\text{ for }r\ge 2
	\end{align*}
	 and hence can also be constructed motivically by what we discussed above. This allows us to extend the $t$-structure to a subcategory $DM^{coh}_2(k)\subset DM^{coh}(k)$ (see \ref{definieCohMotives} for precise definition) -- the additional relations on motivic cohomology required for this extension are also well known. Then, by punctual gluing, we get the $t$-structure corresponding to weight profile $F$ (on an appropriate subcategory restricted by dimensions, $DM^{coh}_{3, dom}(X)\subset DM^{coh}(X)$), as required.
	 
	Finally one notes that the categories $DM^{coh}(X)$ and $DM^{coh}_{3,dom}(X)$ contain the object $j_*1_U$ for $j:U\hookrightarrow X$ an immersion onto an open dense subset, $U$ regular. Then the relation $EM_X^D=w_{\le D}j_*1_U$, for $D\in\{\id, \id+1, F\}$ is formal and does not depend on choices. 

\tocnonum\subsection{Outline}
Section \ref{sec:prelims} contains mostly independent subsections which are preliminaries -- most of them are already well known or implicit in the literature. Section \ref{sec:tStructThroughGen} defines a $t$-structures beginning with truncation of generators, while section \ref{sec:gluing} recalls how to glue $t$-structures in the presence of continuity (``punctual gluing''). Section \ref{sec:chow} recalls Chow motives and mostly serves as a matter of fixing notations. Section \ref{sec:murre} recalls Murre's projectors which are the key technical input in our construction. Section \ref{sec:geometry} recalls some facts from geometry (Stein factorization, deJong's alterations) which are useful for us. Section \ref{sec:motives} recalls facts about triangulated category of motivic sheaves which will be the category in which our constructions are made.

In section \ref{sec:morel} we construct analogue of Morel's $t$-structures in the motivic setting. In section \ref{sec:cohMotives} one defines the categories of cohomological motives, the basic arena we will be working in. While most of the results here have already appeared in the literature, we have to redo them with a control on dimension, which is what is required later for constructing the analogue of Morel's truncations useful for recovering the motivic intersection complex of threefolds. In section \ref{sec:overField} one constructs the analogue of (certain) Morel's $t$-structure over a perfect field, the key input being existence of certain Murre's projectors and vanishing of certain motivic cohomology groups. Finally, in section \ref{sec:overBase} one glues these $t$-structures to construct corresponding analogue of Morel's $t$-structure over a base. 

Section \ref{sec:applications} contains the applications. %In the short sections, section \ref{sec:artin} and section \ref{sec:picard} we relativize the Artin and Picard motive of Ayoub-Viale \cite{ayoub20091motivic}. In section \ref{sec:1motivic} we show that the relative $1$-motivic $t$-structure of \cite{lehalleur2015motivic} respects compact objects. 
In section \ref{sec:invariants} we demonstrate that the truncations $w_{\le\id}$ and $w_{\le\id+1}$ factor through birational motives and use them to construct several invariants, including $EM_X^{\id}$, $EM_X^{\id+1}$. We also demonstrate the expected relationship of these objects with Wildeshaus' intersection complex \cite{wildeshaus_shimura_2012}. Finally, in section \ref{sec:motivicIC} we construct the motivic $IC_X$ for any $3$-fold $X$, and show that the motivic $IC_X$ as well as motives $EM_X^{\id}$, $EM_X^{\id+1}$ have the expected realizations. 

\tocless{\subsection*}{Acknowledgement} Part of this research was supported by the grant under INSPIRE fellowship scheme, DST/INSPIRE/04/2015/000120. I would also like to thank Arvind Nair and Simon Pepin Lehalleur for useful discussions and Indian Statistical Institute, Bangalore for the hospitable atmosphere. I would also like to thank the anonymous referee for the pointed and important suggestions which have improved the readability of this article.

\tocless{\subsection*}{Notation}\hspace{1em} All schemes $X$ will be separated of finite type over a base field. For a scheme $X$, $X_{red}$ denotes the underlying reduced scheme. By $\Spec k\hookrightarrow X$ we mean a Zariski point $x=\Spec k$ in $X$. A locally closed $Z\subset X$ will always be given the reduced induced sub-scheme structure. 

In \S \ref{sec:chow} and \S \ref{sec:motives} we will be considering categories of Chow motives and motivic sheaves respectively, and they will be considered only with $\Q$ coefficients, this is primarily for two reasons: Firstly the Chow-Kunneth factors we consider in \S \ref{sec:murre} are only rationally defined. Secondly the full formalism of Grothendieck's six functors \ref{motives:6functors} is available with integral coefficients (to be precise, for $\Z[1/p]$ coefficients where $p$ is the exponential characteristic of $k$, see \cite{cisinski2015integral})  only for Noetherian equi-characteristic schemes of finite type over $k$ when $k$ is perfect: this would complicate computations for us while using punctual gluing (which involves considering motivic sheaves over generic points) if we do not work rationally.

\section{Preliminaries}\label{sec:prelims}
\subsection{$t$-structures through generators}\label{sec:tStructThroughGen} A $t$-structure on a triangulated category $D$ is a pair of full subcategories $(D^{\le t}, D^{>t})$ satisfying three properties (see \cite{BBD}):
\begin{itemize}
\item (Orthogonality) $\hom (a,b)=0, \forall a\in D^{\le t}, b\in D^{>t}$
\item (Invariance) $D^{\le t}[1]\subset D^{\le t}$, and $D^{> t}[-1]\subset D^{>t}$
\item (Decomposition) $\forall c\in D$, there is a distinguished triangle $a\rightarrow c\rightarrow b\rightarrow $ with $a\in D^{\le t}$ and $b\in D^{>t}$. 
\end{itemize}
The decomposition of $c$ in the last property can be shown to be unique and then $a$ (resp. $b$) is often denoted as $\tau_{\le t}c$ (resp. $\tau_{>t}c$). It is easy to show that $\tau_{\le t}$ (resp. $\tau_{>t}$) is a right (resp. left) adjoint to the fully faithful inclusion $D^{\le t}\hookrightarrow D$ (resp. $D^{>t}\hookrightarrow D$). We will often refer to $D^{\le t}$ as the negative part of the $t$-structure and $D^{>t}$ as the positive part.

We will be frequently interested in constructing $t$-structures where $D^{\le t}$ and $D^{>t}$ are triangulated subcategories of $D$ (in particular the core of the $t$-structure, $D^{\le t}\cap D^{>t}[1]$ by definition, is trivial). In this setting, to construct a $t$-structure is equivalent to give an adjoint:
\begin{proposition}\label{tIsAdjoint}
	Let $D^{\le t}$ be a full triangulated subcategory of a triangulated category $D$. Then to give a $t$ structure $(D^{\le t}, D^{> t})$ on $D$ is to give a right adjoint $\tau_{\le t}$ to the inclusion $D^{\le t}\hookrightarrow D$. In this case, $D^{> t}$ can be identified to the subcategory of objects $x$ in $D$ such that $\tau_{\le t}(x)=0$.
\end{proposition}
\begin{proof}
	See \cite[2.6]{vaish2017punctual}, for example.
\end{proof}
\begin{para}
Given a full triangulated subcategory $D^{\le t}\hookrightarrow D$, there are broadly two ways to construct an adjoint. One is to use Brown representability: if $D$ has small direct sums and $D^{\le t}$ is compactly generated then the existence of an adjoint is automatic (see \cite[8.4.4]{neeman2014triangulated}). However this adjoint is unwieldy, and the real hard work is to show that it preserves compact objects (which, in the motivic categories, correspond to the geometric or constructible objects).

The other method is to begin with an explicit set of objects for which the three defining properties of the $t$-structure is known, and use this explicit construction to conclude the same for the smallest triangulated subcategory finitely generated by these objects. This is not any weaker than the method above (see \ref{compactToNonCompact}) and has the added advantage that constructions are somewhat explicit and have good properties with respect to the compact objects. This method is detailed below.
\end{para}
We make the following definitions:
\begin{definition}
	Let $D$ be a pseudo-abelian triangulated category. Let $A, B, H\subset D$ be full subcategories containing $0$ which are isomorphism closed -- that is if $x\cong y$ for some $x\in D$ and $y\in H$ (resp.  $y\in A$, resp. $y\in B$) then $x\in H$ (resp. $x\in A$, resp. in $x\in B$).
	Then we define the following full subcategories:
	\begin{align*}		
		Ext^1(B,A)	&:= \{ d\in D \big|\exists a\rightarrow d\rightarrow b\rightarrow  \text{ distinguished }\text{with }a\in A,b\in B\} \hspace{-30em}\\
	Ext^1(H) &:= Ext^1(H,H)	& 
		Ext^n(H)	&:= Ext^1(Ext^{n-1}(H))\;\forall n\ge 2\\
	Ext(H)&:=\cup_{i\ge 1}{Ext^i(H)} &
		H^\infty	&:= \{\oplus_{i\in I} h_i\big|h_i\in H, I\text{ small }\}\\	
	Ext^\infty (H) &:=Ext(H^\infty)  &
		H[\Z] &:= \{ h[n]\big| h\in H,n\in \Z\}\\
	\mathcal R(H)&:=\{d\in D \big| d\text{ is a retract of some }h\in H\} \hspace{-30em}\\
	\vvspan{H} &:= \mathcal R(Ext(H[\Z])) &
		\vvtspan{H} &:= Ext^\infty(H[\Z])
	\end{align*}
	(Here $H^\infty, Ext^\infty, \vvtspan{H}$ are defined when $D$ has small direct sums). Also say:
	\begin{align*}
		&\text{$A\perp B$ if $\hom_D(a,b)=0, \forall a\in A, b\in B$.} &  &\text{$(A,B)$ \emph{decompose} $H$ if $H\subset Ext^1(B,A)$.}\\
		&\text{$D$ is \emph{finitely generated by} $H$ if $D=\vvspan H $.} & &\text{$D$ is \emph{generated by} $H$ if $D=\vvtspan H $.}
	\end{align*} 
%%		We also make the following definitions: %
%%
%%	$A\perp B$ if $\hom_D(a,b)=0, \forall a\in A, b\in B$. %
%%	
%%	$(A,B)$ \emph{decompose} $H$ if $H\subset Ext^1(B,A)$. %
%%	
%%	$D$ is \emph{finitely generated by} $H$ if $D=\vvspan H $ and $D$ is \emph{generated by} $H$ if $D=\vvtspan H$.
\end{definition}

\begin{remark} It will be useful to define $\mathcal R(H)$ for any pseudo-abelian additive category. 
\end{remark}

\begin{remark}
	In some places in literature (e.g. in \cite{hebert2011structure}), $Ext^\infty(H)$ is defined as the full subcategory $Ext(Ext(H)^\infty)$. The two definitions are equivalent, see \cite[2.10]{vaish2017punctual}.%: $Ext(H)^\infty \subset Ext(H^\infty)$ since an arbitrary co-product of distinguished triangles is a distinguished triangle, and hence $Ext(Ext(H)^\infty) \subset Ext(Ext(H^\infty))=Ext(H^\infty)$. $Ext(H^\infty)\subset Ext(Ext(H)^\infty)$ since $H\subset Ext(H)$, giving the other inclusion. 
\end{remark}
\begin{remark}	
 	$\mathcal R(\vvtspan H)=\vvtspan H$ because the image of a projector can be constructed using homotopy colimits (see \cite[1.6.8]{neeman2014triangulated}) and $\vvtspan H $ is closed under countable direct sums and cones, hence under homotopy colimits.
\end{remark}
Then we have the following:
\begin{proposition}\label{tFromGenerators}
	If $A\perp B[n]$ for all $n\in \Z$, and $(A,B)$ decompose $H$, with $\vvspan A, \vvspan B\subset \vvspan H$. Then $(\vvspan A , \vvspan B )$ is a $t$-structure on the triangulated subcategory $\vvspan H $. If in addition objects in $A$ are compact, and $D$ has small direct sums, then $(\vvtspan A , \vvtspan B )$ is a $t$-structure on the triangulated subcategory $\vvtspan H$. 
\end{proposition}
\begin{proof}
	As $A\perp B[n]$ for all $n$, hence $A[n]\perp B[m]$ for all $n, m$ and hence $A[\Z]\perp B[\Z]$. That these are triangulated subcategories is easy and hence invariance is obvious. Orthogonality is \ref{gluing:orthogonality} and decomposition is \ref{gluing:decomposition} applied to $A[\Z]$, $B[\Z]$ and $H[\Z]$. 
%	
%	Finally the last statement follows, by taking $B=D^{>}$ and $H=Ext^1(B, A[\Z])\subset D'$. Since $D'\supset \vvspan H\supset Ext(H) \supset Ext^1(H, Ext(H)) \supset Ext^1(B, Ext(A))\supset D'$, so $\vvspan H=D'$.
\end{proof}
\begin{lemma}\label{gluing:orthogonality}
	If $A\perp B$ then $Ext^1(A)\perp B$ and $A\perp Ext^1(B)$. Hence $Ext^n(A)\perp Ext^n(B)$ and $Ext(A)\perp Ext(B)$. If objects in $A$ are compact, then $A^\infty\perp B^\infty$ and so $Ext^\infty(A)\perp Ext^\infty(B)$.
\end{lemma}
\begin{proof} If $A\perp B$, then $A^\infty \perp B^\infty$ provided objects in $A$ are compact. Therefore, enough to show that $Ext^1(A)\perp B$ (resp. $A\perp Ext^1(A)$) -- other claims follow by an obvious induction. But $Ext^1(A)\perp B$ (resp. $A\perp Ext^1(A)$) if $A\perp B$, writing long exact sequence for $\hom(-,b)$ (resp. $\hom(a, -)$) for any $b\in B$ (resp. $a\in A$). 
\end{proof}
\begin{lemma}\label{gluing:decomposition}
	Let $A, B, H$ be isomorphism closed full triangulated subcategories of $H$ containing $0$ such that $(A, B)$ decompose $H$.
	\begin{enumerate}
		\item Assume $A\perp (B\cup B[-1])$. If $D$ is pseudo-abelian then $(\mathcal R(A), \mathcal R(B))$ decompose $\mathcal R(H)$.
		\item Assume $A\perp B[1]$. Then $(Ext(A),Ext(B))$ decompose $Ext(H)$.
		\item Assume $A\perp B[1]$. If $D$ has small direct sums and all the objects in $A$ are compact then $(Ext^{\infty}(A),Ext^{\infty}(B))$ decompose $Ext^{\infty}(H)$.
	\end{enumerate}
\end{lemma}
\begin{proof}
	\begin{enumerate}
		\item  Let $h'\in \mathcal R(H)$. Then there is $h\in H$ and a decomposition of identity $h'\xrightarrow \alpha h\xrightarrow\beta h'$ and hence $p:\alpha\circ\beta:h\rightarrow h$ is a projector. There is a triangle: 
		\[
			a\overset f\longrightarrow h \longrightarrow b \longrightarrow 
		\]
		 with $a\in A$ and $b\in B$. Using $A\perp B$ it is easy to see that $p$ gives rise to a morphism $p':a\rightarrow a$. Since $A\perp B[-1]$ this is unique and hence gives rise to a projector $p'$. Let the kernel of $p'$ be $a'$ and that for the projector $1-p,1-p'$ be $h'',a''$ respectively. Then $f$ induces maps $f':a'\rightarrow h'$ and $f'':a''\rightarrow h''$. We get an induced morphism of distinguished triangles:
		\[\xymatrix{
			a'\oplus a''\ar[r]^{f'\oplus f''}\ar[d]		&h'\oplus h''\ar[r]\ar[d]	&b'\oplus b''\ar[d]\ar[r]	&\-	\\		 
		 	a\ar[r] 				&h\ar[r] 			&b\ar[r]	&\-					
		}\]
		 where $b',b''$ are defined as cones of the morphism $f',f''$. First two vertical maps are isomorphisms, hence so is the third. In particular, $b'$ is a summand of $b$, hence in $\mathcal R(B)$. Since $a'$ is in $\mathcal R(A)$ and $a'\rightarrow h'\rightarrow b'\rightarrow$ is distinguished, we are done.
	\item 	Let $h\in Ext^1(H)$. Then there is a distinguished triangle $h_1\rightarrow h\rightarrow h_2\rightarrow h_2[1]$ in $D$ with $h_1, h_2$ in $H$. Since $H\in Ext^1(B,A)$, we have a commutative diagram:
\[
\xymatrix{%a_1\ar[r] & h_1\ar[r]\ar[d] & b_1\ar[r] & a[1]\\
a\ar@{..>}[d]\ar@{..>}[r] & h\ar[d]\ar@{..>}[r] & b\ar@{..>}[d]\ar@{..>}[r] & a[1]\ar@{..>}[d]\\
a_2\ar[r]\ar@{-->}[d] & h_2\ar[r]\ar[d] & b_2\ar[r]\ar@{-->}[d] & b[1]\ar@{-->}[d]\\
a_1[1]\ar[r] & h_1[1]\ar[r] & b_1[1]\ar[r] & a[2]
}
\]
	Here (bottom) dashed arrows exist, since $\hom(a_2,b_1[1]) =0$. Defining $a$ by the triangle $a\rightarrow a_2\rightarrow a_1[1]\rightarrow $ and similarly for $b$, we complete the diagram using dotted arrows. Hence $a\in Ext^1(A)$, $b\in Ext^1(B)$, and so $Ext^1(H)$ is decomposed by $(Ext^{1}(A),Ext^{1}(B))$. Claim for $Ext^n$ and $Ext$ follows by induction using \ref{gluing:orthogonality}.
	\item For $Ext^\infty$, we only need to additionally observe that if objects in $A$ are compact $A^\infty \perp B^\infty[1]$. Also, since arbitrary sum of distinguished triangles is distinguished $(A^\infty,B^\infty)$ is a decomposition structure on $H^\infty$. Now we use the the previous part.
	\end{enumerate}
\end{proof}

In particular if we can construct a $t$-structure on $\vvspan H$ invariant under shifts and $H$ consisting of compact objects, we can do so on $\vvtspan H$:
\begin{corollary}\label{compactToNonCompact}
		Let $(A,B)$ form a $t$-structure on $H$ and assume that $A$ consists of compact objects, $A[\Z]\subset A$. Then $(Ext^\infty(A), Ext^\infty(B))$ forms a $t$-structure on $Ext^\infty(H)$.
\end{corollary}
In the motivic setting, it is standard to lay down results in the situation of (subcategories) of the form $\vvtspan S$ where $S$ is a compact set. Then the claim that a (shift-invariant) $t$-structure restricts to compact objects is equivalent to saying that we have a $t$-structure on $\vvspan S$ by the previous corollary and the following:
\begin{proposition}\label{compactAreGenerated}
	Assume that $D=\vvtspan S$ and that the objects in $S$ are compact. Then the full subcategory of compact objects in $D$ can be identified with $\vvspan S$.
\end{proposition}
\begin{proof}
	This follows immediately from \cite[4.4.5]{neeman2014triangulated}.
\end{proof}

\subsection{Gluing $t$-structures in presence of continuity}\label{sec:gluing}
	In \cite{vaish2017punctual} we laid down a procedure for so called \emph{punctual gluing} which is a mild generalization of the standard gluing of $t$-structures in \cite{BBD}. We briefly recall the same below.

%A general procedure of gluing $t$-structures was laid down in \cite{BBD}: if $D(X)$ denotes the (bounded) triangulated category of sheaves on $X$, then one can construct a  $t$-structure on $D(X)$ given a $t$-structure each on $D(Z)$ and $D(U)$ where $Z\subset X$ is closed and $U$ is its open complement. This can be generalized to yield gluing on any fixed stratification by induction. Under fairly mild conditions (``constructibility'' of objects under consideration), one can construct a stratification independent version of it (which was implicitly used in \emph{loc.~cit.} to construct the perverse $t$-structure). 
%
%While constructibility is usually formulated as property of an individual sheaf, it is better seen to manifest as ``continuity'', a property of the category of constructible sheaves. Through a simple argument, it is possible to show that, in presence of continuity, to define a $t$-structure on $D(X)$ one only needs a $t$-structure on the triangulated categories $D(x)$ for $x$ any (Zariski) point of $X$. This is made precise below.

We begin with describing the formal situation we work in:
	\begin{definition}[Grothendieck's four functors]\label{gluing:4functors}
		Given a scheme $X$ let $Sub(X)$ be the category of sub-schemes of $X$ (i.e.. schemes $f:Y\rightarrow X$, $f$ an locally-closed immersion as objects and obvious morphisms). In particular, all morphisms in $Sub(X)$ are immersions, of finite type. 
		
		We say \emph{formalism of Grothendieck's four functors} holds if:
		\begin{enumerate}[i.]
			 \item For any $Y$ in $Sub(X)$ we are given a triangulated category $D_Y$. 
			 \item For any morphism $f:Y\rightarrow Z$ in $Sub(X)$, we are given adjoint functors
			 	\begin{align*}
					f^*: D_Y \leftrightarrows D_Z : f_* & &\text{ and }& &
					f_!: D_Y \leftrightarrows D_Z : f^!
				\end{align*}
				such that there are isomorphisms of functors:
				\begin{align*}
					(f g)_* \iso f_*g_* & & (f g)_! \iso f_!g_! & & (f g)^* \iso g^*f^* & & (f g)^! \iso g^!f^!
				\end{align*}
				Also if $1:X\rightarrow X$ is the identity morphism, and $\id:D_X\rightarrow D_X$ denotes the identity natural transformation. Then we must have an isomorphism of functors:
				\begin{align*}
					1_* \iso 1_! \iso 1^* \iso 1^! \iso \id
				\end{align*}
				
			\item We are given natural transformations
				\begin{align*}
					f^!	\rightarrow f^* & & f_!	\rightarrow f_*
				\end{align*}
				where the first is an isomorphism for $f$ an open immersion, and the second an isomorphism for $f$ a closed immersion.
			\item If $i:Z\rightarrow Y$ is a closed immersion with $j:U\rightarrow Y$ the open immersion of the complement, then the functors $(i^*,i_*,i^!),(j_!,j^*,j_*)$ follow the formalism of gluing \ref{gluing:openClosed}.
		\end{enumerate}
	\end{definition}
	\begin{definition}[Formalism of gluing]\label{gluing:openClosed}
		The triangulated functors:
			\begin{align*}
				j_*, j_!: D_U	\leftrightarrows D_X	: j^* & & i_*: D_Z	\leftrightarrows D_X	: i^*, i^!
			\end{align*}
		between triangulated categories $D_U,D_Z,D_X$ are said to satisfy \emph{formalism of gluing} if:
		\begin{enumerate}[i.]
			\item $(j_!, j^*)$, $(j^*, j_*)$, $(i^*, i_*)$ and $(i_*, i^!)$ are adjoints tuples.
			\item $j^*i_*=0$ and hence by adjunction $i^*j_! = 0$ as well as $i^!j_* = 0$.
			\item (Localization)  $\forall A\in D_X$, we have functorial distinguished triangles:
			\begin{align*} j_!j^*A\rightarrow A\rightarrow i_*i^*A\rightarrow & & i_*i^!A\rightarrow A\rightarrow j_* j^*A\rightarrow \end{align*}
			where the morphisms come from adjunction.
			\item We have isomorphism of functors
			\begin{align*} \forall A\in D_U,\;j^*j_*A\xrightarrow\cong A\xrightarrow\cong j^*j_!A & & \forall B\in D_Z,\;i^*i_* B\xrightarrow\cong B\xrightarrow\cong i^!i_*B  \end{align*}
			with the morphisms coming from adjunction.
		\end{enumerate}
	\end{definition}
	We will need $f^!$ and $f^*$ defined not only for locally closed immersions, but also when $f:\Spec K\hookrightarrow X$ denotes any (Zariski) point of $X$ (that is $K$ is the residue field of $x\in X$, and we consider the induced map $\Spec K\rightarrow X$). While, in the situations we intend to apply this, $f^*$ is a given, $f^!$ needs to be defined by hand. We make this precise below:
	\begin{definition}[Extended formalism of gluing]\label{gluing:extended4functors}
		Let $X$ be a scheme and let us assume that Grothendieck's four functors \ref{gluing:4functors} exist. Let $f_Y:Y\rightarrow X$ denote the natural immersion for any $Y\in Sub(X)$. The situation is said to satisfy \emph{extended formalism of gluing} if the following happens:
		
		Assume that for each Zariski point $\Spec k\hookrightarrow X$ we are given a triangulated category $D(k)$. Let $y=\Spec k$ denote the corresponding point in $X$ and let $Y=\bar y$. Let $Z\in Sub(X)$ be such that $y\in Z$ as well. Let $\epsilon_Z:\Spec k\hookrightarrow Z$ denote the natural morphism. Assume that in such a situation we are given a functor
			\begin{align*}
				\epsilon_Z^*:D_Z\rightarrow D(k)&  &\text{ such that }\epsilon_Z^* = \epsilon_{Z'}^{*}\circ f^*\text{ for all }\Spec k\xrightarrow{\epsilon_{Z'}} Z'\xrightarrow f Z\text{ factoring }\epsilon_Z
			\end{align*}

		Since $\bar y = Y$ therefore $Y\subset \bar Z$ and $i:Y\cap Z\hookrightarrow Z$ is a closed immersion. Let $\epsilon_{Y\cap Z}:\Spec k\rightarrow Y\cap Z$ be the natural morphism. Then we define:
		\[
			\epsilon_Z^! := \epsilon_{Y\cap Z}^*\circ i^!: D_Z\rightarrow D_{Y\cap Z}\rightarrow D(k)
		\]
		Notice that if $\Spec k\xrightarrow{\epsilon_{Z'}} Z'\xrightarrow f Z$ factors $\epsilon_Z$, with $i':Y\cap Z'\hookrightarrow Z'$ being the immersion, then $Y\cap Z\subset Y=\bar y$ lies in closure of $Y\cap Z'$, since $y\in Y\cap Z'$. Since $Y'\cap Z\subset Y\cap Z$ is also locally closed, it must be open in $Y\cap Z$. It follows that if $j:Y\cap Z'\rightarrow Y\cap Z$ denote the immersion, $j^*=j^!$, and hence we have a canonical identification of functors: 
		\[
			\epsilon_Z^! = \epsilon_{Y\cap Z}^*\circ i^! = (\epsilon_{Y\cap Z'}^*\circ j^*)\circ i^! = \epsilon_{Y\cap Z'}^*\circ (i\circ j)^! =\epsilon_{Y\cap Z'}^*\circ (f\circ i')^!=\epsilon_{Y\cap Z'}^*\circ i'^!\circ f^!=\epsilon_{Z'}^!\circ f^!
		\]
	\end{definition}
	
	Finally, we make the essential definitions:
	\begin{definition}[Continuity]\label{gluing:continuity}
		Assume the situation in \ref{gluing:extended4functors}.  Let $\epsilon_Y:\Spec k\hookrightarrow X$ be a Zariski point in $X$ with closure $Y$. The situation is said to satisfy \emph{continuity} if for any such $\epsilon_Y$:
		\begin{enumerate}[i.]
			\item (Essentially surjective) Let $a\in D(k)$. %factoring through $\epsilon_Z:\Spec k\rightarrow Z$. Let $a\in D(j)$. 
			Then $\Spec k$ has a neighborhood $\Spec k\xrightarrow{h}U\subset Y$  open dense in $Y$ and an object $\bar a\in D_U$ such that $h^*(\bar a) = a$.
			\item (Full) Let $a, b\in D_Y$. Then for any morphism $\alpha\in \hom(\epsilon_Y^*a,\epsilon_Y^*b)$, there is a $\Spec k\xrightarrow{g} U\xrightarrow{f} Y$ open dense, and a map $\bar\alpha\in \hom(f^*a, f^*b)$, such that $\alpha = g^*\bar\alpha$.
			\item (Faithful) Let $a, b\in D_Y$. Then for any morphism $\bar\alpha\in \hom(a,b)$, such that $\epsilon^*(\bar \alpha)=0$, there is a $\Spec k\xrightarrow{g} U\xrightarrow{f} Y$ open dense, $f^*(\bar\alpha)=0$.
		\end{enumerate}
	\end{definition}
	\begin{definition}[Punctual gluing]\label{gluing:spreadingOut}
		Assume that for each $\Spec k\hookrightarrow X$, we are given a $t$-structure $(D^{\le}(k), D^{>}(k))$ on the category $D(k)$. For any $U\in Sub(X)$, define
		\begin{align*}
			D^{\le}(U)	:=\{a&\in D_U\big| \epsilon^*(a)\in D^{\le}(k)\text{ for }\epsilon:\Spec k\rightarrow U\text{ any point of }U\} \\
			D^{>}(U)	:=\{a&\in D_U\big| \epsilon^!(a)\in D^{>}(k)\text{ for }\epsilon:\Spec k\rightarrow U\text{ any point of }U\}
		\end{align*}
		as full subcategories. In particular if $f:S\hookrightarrow T$ is an immersion, $S, T\in Sub(X)$:
		\begin{align*}
			f^*(D^{\le}(T))\subset D^{\le}(S)& &f^!(D^{>}(T))\subset D^{>}(S)
		\end{align*}
	\end{definition}
	\begin{definition}[Continuity for $t$-structures]\label{gluing:continuityForT}
		Assume the situation of \ref{gluing:spreadingOut}. Let $\epsilon_Y:\Spec k\hookrightarrow X$ be a Zariski point in $X$ with closure $Y$. %The situation is said to satisfy \emph{continuity for the negative part} (resp. \emph{continuity of the positive part}) if for any $a\in D^{\le}(k)$ (resp. $a\in D^{>}(k)$) there is an $\bar a\in D^{\le}(U)$ (resp. $\bar a\in D^{>}(U)$) such that $f^*\bar a = a$ (resp. $f^!\bar a=a$). 
%		
		%The situation is said to satisfy \emph{continuity for $t$-structures} if it satisfies continuity for both positive and negative part. % $\epsilon_Y$: 
		The situation is said to satisfy \emph{continuity for $t$-structures} if for any such $\epsilon_Y$:
		\begin{enumerate}[i.]
			\item (Continuity for the negative part) For any $a\in D^{\le}(k)$, there is a neighborhood $U$ of $\Spec k$ in $Y$, $\Spec k\xrightarrow{f}U\subset Y$ and an object $\bar a\in D^{\le}(U)$ with $f^*\bar a = a$.
			\item (Continuity for the positive part) For any $a\in D^{>}(k)$, there is a neighborhood $U$ of $\Spec k$ in $Y$, $\Spec k\xrightarrow{f}U\subset Y$ and an object $\bar a\in D^{>}(U)$ with $f^!\bar a = a$.
		\end{enumerate}
	\end{definition}
	\begin{remark}
		If $D^{\le }(k)$ resp. $D^{>}(k)$ is finitely generated as $\vvspan S$, then one can verify continuity for positive resp. negative part on objects $a\in S$, since any object in $\vvspan S$ is constructed using extensions, shifts, and retracts from \emph{finitely many} objects in $S$, and by continuity of $D(k)$ the operations can be extended onto the open intersection of the corresponding finitely many open sets.
	\end{remark}
	
	Then the main result of punctual gluing is the following:
	\begin{theorem}[Punctual gluing of $t$-structures]\label{gluing:mainresult}
		Assume the situation in \ref{gluing:continuity} and \ref{gluing:continuityForT} on a Noetherian scheme $X$. Then $(D^{\le}(X), D^{>}(X))$ is a $t$-structure on $D_X$.
	\end{theorem}		
	\begin{proof}
	See \cite[3.7]{vaish2017punctual}.
	\end{proof}
\subsection{Chow motives}\label{sec:chow}
	We refer the reader to Scholl's exposition  \cite{scholl1994classical} for a exposition on the classical category of Chow motives. 
	
	In particular given any field $k$, we have a category of Chow motives denoted $CHM(k)$ (resp. effective Chow motives denoted $CHM^{eff}(k)$), whose objects are given by triples $(X,p,n)$ (resp. tuples $(X,p)$) with $X$ a smooth, projective over $k$, $p$ a correspondence such that $p\circ p = p$ and $n\in \Z$. Here $\circ$ denotes the composition of correspondences. Morphisms in this category are given by appropriate correspondences, composing under $\circ$ (see \cite[1.4]{scholl1994classical} for details).
	
%\begin{para}
		$CHM^{eff}(k)$ resp. $CHM(k)$ are additive, pseudo-abelian categories. There are natural functors:
		\[
			(SmProj/k)^{op} \xrightarrow h CHM^{eff}(k)\hookrightarrow CHM(k)
		\]
		where $h(X)=(X,\Delta_X)$ ($\Delta_X$ being the class of diagonal) while $h(f):=f^*$ is the transpose of the graph of $f:Y\rightarrow X$ in $X\times Y$. Also $CHM^{eff}(k)\hookrightarrow CHM(k)$ is given by $(X,p)\mapsto (X,p,0)$.
\begin{para}\label{chow:duality}
%		There is an associative tensor product on $CHM(k)$ (resp. $CHM^{eff}(k)$) given by 
%		\[
%			(X,p,m)\otimes (Y,q,n) := (X\times Y, p\times q, m+n)
%		\]
%		Any cycle $f\in CH^d(X\times Y)$ gives rise to a cycle $^tf\in CH^d(Y\times X)$. Hence, there is also a contravariant involution (duality) on $CHM(k)$ with 
%		\begin{align*}
%			(X,p, m)&\mapsto {(X,p, m)}^\vee := (X, {^t p}, \dim X-m)\text{ if }X\text{ is irreducible with }\dim X=d.
%		\end{align*}
%		and which acts by transpose on morphisms. Then, for $f:Y\rightarrow X$, we define, 
%			\[ f_*= (f^*)^\vee:h(Y)\rightarrow h(X)\otimes \mathbb L^{\dim X-\dim Y} \]
%		as the transpose of class of $f$. We also define 
%			\[
%				\underline \hom(M, N) := M^\vee \otimes N.
%			\]
%	\end{definition}
%	Then 
%	\begin{proposition}[{see \cite[1.15]{scholl1994classical}}]\label{chow:duality}
	 $CHM(k)$ is a rigid tensor additive category ({see \cite[1.15]{scholl1994classical}}) under tensor $\otimes$ given by 
	 	\[
			(X,p,m)\otimes (Y,q,n) = (X\times Y,p\times q,m+n)
		\]
		and dual of $(X,p,m)$ for $X$ irreducible of dimension $d$, is:
	 	\[
			(X,p,m)^{\vee} = (X,{^{t}p}, d- m)
		\]
	 
	 In particular internal $\hom$, denoted $\underline\hom$, is right adjoint to $\otimes$  and we have,
	 	\[	\hom(M\otimes N, P) = \hom (M, \underline \Hom(N, P)) = \hom (M, N^\vee\otimes P) \]
%	\end{proposition}

%\begin{para}
	The Lefschetz motive $\mathbb L$ is defined as the object $(\Spec k, \Delta_{\Spec k}, -1)$ in $CHM(k)$. It lives in $CHM^{eff}(k)$ because of the decomposition $h(\mathbb P^1) \cong h(\Spec k)\oplus \mathbb L$ ({see \cite[1.13]{scholl1994classical}}). This object becomes invertible in $CHM(k)$ for the tensor product.
%\end{para}
\end{para}

\begin{para}
	Due to Voevodsky \cite{voevodsky}, the category of Chow motives has a full embedding of category of Chow motives in motives over a field. Due to work of Bondarko \cite{bondarko_weights} they can be identified as the full subcategory of objects with Bondarko weight zero. The Chow-K\"unneth decomposition of the Chow motive of a variety, which we discuss in the next section, will be the starting point for our construction of $t$-structures in the triangulated setting. 
\end{para}

	Often, it is easier to construct decompositions of motives over some finite extension of $k$. In such cases following proposition is useful to descend to $k$:
	\begin{proposition}[{Galois descent, see \cite[1.17]{scholl1994classical}}]\label{chow:descent}
		Assume that $X\in SmProj/k$ is purely $d$ dimensional. Let $k'/k$ be a finite Galois extension of degree $m$. Let $X':=X\otimes_k k'\in SmProj/k'$, and let $\beta:X'\times_{k'}X'\rightarrow X\times_k X$ denote the natural projection.
		
		If $p'_1, p'_2, \ldots, p'_r\in CH^d(X'\times_{k'}X')$ are orthogonal idempotents (i.e. $p'_i\circ p'_j =0$ if $i\ne j$ and $p'_i\circ p'_i=p'_i$), which are invariant under action of $Gal(k'/k)$ on $X'\times_{k'}X'=(X\times_k) X\otimes_k k'$, then $p_i=\frac{1}{m}\beta_*p'_i\in CH^d(X\times_k X)$ are also orthogonal idempotents. 
		
		Furthermore if $p'_1, p'_2, \ldots, p'_r$ form a complete system (that is $p'_1 + p'_2 + \cdots + p'_r = \Delta_{X'}\in CH^d(X'\times_{k'}X')$), then  so do $p_1,p_2,\ldots,p_r$ (that is  $p_1 + p_2 + \cdots + p_r = \Delta_{X}\in CH^d(X\times_{k}X)$).	
	\end{proposition}		

Finally, following results provide some technical convenience later:

\begin{notation} Since we will be dealing with different categories of Chow motives with varying fields it will be convenient to use the expanded version $h(s:Y\rightarrow \Spec k)$ instead of the notation $h(Y)\in CHM(k)$ where $Y$ is a variety over $\Spec k$ with structure morphism $s$. Similarly, we will also explicitly spell out the morphisms which are implicit in the notation of a Chow motive -- e.g. $(g:X\rightarrow \Spec k, p\in CH^d(X\times_k X), m)$ is the expanded version of $(X, p, m)$ inside $CHM(k)$. \end{notation}
%	\begin{remark}\label{chow:overBase}
%		The definition of the category $CHM(S)$ does not use that $S=\Spec k$ is spectrum of a field, just that we restrict to smooth proper $X/S$ (with $S$ regular). One still defines $Corr^r(X,Y) = CH^{d+r}(X\times_S Y)$ if $X$ is irreducible of relative dimension $d$ over $S$ and use the same formula for composition. See \cite[\S 1]{deninger1991motivic} for details. 
%	\end{remark}
% In this case we also have:
	\begin{proposition}
		Let $L/k$ be a finite separable field extension, inducing $S\rightarrow T$, with $T=\Spec k$, $S=\Spec L$. Then there is a pullback functor $f^*:CHM(k)\rightarrow CHM(L)$:
		\begin{align*}
			&f^*(g:Y\rightarrow \Spec k, p\in CH^d(Y\times_k Y), n) = (f_S:Y_L\rightarrow \Spec L, (p\otimes_k L)\in CH^d(Y_L \times_L Y_L), n)
		\end{align*}
		where $Y_S:=Y\times_k L$ and the natural projection. We also have a natural pushforward functor $f_*:CHM(L)\rightarrow CHM(k)$:
		\[
			f_*(g:X\rightarrow \Spec L, p\in CH^d(X\times_L X), n) = (fg:X\rightarrow \Spec k, i_*p\in CH^d(X\times_k X), n)
		\]
		where $i:X\times_L X\hookrightarrow X\times_k X$ is the natural immersion.
	\end{proposition}
	\begin{remark*}
		As mentioned before, there is a fully faithful embedding $CHM(k)\hookrightarrow DM(k)$. Then the functors $f^{*}$ resp. $f_{*}$ extend to functors $DM(k)\rightarrow DM(L)$ resp. $DM(L)\rightarrow DM(k)$. This is true more generally due to the formalism of six functors on motivic sheaves (see \ref{sec:motives}), but in the specific case of fields is already due to \cite{voevodsky}. The more by-hand constructions here, which are also well known, give us an explicit handle on these functors which is useful for understanding the Chow-K\"unneth decomposition in the next section. 
	\end{remark*}
	\begin{proof}
		The case for $f^*$ is \cite[1.7]{deninger1991motivic}, we deal with $f_*$. If $X$ and $Y$ are etale proper over $\Spec L$, the immersion $i_{XY}:X\times_L Y\rightarrow X\times_k Y$ is both open and closed. Define $f_*(\alpha) = i_{XY*}(\alpha)$ for a cycle $\alpha\in CH^*(X\times_L Y)$. Let $\beta\in CH^*(Y\times_L Z)$. Consider the diagram:
		\[\begin{tikzcd}
			X\times_L Y\times_L Z\ar[r, hookrightarrow, "u'"]\ar[rd,"t", dotted, hookrightarrow]\ar[d, hookrightarrow, "v'"]	&X\times_k (Y\times_L Z)\ar[r, "s"]\ar[d, hookrightarrow, "v"]	&Y\times_L Z\ar[d, hookrightarrow, "i_{YZ}"]\\
			(X\times_L Y)\times_k Z\ar[r, hookrightarrow, "u"]\ar[d, "r"]	&X\times_k Y\times_k Z\ar[r, "p_{YZ}"]\ar[d, "p_{XY}"]	&Y\times_k Z\\ 
			(X\times_L Y)\ar[r, hookrightarrow, "i_{XY}"]	&X\times_k Y	&\text{with }s\circ u'=q_{YZ}, r\circ v'=q_{XY}
		\end{tikzcd}\]
		where the squares are Cartesian and immersions are both open and closed. Then:
		\begin{align*}
			f_*(\alpha)\circ f_*(\beta) &= p_{XZ*}(p_{XY}^*i_{XY*}\alpha \cdot p_{YZ}^*i_{XY*}\beta) \overset 1= p_{XZ*}(u_*r^*\alpha \cdot v_*s^*\beta) \\
										&\overset 2= p_{XZ*}(t_*(q_{XY}^*\alpha\cdot q_{YZ}^*\beta)) \overset 3= i_{XZ*}q_{XZ*}(q_{XY}^*\alpha\cdot q_{YZ}^*\beta)
										= i_{XZ*}(\alpha\circ \beta) = f_*(\alpha\circ \beta)
		\end{align*}
		(where $q_{XZ}$ resp. $p_{XZ}$ denote the obvious projection over $S$ resp. $T$). Here we use the Gysin formula \cite[6.2]{fulton2013intersection} for the equality $1$, while equality $2$ is obtained by noting that $u', v', u, v$ are all open as well as closed, and hence the intersection takes place in the common component $X\times_L Y\times_L Z$. Equality $3$ holds since $p_{XZ}t=i_{XZ}q_{XZ}$. Other equalities are obvious by definition 
		
		Hence $i_*p$ is a projector above and $f_*$ is a functor as well.
	\end{proof}

\begin{proposition}[Base Change]\label{chow:basechange}
	Let $q:\Spec L\rightarrow \Spec k$ and $p:\Spec K\rightarrow \Spec k$ be finite separable extensions. Then $K\otimes_k L$ is a reduced Artinian ring, \'etale over $k$, and hence can be written as $\oplus_\tau M_\tau$ where $\tau$ varies over a finite indexing set. In particular $M_\tau$ is a separable extension of $k$ with fixed morphisms $q_\tau:\Spec M_\tau \rightarrow \Spec K$ and $p_\tau:\Spec M_\tau \rightarrow\Spec L$.
	Then for any motive $A\in CHM(L)$ we have a natural isomorphism:
	\[
		p^*q_*A \cong \oplus_\tau q_{\tau*}p^*_\tau A
	\]
\end{proposition}
\begin{proof}
For any $X$ smooth over $\Spec L$, we have a diagram:
	\[\begin{tikzcd}[column sep=small]
		X\otimes_L M_\tau\ar[r, equal] &X_\tau \ar[d, "r_\tau"]\ar[r]	&\bigsqcup_\tau X_\tau\ar[d] \ar[r, equal]				&X\otimes_k K \ar[r]\ar[d]													&X\ar[d, "r"]\\	
		&\Spec M_\tau\ar[r]\ar[rrd, "q_\tau" below]	&\bigsqcup_\tau \Spec M_\tau\ar[r, equal]	&\Spec L\otimes_k K\ar[r, "\sqcup_\tau p_\tau"]\ar[d, "\sqcup_\tau q_\tau"]	&\Spec L\ar[d, "q"]\\
		&					&											&\Spec K\ar[r, "p"]															&\Spec k		
	\end{tikzcd}\]
	where the squares are cartesian. Then $X\otimes_k K=\sqcup_\tau X_\tau$, hence 
	\begin{align*}
		p^*&q_*h(r:X\rightarrow \Spec L) 	= p^*h(qr: X\rightarrow \Spec L)  = h(qr\otimes_k K: X\otimes_k K\rightarrow \Spec K) \\
										&=\oplus_\tau h(q_\tau r_\tau:X_\tau \rightarrow \Spec K)
										=\oplus_\tau q_{\tau*}h(r_\tau:X\otimes_L M_\tau \rightarrow \Spec M_\tau)
										=\oplus_\tau q_{\tau*}p_\tau^*h(r:X\rightarrow \Spec L)
	\end{align*}
	as required. The case for a motive $(X, p, n)$ is similar.	
\end{proof}

\subsection{Murre's projectors}\label{sec:murre}	 In \cite{murre1990motive, murre1993conjectural} Murre defines orthogonal projectors $\pi_0(X)$ and $\pi_1(X)$ (the ``trivial'' and the ``Picard'' projector) for the Chow motive of any smooth projective variety $X$ over $k$, and also $\pi_2(X)$ for $X$ a surface. We briefly review the construction of the projectors $\pi_i(X)$ for $i=0, 1, 2$ below for the sake of completeness. We follow the exposition in \cite{scholl1994classical} with minor modifications.
\begin{para}[Construction]\label{chow:murreConstruction}
	The construction can be divided in following steps:
\begin{enumerate} 
	\item First assume that $X$ is geometrically connected.
		\begin{itemize}
			\item $\pi_0(X)$: Let $\zeta\in CH_0(X)$ be a zero cycle of degree $d$. This gives rise to a map $\zeta^*:X\rightarrow \Spec k$, while the structure map $g:X\rightarrow \Spec k$ induces a map $g^*:=h(g):\Spec k\rightarrow X$ in $CHM^{eff}(X)$. It is easy to see that $\zeta^*g^*$ is multiplication by $d$ on $\Spec k$ and hence $\pi_0=\frac{1}{d}g^*\zeta^*$ gives the required projector.
		
			\item $\pi_1(X)$: If $\dim X=0$, define $\pi_1(X)=0$. 
			
			If $\dim X=1$, define $\pi_1(X) = 1 - \pi_0(X) - {^t\pi_0(X)}$. 
			
			Assume $\dim X\ge 2$ and assume that there is a smooth curve $i:C\hookrightarrow X$ obtained by taking successive hyperplane sections ($X$ is assumed to be projective). Let $\zeta$ be the zero cycle obtained by a hyperplane section on $C$. We have natural maps: 
				\begin{equation} \mathbf{Pic}^0(X/k)_{red}\rightarrow \mathbf{Pic}^0(C/k) \cong \mathbf{Alb}^0(C/k) \rightarrow \mathbf{Alb}^0(X/k) \label{eq:isogeny}\end{equation}
			where $\mathbf{Pic}^0(X/k)$ denotes the Picard scheme of $X/k$ (hence $\mathbf{Pic}^0(X/k)_{red}$ denotes the Picard variety of $X/k$) and $\mathbf{Alb}^0(C/k)$ denotes the Albanese.
			The composite $\alpha$ is known to be an isogeny (see \cite{weil1954criteres}, or the discussion in \cite[4.4]{scholl1994classical}), and does not depend on $C$ (once the projective embedding is fixed). 
			%Here $\mathbf P^0(X/k) = \mathbf{Pic}^0(X/k)_{red}$, since $\mathbf{Pic}^0(X/k)$ is proper, $X$ being geometrically normal \cite[Theorem 2.2]{grothendieck1962theoremes}.
			One picks a $\beta:\mathbf{Alb}^0(X/k)\rightarrow \mathbf{Pic}^0(X/k)_{red}$ such that $\alpha\circ\beta = [\times n]$. This gives rise to a cycle $\tilde \beta\in CH^1(X\times X)$ (since we are working with $\Q$ coefficients, see \ref{chow:h1toh1}) which satisfies $\tilde \beta\circ \zeta_* = 0$ and $\zeta^*\circ \tilde \beta = 0$. Then one defines
			\begin{align*}
				\pi_1^? &=\frac{1}{n}\tilde\beta\circ i_*\circ i^* &\pi_1	=\pi_1^?\circ (1-\frac{1}{2}{^t\pi_1^?})
			\end{align*}
			and verifies that that $\pi_1$ is a projector, orthogonal to $\pi_0$.
			
			Since such a curve $C$ certainly exists after taking a finite extension $k'/k$ which we can assume to be Galois, the claim follows using Galois descent.
			
			\item $\pi_2(X)$ for $\dim X\le 2$: Following \cite{murre1990motive} we define
			\[
				\pi_2(X) = \begin{cases}
									0																	&\text{ if }\dim X = 0\\
									^t\pi_0(X)	=\Delta_X - \pi_0(X)-\pi_1(X)							&\text{ if }\dim X = 1\\
									\Delta_X - \pi_0(X) - \pi_1(X) - {^t\pi_0(X)} - {^t\pi_1}(X)		&\text{ if }\dim X = 2
							\end{cases}
			\]
			where $\Delta_X$, the diagonal inside $X\times X$, corresponds to the identity projector. %The equality of two definitions for curves is equivalent to the decomposition
%			\[
%				\Delta_X = \pi_0(X) + \pi_1(X) + {^t\pi_0(X)}\;\;\;\;\;\;\;\;\;\;\;\;\;\text{ if }\dim X=1
%			\]
%			which follows from \cite[Section 3]{scholl1994classical}.
			$\pi_2(X)$ is a projector in the case $\dim X=2$ using mutual orthogonality in \ref{murre:mutualOrthogonality}. %requires some more work, and in this case one shows that $\pi_0(X), \pi_1(X), {^t\pi_0(X)}, {^t\pi_1}(X)$ are mutually orthogonal projectors \cite[Theorem 4.4]{scholl1994classical}. 
			
%%%			Note that by construction, $^t\pi_2(X)=\pi_2(X)$. Therefore
%%%			\[
%%%				\begin{aligned}
%%%					\pi_{>1}(X) &:= \pi_2(X)+{^t\pi_0(X)}+{^t\pi_1(X)} = {^t\pi_{\le 2}(X)}\\
%%%					\pi_{>2}(X) &:= {^t\pi_0(X)}+{^t\pi_1(X)} = {^t\pi_{\le 1}(X)}
%%%				\end{aligned}\;\;\;\;
%%%				\text{ for }\dim X=2.
%%%			\]
		\end{itemize}
	\item Now assume $g:X\rightarrow \Spec k$ is connected, but not geometrically so. Then, since $g$ is smooth, the Stein factorization of $g$, $X\xrightarrow p\Spec L\xrightarrow q\Spec k$ has $p$ smooth and geometrically connected, and $q$ etale. Then there is an immersion $i:X\times_L X\rightarrow X\times_k X$ which is both open and closed. We let $\pi_i(X)$ in $CHM(k)$ to be $i_*\pi_i^L(X)$ where $\pi_i^L(X)\in CH^{\dim X}(X\times_L X)$ is a projector in $CHM(L)$, constructed since $p:X\rightarrow \Spec L$ is geometrically connected.
	\item Finally, if $g:X\rightarrow \Spec k$ is not connected, $X=\sqcup_r X_r$, then in $CHM(k)$, we have $h(X)=\oplus h(X_r)$ and we define $\pi_i(X)=\oplus_r \pi_i(X_r)$.
\end{enumerate}
\end{para}	

\begin{remark}
	In the exposition \cite{scholl1994classical} there appear to be two definitions of $\pi_1(X)$ for curves. Let us define $\pi_1(X)$ for $\dim X=1$ as we do above for the case of $\dim X=2$, and denote $p_1(X) = 1 - \pi_0(X) - {^t}\pi_0(X)$. Then one can show that:
	\begin{align*}
		\pi_1^?(X) = p_1(X)& &\pi_1(X) = \frac{1}{2}p_1(X)
	\end{align*}
	To see this note that $i_*=i^*=\id$, $n=1$, and $\tilde \beta\in CH^1(X\times X)$ can be thought of as the cycle $\Delta_X - X\times {e} - {e}\times X = p_1(X)$ where  $e\in C(k)$ is any $k$-rational point (if $C(k)$ is empty, one uses Galois descent). Therefore, in particular, $\pi_1^?(X) = \tilde \beta = p_1(X)$ and since $\pi_1(X) = p_1(X)(1-\frac{1}{2}p_1(X)) = \frac{1}{2}p_1(X)$.
	
	However, other results of \cite{scholl1994classical} remain valid. In particular, we seem to have no problem in using results \cite[3.9, 4.4, and 4.5]{scholl1994classical} for curves.
\end{remark}

	Note that the construction of $\pi_i(X)$ depends on choices. However we have the following:
\begin{proposition}\label{chow:pushforwards} Let $X$ be smooth projective over a field $K$. Let $i\in \{0,1\}$ or $i=2$ and $\dim X=2$.
	Let $q:\Spec K\rightarrow \Spec k$ be finite, separable. For any of choice $\pi_i(X/K)$, the pushforward $q_*(\pi_i(X/K))$, a projector on $X/k$ in $CHM(k)$, can be identified with a choice of $\pi_i(X/k)$. 

Furthermore, given a choice of $\pi_i(X/k)$, we can find a choice of $\pi_i(X/K)$ such that $\pi_i(X/k)$ is a pushforward $q_*(\pi_i(X/K))$.
\end{proposition}
\begin{proof} 
	If $X\xrightarrow f\Spec K$ is the structure map with $X\xrightarrow g\Spec L\xrightarrow h\Spec K$ the Stein factorization, stein factorization of $X\xrightarrow {pf}\Spec K$ is $X\xrightarrow g\Spec L\xrightarrow {qh}\Spec k$. Now the claims are obvious from definitions.
\end{proof}
	
\begin{proposition}\label{chow:pullbacks} Let $X$ be smooth projective over $k$. Let $i\in \{0,1\}$ or $i=2$ and $\dim X=2$.
	Let $q:\Spec K\rightarrow \Spec k$ be a finite separable extension. For any choice of $\pi_i(X/k)$, the pullback $q^*(\pi_i(X/k)$, a projector on $(X\otimes_k K)/K$ in $CHM(K)$, can be identified with a choice of $\pi_i(X\otimes_k K)$.
\end{proposition}
\begin{proof}
	Only the case of $\pi_1(X)$ is non-obvious. If $X$ is geometrically connected, and we begin with $C\subset X\otimes_k k'$, then given any $K/k$ we can begin with the generic curve $C\otimes_k K\subset X\otimes_k K\otimes_k k'$. Then, by functoriality of $\mathbf{Pic}^0$ and $\mathbf{Alb}^0$ and since isogeny continues to be an isogeny under pullback, the claim follows easily from the constructions.
	
	Even if $X$ is not geometrically connected, we can take $X\xrightarrow{p} \Spec L\xrightarrow{q} \Spec k$ to be the Stein factorization, and then $X\otimes_k K\rightarrow \Spec L\otimes_k K$ also has geometrically connected fibers. Let $\Spec L\otimes_k K= \oplus_\tau \Spec M_\tau$ and then $X\otimes_k K = \oplus_\tau X_\tau$ where $M_\tau$ are finite separable extensions of $L$ and $K$ and there are maps $X_\tau\rightarrow \Spec M_\tau$. Then the projector $\pi_i(X/L)$ lifts to give a projector $\pi_i(X_\tau/M_\tau)$ since $M_\tau/L$ is finite separable, by the computation for geometrically connected $X$. Now the claim is easy to see using \ref{chow:basechange} and \ref{chow:pushforwards}.
\end{proof}
	We have the following components of Chow-K\"unneth decomposition:
	\begin{proposition}\label{murre:mutualOrthogonality}
		Let $X$ be smooth projective over $k$ purely of dimension $d$. Then
		\[
			\pi_0(X), \pi_1(X), \pi_{2d-1}(X):={^t\pi_1(X)}, \pi_{2d}(X):={^t\pi_0(X)}
		\]
		are mutually orthogonal projectors. Let $h^i(X)$ denote the effective motive $(X,\pi_i(X))$ for $i\in\{0,1,2d-1,2d\}$. Then there are natural isomorphisms:
		\begin{align*}
			h^{2d}(X)	&\cong h^0(X)\otimes \L^{\otimes d} & 
			h^{2d-1}(X)	&\cong h^{1}(X)\otimes \L^{\otimes d-1}
		\end{align*}
	\end{proposition}
	\begin{proof}
		This is \cite[1.13 and 4.4]{scholl1994classical}.
	\end{proof}

This allows us to make the following definitions:
\begin{npara}[Chow-K\"unneth decomposition]\label{piconstruction} 
	Let $X$ be pure of dimension $d$. Let $\pi_i(X)$ be as in the proposition above, for $i\in \{0,1,2d-1,2d\}$.
	If $\dim X\le 2$, we also have $\pi_2(X)$ by construction in \ref{chow:murreConstruction}. Define:
		\[
			h^i(X):=(X, \pi_i(X)) \in CHM^{eff}(k)
		\]
		for $i\in{0,1,2d-1,2d}$ and any  $X$ or $i=2$ and $\dim X\le 2$. 
		
	Then for an arbitrary $X$ we define:
	\begin{align*}
			\pi_{\le 0}(X)	&:=\pi_0(X)						&h^{\le 0}(X) &:= h^0(X) \\
			\pi_{\le 1}(X) 	&:= \pi_0(X)+\pi_1(X), 			&h^{\le 1}(X) &:= (X,\pi_{\le 1}(X)) = h^0(X)\oplus h^1(X)  	\\
			\pi_{>0}(X) 	&:= \Delta_X - \pi_{\le 0}(X) 	&h^{>0}(X) &:= (X, \pi_{>0}(X)) \\
			\pi_{>1}(X) 	&:= \Delta_X - \pi_{\le 1}(X) 	&h^{>1}(X) &:= (X, \pi_{>1}(X)) 
	\end{align*}
	where $\Delta_X$, the diagonal inside $X\times X$, is the identity projector on $X$.
	
	For $\dim X\le 2$, we also let:
	\begin{align*}
			\pi_{\le 2}(X)	&:=	\pi_0(X) + \pi_1(X) + \pi_2(X)	&h^{\le 2}(X) &:= (X, \pi_{\le 2}(X))=h^0(X)\oplus h^1(X)\oplus h^2(X)\\
			\pi_{> 2}(X)	&:=	\Delta_X - \pi_{\le 2}(X)		&h^{> 2}(X) &:= (X, \pi_{> 2}(X))
	\end{align*}
	
	Finally, if we need to stress the field, we denote it in the subscript. Thus $h^i_K(X)$ denotes that we are thinking of $X$ as a variety over $K$, that is in $CHM(K)$.
\end{npara}
\begin{remark}While it is not made explicit in the original articles of Murre (or Scholl), it follows from the constructions that for $i=0, 1$, $h^{\le i}$ and $h^{>i}$ can be made into functors from the category of smooth projective varieties to the category of (effective) Chow motives, even though there are choices to be made when writing down the specific projectors $\pi_{\le i}, \pi_{>i}$. For $i=0$ this is almost immediate as it can be shown that for a smooth projective variety $X$, $h^{\le 0}(X) = h(\Spec \Gamma(X,\mathcal O_{X}))$. For $h^{\le 1}(X)$ this requires more work, but will be implied from our constructions later.

Similarly, $h^{\le 2}$ (resp. $h^{>2}$) forms a functor on varieties of $\dim \le 2$ and extends to the full subcategory of effective Chow motives consisting of objects of the form $(X, p, r)$ with $r\le 0$ and $\dim X\le 2$.

We will not use this fact (which will follow from our constructions), but we will use a milder statement, namely $h^i(X\sqcup Y)=h^i(X)\oplus h^i(Y)$ for $i=0,1, 2$  (for a corresponding choice of projectors on the other side) -- this is build into the definitions as we have stated here. 

In the interim, we let $h^{\le i}(X)$ (resp. $h^{> i}(X)$, resp. $h^i(X)$) for $i=0,1$ and $i=2$ denote the object obtained by making \emph{some particular} choice in the constructions.
\end{remark}

	The following is now immediate:
	\begin{corollary}
		We have a full Chow-K\"unneth decomposition in dimension $\le 2$:
		\begin{itemize}
			\item If $P$ is of dimension $0$, then $h(P) = h^0(P)$.
			\item If $C$ is a smooth projective curve, then $h(C)=h^0(C)\oplus h^1(C)\oplus h^2(C)$.
			\item If $S$ is a smooth projective curve, then $h(S)=h^0(S)\oplus h^1(S)\oplus h^2(S)\oplus h^3(S)\oplus h^4(S)$.
		\end{itemize}
	\end{corollary}
	\begin{proof}
		The case of dimension $0$ is trivial, while that of dimension $2$ and $1$ also follow from the way $\pi_2(S)$ and $\pi_1(S)$ is defined.
	\end{proof}
	
	Following results tell us a bit about the morphisms between $h^i(X)$ and $h(X)$ implicit in the Chow-K\"unneth decomposition:
	\begin{proposition}\label{chow:h0ispullback}
		Let $p:X\rightarrow \Spec k$ be a smooth projective variety equidimensional of dimension $d$, and let $X\xrightarrow f \Spec M\rightarrow \Spec k$ be the Stein factorization, with $M$ a finite algebra over $k$. Then $f^*=h(f)$ and $f_*={^th(f)}$ decompose as:
		\begin{align*}
			h(f)	&:h(\Spec M)\overset\cong\rightarrow h^0(X)\hookrightarrow h(X)\hspace{3em} &
			^th(f)	&:h(X) \twoheadrightarrow h^{2d}(X)\overset\cong\rightarrow h(\Spec M)\otimes\L^{\otimes 2d}
		\end{align*}
	\end{proposition}
	\begin{proof}
		This is directly through construction, and definition of $h^0(X)$ as a retract in \ref{chow:murreConstruction}.
	\end{proof}
	
	\begin{proposition}\label{chow:h1ispullback}
		Let $X\rightarrow\Spec k$ and $Y\rightarrow \Spec k$ be geometrically connected, smooth, projective. Let $X$ be of pure dimension $d$, and fix a projective embedding of $X, Y$. Consider:
		\[\begin{tikzcd}
			\hom(h^0(X)\otimes\L, h(Y))\oplus \hom(h^1(X), h^1(Y)) \ar[r, hookrightarrow, "\gamma"]	&	\hom(h^{\ge 2d_X-1}(X), h(Y)\otimes \L^{\otimes d-1}) \ar[d, "\beta"]\\
			\hom(h(X), h^0(Y)\otimes\L^{\otimes d-1})\ar[r,"\alpha"]		&	\hom(h(X), h(Y)\otimes \L^{\otimes d-1})
		\end{tikzcd}\]
		where $\alpha$ is induced by the inclusion $h^0(Y)\hookrightarrow h(Y)$, $\beta$ by the surjection $h(X)\twoheadrightarrow h^{\ge 2d_X-1}(X)$ and $\gamma$, by the isomorphisms $h^{2d-1}(X)\cong h^{1}(X)\otimes\L^{\otimes d-1}$ and $h^{2d}(X)\cong h^0(X)\otimes\L^{\otimes d}$. This gives a natural map:
		\begin{align*}
			\theta: CH^1&(X)\oplus CH^1(Y)\oplus \hom(h^1(X), h^1(Y)) \\
									&\cong \hom(h(X), h^0(Y)\otimes\L^{\otimes d-1})\oplus \hom(h^0(X)\otimes\L, h(Y))\oplus \hom(h^1(X), h^1(Y))\\
									&\xrightarrow{\alpha\oplus\beta\circ\gamma} \hom(h(X), h(Y)\otimes \L^{\otimes d-1})\cong CH^1(X\times Y).
		\end{align*}
		This map is an isomorphism.
	\end{proposition}
\begin{proof}
	By assumption $X$ and $Y$ are geometrically connected, so $h^0(X)=\Spec k$ and $h^0(Y)=\Spec k$. By \ref{chow:h0ispullback} it follows that the composite map $CH^1(X)\rightarrow CH^1(X\times Y)$ (resp. $CH^1(Y)\rightarrow CH^1(X\times Y)$ is obtained by pullback. Now the claim follows from the following lemma and the remark \ref{chow:h1toh1:knowingthemap}.
\end{proof}	

\begin{lemma}\label{chow:h1toh1}
	Assume $X, X'$ be geometrically connected, smooth over $k$ and assume that they are given with a fixed projective embedding. Let $\zeta\in CH_0(X)$ resp. $\zeta'\in CH_0(X')$ be zero-cycles obtained by taking successive generic hyperplane sections. Then:
	\begin{align*}
		\hom(h^1(X),h^1(X')) &\cong \hom(\mathbf{Pic}^0(X/k)_{red}, \mathbf{Pic}^0(X'/k)_{red})\otimes\Q \\
				&\cong \hom (\mathbf{Alb}^0(X/k),\mathbf{Pic}^0(X'/k)_{red})\otimes \Q \\
				&\cong \{c\in CH^1(X\times X')|\zeta'^*\circ c = 0 =c\circ \zeta_*\}\subset CH^1(X\times X')
	\end{align*}
%	where $P_X, P_{X'}$ (resp. $J_X, J_{X'}$) denote the Picard (resp. Albanese). 
	This inclusion induces an isomorphism:
	\begin{align*}
		CH^1(X)\oplus CH^1(X') \oplus \hom(h^1(X), h^1(X')) \cong \ CH^1(X\times X')
	\end{align*}
	where the maps from $CH^1(X)$ (resp. $CH^1(X')$) to $CH^1(X\times X')$ are obtained by pullback along natural pullback along projection from $X\times X'$ to $X$ (resp. $X'$).
\end{lemma}
\begin{proof} See \cite[4.5]{scholl1994classical} for the first statement. The second follows from \cite[3.9]{scholl1994classical}, using that the natural pullback map $CH^1(X)\rightarrow CH^1(X\times X')$ (resp. $CH^1(X')\rightarrow CH^1(X\times X')$) are injective (one can assume that $X, X'$ have $k$ points using Galois descent, whence the pullback has a section).
\end{proof}
\begin{para}\label{chow:h1toh1:knowingthemap}
	To complete the argument in \ref{chow:h1ispullback} we need to know the map
	\[	\hom(h^1(X),h^1(X'))\rightarrow CH^1(X\times X')	\]
	in \ref{chow:h1toh1} is the same. This is \cite[4.4]{scholl1994classical}. The composite:
	\[
		h^1(X)\hookrightarrow h(X) \xrightarrow{i_*\circ i^*} h(X)\otimes \mathbb L^{1-d}\twoheadrightarrow h^{2d-1}(X)\otimes \mathbb L^{1-d}
	\]
	is an isomorphism, where $i:C\hookrightarrow X$ is a generic smooth curve obtained through successive hyperplane sections. It does not depend on the specific choice of sections. Using this isomorphism, we get:
	\[
		\hom(h^1(X),h^1(X')) \cong \hom(h^{2d-1}(X')\otimes \mathbb L^{1-d}, h^1(X))\subset \hom(h(X)\otimes \mathbb L^{1-d}, h(X')) = CH^1(X\times X')
	\]
	Here the subset relationship is well defined because we have fixed $\pi_{2d-1}, \pi_1$ by fixing the projective embedding of $X,X'$. It is under this inclusion the isomorphism in \ref{chow:h1toh1} holds. Now the comparison with the map in \ref{chow:h1ispullback} is obvious.
\end{para}
	The following result allows us to cut down dimensions in our calculations:	

\begin{corollary}\label{h1isInCurve}
	For any $X$ geometrically connected, smooth, over $k$ with a fixed projective embedding. Let $C$ be a smooth curve which is a successive generic hyperplane section. Then $h^1(X)$ is a summand of $h^1(C)$.
\end{corollary}
\begin{proof}
	If one takes $C\hookrightarrow X$, a generic curve obtained by successive hyperplane sections. Then by \cite{weil1954criteres} (see the discussion in \cite[4.3]{scholl1994classical}) the composite:
	\[
		\mathbf{Pic}^0(X/k)_{red}\rightarrow \mathbf{Pic}^0(C/k)_{red} \cong \mathbf{Alb}^0(C/k) \rightarrow \mathbf{Alb}^0(X/k)
	\]
	is an isogeny, hence there is also an isogeny $\mathbf{Alb}^0(X/k)\rightarrow \mathbf{Pic}^0(X/k)_{red}$. Now the claim follows from \ref{chow:h1toh1}.
\end{proof}

\subsection{Geometry}\label{sec:geometry}
In this section we list some results from geometry which are going to be useful later.

Since we are working with rational coefficients, alterations will be an effective substitute for resolution of singularities for us in finite characteristics:
\begin{definition}
	An \emph{alteration} of $X$ is a proper surjective map $\pi:X'\rightarrow X$ which is finite on a non-empty dense open subset of $X$.
\end{definition}

\begin{definition}
	By a \emph{simple normal crossing variety} we mean a variety $Y$ such that if $Y=\cup_{i\in I} Y_i$ into irreducible components, $Y_J:=\cap_{j\in J}Y_j$ is regular for all $\phi\ne J\subset I$. If in addition $Y\subset X$ with $X$ irreducible, and $\dim Y_J = \dim X - \# J$ (whenever $Y_J$ is non-empty), then we say $Y$ is a \emph{simple normal crossing divisor} in $X$.
\end{definition}

\begin{theorem}[deJong's alterations {\cite[Theorem 4.1, Remark 4.2]{jongAlterations}}]\label{geometry:djong:raw} Let $X$ be any irreducible variety over a field $k$ and $Z\subset X$ be a proper closed subset. Then there is an alteration $p:X'\rightarrow X$ together with an open immersion $j:X'\hookrightarrow \bar X'$ such that 
\begin{enumerate}[i)]
	\item $\bar X'$ is regular.
	\item $\bar X'$ is projective over $\Spec k$.
	\item $p^{-1}(Z)\cup (\bar X' - X')$ is a simple normal crossing divisor in $\bar X'$.
	\item There is $k'/k$ finite such that $\bar X'\rightarrow \Spec k$ factors via $\bar X'\rightarrow \Spec k'$ which is smooth.
	\item If in addition $k$ is perfect, $X'$ is smooth over $\Spec k$ and $p$ is generically etale.
\end{enumerate}
\end{theorem}

The following definition allows us to have effective substitutes for generic smoothness in finite characteristics:
\begin{definition}
	A morphism of finite type $f:Y\rightarrow X$ is said to be \emph{essentially smooth} (resp. \emph{essentially etale}) if it factors as $Y\xrightarrow p X'\xrightarrow q X$ with $p$ smooth (resp. etale) and $q$ a finite universal homeomorphism.
\end{definition}

\begin{corollary}\label{geometry:djong}
 Let $X$ be any irreducible variety over a field $k$ and $Z\subset X$ be a proper closed subset. Then there is an alteration $p:X'\rightarrow X$, and an open immersion $j:X'\hookrightarrow \bar X'$ such that
 \begin{enumerate}[i)]
 \item $\bar X'$ is regular.
 \item $\bar X'$ is projective over $\Spec k$.
 \item $p^{-1}(Z)\cup (\bar X' - X')$ is a simple normal crossing divisor in $\bar X'$.
 \item $\bar X'$ is essentially smooth over $k$.
 \item $p$ is generically essentially etale.
 \end{enumerate}
\end{corollary}
\begin{proof}
	First working with $k^{perf}$, the perfect closure, and $X_1 = X\otimes_k k^{perf}$, we can find a generically etale alteration $X'_1\rightarrow X_1$ open inside $\bar X'_1$ with $\bar X'_1$ smooth and projective over $k^{perf}$ by \ref{geometry:djong:raw}. The schemes being of finite type, this descends to $L\subset k^{perf}$ finite over $k$, that is a generically etale alteration $X'\rightarrow X\otimes_k L$ with $X'$ open inside $\bar X'$ which is smooth over $L$, and thus satisfying $i), ii), iii)$. Since $L/k$ is purely inseparable, $\bar X'$ is essentially smooth over $k$ and since  $X\otimes_k L\rightarrow X$ is a finite, universal homeomorphism, $X'$ is a generically essentially etale alteration of $X$.
\end{proof}

Following proposition tells us that essentially smooth schemes over a field are (essentially) regular:
\begin{proposition}%Following holds: \begin{itemize}
%	\item Essential smoothness is preserved under base change. 
%	\item 
Let $f:X\rightarrow \Spec k$ is essentially smooth for $k$ a field. Then $X_{red}$ is regular.
%	\end{itemize}
\end{proposition}
\begin{proof}
	% First statement is true since it is so for smooth morphisms and finite radicial morphisms. 
	Let $X\xrightarrow{p} V\xrightarrow{q} \Spec k$ be the factorization with $p$ smooth, $q$ finite radicial. Hence $V_{red}$ is also a reduced scheme of $\dim 0$ of finite type over $k$, that is it is $\Spec l$ for an Artin ring $l/k$ which must be purely inseparable extension of $k$ since $q$ is a universal homeomorphism. Let $X'=X\times_V V_{red}$. Then $X'/V_{red}$ is smooth since $p$ is, and hence $X'$ is regular since $V_{red}=\Spec l$ for $l$ Artinian. Hence $X'$ is also reduced. Since $V_{red}\hookrightarrow V$ is a closed immersion, so is $X'\rightarrow X$. It follows that $X'=X_{red}$ and is regular as required. 
\end{proof}

Next we focus on the Stein factorization. The standard definition for Stein factorization of a proper morphism $f:Y\rightarrow X$ is the natural factorization of $f$ through the scheme $st(Y):=\mathbf{Spec}(\mathcal O_Y)$ (see {\cite[III.11.5]{hartshorne}}, for example). However for our purposes the following weaker notion is more useful:
	\begin{definition}
		Let $f:Y\rightarrow X$ is a proper morphism. We say that morphisms $f_1:Y\rightarrow st(Y)$, and $f_2:st(Y)\rightarrow X$ are \emph{essentially a Stein factorization} of $f$ if $f=f_2\circ f_1$ such that $f_2$ is finite and $f_1$ has non-empty geometrically connected fibers.
	\end{definition}

	\begin{proposition}[Stein factorizations are essentially preserved under pullback]\label{essential:steinEssentiallyPullsback}
		If $Y\overset{f_1}\longrightarrow Y_1 \overset{f_2}\longrightarrow X$ is essentially a Stein factorization of $f=f_2\circ f_1$ and $g:X'\rightarrow X$ be an arbitrary morphism, then the pullback diagram $Y'\overset{f'_{1}}\longrightarrow Y'_1 \overset{f'_{2}}\longrightarrow X'$ is essentially a Stein factorization of the pullback $f' = f'_{2}\circ f'_{1}$.
	\end{proposition}
	\begin{proof}
		Immediate since both finiteness and the property of having non-empty geometrically connected fibers are preserved under pullbacks.
	\end{proof}
	
While the essential Stein factorization is not unique, it is not too different from the standard one, at least motivically. This follows from the following well known proposition:
	
	\begin{proposition}\label{proposition:essentiallySteinIsStein}
		If a proper morphism $f:Y\rightarrow X$ has Stein factorization $Y\overset{f_1}\longrightarrow st(Y)\overset{f_2}\longrightarrow X$, and an essentially Stein factorization $Y\overset{f'_1}\longrightarrow Z\overset{f'_2}\longrightarrow X$, then there is a finite surjective universal homeomorphism $r:st(Y)\rightarrow Z$ such that $f'_2\circ r =f_2$ and $r\circ f_1 =f'_1$. %In particular Stein factorization is essentially the same as $Z$.
	\end{proposition}
	\begin{proof}
		This is the standard universal property of a Stein factorization. A short proof can be found in \cite[Prop. 1.37, Appendix A]{vvthesis}.
	\end{proof}
	
	 \begin{remark}\label{v2:essential:steinIsGood}
		We will see later that $r^{*}=r^{!}:DM(Z)\leftrightarrows DM(st(Y)):r_*=r_!$ are equivalences (separatedness of $DM(-)$) with $DM(-)$ being the triangulated category of motivic sheaves. Hence it would be possible to confuse between Stein factorization and an essential Stein factorization in the motivic setting.
	 \end{remark}

	 Finally we need the following standard result, which we state without proof:
	 \begin{proposition}[Spreading out]
	 	Let $X\xrightarrow p \Spec L$ be any morphism of finite type, and assume that $\Spec L\xrightarrow s U$ is the generic point for some irreducible scheme $U$. Then:
		\begin{itemize}
			\item We can find a scheme $\bar X$ and a morphism $\bar p:\bar X\rightarrow  U$, such that $\bar X|_{\Spec L} = X$ and $\bar p|_{\Spec L} = p$.
			\item If $p$ is proper (resp. smooth, resp. quasi-finite, finite, resp. radicial, resp. essentially smooth, resp. has geometrically connected fibers, resp. is of fiber dimension $d$) we can find a $V\subset U$ open dense such that same is satisfied by $\bar p|_V$.
		\end{itemize}
	 \end{proposition}
\subsection{Triangulated category of motivic sheaves}\label{sec:motives}
		In this section we review the triangulated category of motivic sheaves with rational coefficients (see \ref{prop:264}). We summarize the limited properties we will use in this setting below:
\begin{para}\label{motives:6functors}\label{motives:comparision}\label{motives:vanishingsForSmooth}
	Associated to any separated scheme $X$ of finite type over any field $k$ one can associate a category $DM(X)$, stable homotopy category of motivic sheaves over $X$, such that following properties hold:
	\begin{enumerate}
		 \item $DM(X)$ is a $\Q$-linear rigid tensor triangulated category which is idempotent complete, and for any morphism $f$ from $Y$ to $X$ we are given adjoint morphisms:
			 	\begin{align*}
					f^*: DM(X) \leftrightarrows DM(Y) : f_* & \text{ for }f\text{ arbitrary}\\
					f_!: DM(Y) \leftrightarrows DM(X) : f^! & \text{ for }f\text{ separated of finite type over }k
				\end{align*}
				with $f^*$ monoidal, and isomorphism of functors
				\begin{align*}
					(f g)_* \iso f_*g_* & & (f g)_! \iso f_!g_! & & (f g)^* \iso g^*f^* & & (f g)^! \iso g^!f^!
				\end{align*}
				whenever both sides make sense. We let $1_X$ denote the unit for the monoidal structure in $DM(X)$. In each $DM(X)$ there is an invertible (for the monoidal structure) object $1_X(1)$ (Tate motive) which is preserved under the four functors. We let $M(r):=M\otimes 1_X(1)^{\otimes r}$ for any $M$ in $DM(X)$ and any $r\in \Z$.
			\item We have a natural transformation:
				\begin{align*}
					f_!	\rightarrow f_*
				\end{align*}
				which is an isomorphism for $f$ proper.
			\item (Localization) If $i:Z\hookrightarrow X$ is a closed immersion and $j:U\hookrightarrow X$ denote the open complement, then the functors $i^*,i_*,i^!,j_!,j^*,j_*$ satisfy the formalism of gluing \ref{gluing:openClosed}.
			\item (Base change) Any cartesian square (for morphisms of finite type) leads to natural isomorphisms as follows:
				\[\begin{array}{lcl}
					\begin{tikzcd}
						Y'\ar[r, "f'"]\ar[d, "g'"]	&X'\ar[d, "g"]		\\
						Y\ar[r, "f"]				&X	
					\end{tikzcd}
							& \Rightarrow 
							 \begin{array}{l}
								f^*g_! \overset\iso\longrightarrow g'_!{f'}^{*} \\
								g'_*{f'}^{!} \overset\iso\longrightarrow f^!g_{*}
							  \end{array}
				\end{array}\]
				For $g$ proper, identifying $g_!$ with $g_*$ yields so called \emph{proper base change}, and in this case the isomorphism holds without finite type assumption. For $f$ smooth, purity below allows us to replace $f^!$ with $f^*$ and yields so called \emph{smooth base change}.
			\item (Purity)
					If $f:Y\rightarrow X$ is quasi-projective, smooth, of relative dimension $d$ then for any object $M\in DM(X)$, we have functorial isomorphisms:
				\[
					f^!(M)(-d)[-2d]\iso f^*(M) \text{ where } d=\dim Y-\dim X
				\]
				 (Absolute purity) If $i:Z\hookrightarrow X$ denote a closed immersion of everywhere co-dimension $c$ with $Z,X$ regular then we have an isomorphism:
				\[
					i^! 1_X \iso  1_Z(-c) [-2c]
				\]
				where $1_Y$ is the unit of the monoidal structure in $DM(Y)$.
%			\item (Projective Bundle Formula) (\todo state what you need) If $p:P\rightarrow S$ is a projective bundle of rank $n$, then the first Chern class of the canonical line bundle over $p$ leads to a natural decomposition:
%				\[
%					p_* 1_P \iso \oplus_{0\le i\le n}  1_S(-i)[-2i]
%				\]
%%%			\item (Blow up triangle) Suppose $p:X'\rightarrow X$ is a birational morphism. Assume that we have a Cartesian diagram:
%%%			\[\xymatrix{
%%%				Z'\ar[r]\ar[d]^q	&X'\ar[d]^p		\\
%%%				Z\ar[r]^i			&X		
%%%			}\]
%%%			such that $i$ is a closed immersion, $U=X-Z$ is smooth (FIXME) and on $p|_U:U':=p^{-1}(U)\rightarrow U$ is is an isomorphism. Then
%%%			\[
%%%				 1_X	\longrightarrow i_* 1_Z\oplus p_* 1_{X'}	\longrightarrow  i_*q_* 1_{Z'}	 \longrightarrow 
%%%			\]
%%%			is a distinguished triangle in $DM(X)$.
% 		\item (Galois descent)
% 			Let $p:Y\rightarrow X$ be an \'etale morphisms of smooth schemes which is Galois, and Galois group is $G$. (i.e. $ Y/G\iso X$ under the natural map). Then we have a natural isomorphism for all $A\in DM(X)$:
% 			\[
% 				A\overset\iso\rightarrow (p_*p^*A)^G
% 			\]
%		\item (Projection formula) If $f:Y\rightarrow X$ be any morphism and $A\in DM(X), B\in DM(Y)$ then there is a natural isomorphism:
%		\[
%			f_!(f^*A \otimes B) \iso A\otimes f_!B
%		\]
%		\item If $f:Y \rightarrow X$ and $f':Y'\rightarrow X'$ are any morphisms, then if $A\in DM(X)$, $B\in DM(Y)$, we have that $f^!A\boxtimes g^!B \iso (f\times B)^!(A\boxtimes B)$. Here $A\boxtimes B$ denotes the object $q^*A\otimes {q'}^*B$ where $q$ and $q'$ denote the two projections.
		\item (Separated-ness) If $f:Y \rightarrow X$ is a finite surjective universal homeomorphism then $f^*$ induces an equivalence $DM(X)\rightarrow DM(Y)$ and hence $(f^*\iso f^!, f_*\iso f_!)$ form inverses up to equivalence.
%%%		\item If $p:X\rightarrow Y$ is proper \'etale then $ 1_Y$ is a summand of $p_* 1_X$.
		\item (Embedding of Chow motives) Let $k$ be perfect. There are natural functors
		\begin{align*}
			(SmProj/k)^{op}\overset h \rightarrow CHM^{eff}(k) \hookrightarrow CHM(k)\overset M\hookrightarrow	DM(\Spec k)
		\end{align*}
		where $M$ is a fully faithful embedding preserving tensor products.
		
		By abuse of notation we will denote the extended functor from $(SmProj/k)^{op}$ (resp. $CHM^{eff}(k)$, resp. $CHM(k)$) to $DM(k)$, all as $M$. Then we have also have a functorial identification:
		\[
			M(X) = p_*1_X \in DM(k)\text{ for all }(p:X\rightarrow k)\in SmProj/k.
		\]
		and in particular $M(\mathbb L)=1_k(-1)[-2]$. For any map $f:Y\rightarrow X$ between irreducible schemes in $SmProj(k)$ with $d=\dim Y-\dim X$. Then we have 
		\begin{align*}
			M(f^*)&: p_*1_X\longrightarrow p_*f_*f^*1_X \cong p_*f_*1_Y \\
			M(f_*)&: p_*f_*1_Y \cong p_*f_*f^*1_X \cong p_*f_!f^!1_X(-d)[-2d] \longrightarrow p_*1_X(-d)[-2d] 
		\end{align*}
		where the maps on the right hand side are coming from the natural maps of adjunctions and using that $f$ is proper, and purity.
		\item (Comparison with (motivic) cohomology)  For any $\pi:X\rightarrow \Spec k$ with $X$ regular and connected and $k$ perfect. Then we have:
		\[
			\Hom_{DM(\Spec k)}( 1_k,\pi_* 1_X(q)[p]) 	=  H^{p,q}(X) = \begin{cases}
																						0 & p>2q\text{ or }p>q+\dim X\\
																						0 & q\le1\text{ and }(p,q)\notin\{(0,0),(1,1),(2,1)\}\\
																						\mathbb{Z}(X)\otimes\Q & (p,q)=(0,0)\\
																						\mathcal{O}^{*}(X)\otimes\Q & (p,q)=(1,1)\\
																						Pic(X)\otimes\Q & (p,q)=(2,1)
																		\end{cases}%\left\{ 
%									\begin{array}{l l}
%											0 	&\quad \text{if }r\ne 0\text{ or }c>0 \\
%											\Q	&\quad \text{if }r=0\text{ and }c=0
%									\end{array} \right.
		\]
		where $H^{p,q}(X)$ denotes the motivic cohomology of $X$ with rational coefficients.
		\item (Continuity) Assume $\{X_i\}_{i\in I}$ is a pro-object in category of schemes where $I$ is a small co-filtering category, such that the transition maps $g_{j\rightarrow i}:X_j\rightarrow X_i$ are affine for all maps $j\rightarrow i$ in $I$. Let $X=\lim _{i\in I}X_i$ in the category of schemes (which exists because of the assumption on $g_{j\rightarrow i}$) and let $f_i:X\rightarrow X_i$ denote the natural map. Let $i_0$ be any object in $I$, and let $I/i_0$ be the category over $i_0$: 
		\begin{align*}
			 \forall A\in DM(X) &\;\;\;\;\; \exists j\in I, j\in A' \text{ s.t. }A\iso f_j^*(A')& &\\
			\forall A, B\in DM(X_{i_0}) &\;\;\;\; \lim_{\underset{j\in I/i_0}\longleftarrow}\Hom_{DM(X_i)}(g^*_{j\rightarrow i_0}A,g^*_{j\rightarrow i_0}B) = \Hom_{DM(X)}(f_{i_0}^*A,f_{i_0}^*B)
		\end{align*}
		\end{enumerate}
\end{para}
\begin{para}\label{motives:DAandDM}
	One choice for such a construction is the stable category of motivic sheaves without transfers as constructed by Ayoub in \cite{ayoub_thesis_1, ayoub_thesis_2}. This is the category $\mathbb{SH}_\mathfrak{M}^T(X)$ of \cite[4.5.21]{ayoub_thesis_2} with $\mathfrak{M}$ being the complex of $\Q$-vector spaces (and one works with the topology \'etale topology), also denoted as $DA(X)$ in the discussion \cite[2.1]{ayoub2012relative}. The second choice for the construction is the compact objects in the category of motivic sheaves with transfers, the Beilinson motives $DM_{B,c}(X)$ as described in the article \cite{cisinski2012triangulated}. Note that the two constructions are known to be equivalent \cite[16.2.18]{cisinski2012triangulated}.
\end{para}
\begin{remark}\label{remark:qp}
		Constructions of \cite{ayoub_thesis_1, ayoub_thesis_2} require that all schemes are assumed to be quasi-projective (see \cite[\S 1.3.5]{ayoub_thesis_1}) over the base. Since such objects are stable under immersions, base change, etc., and also because Jong's alterations can be chosen to be quasi-projective \cite{jongAlterations}, this does not cause any problem in the proofs with this additional restriction.
\end{remark}
Our assumptions are then valid in both these situations: 
\begin{proposition} \label{prop:264}
	The assumption \ref{motives:6functors} holds if we work either with $DA(X)$ (and restrict to quasi-projective $X$) or $DM_{B,c}(X)$ of \ref{motives:DAandDM}.
\end{proposition}
\begin{proof}
	For $DA(X)$ \cite[4.5]{ayoub_thesis_2} shows that $DA(-)$ gives rise to an homotopy-stable algebraic derivator in the sense of \cite[2.4.13]{ayoub_thesis_1}. The axiom DerAlg 5 allows one to invoke the results of \cite[1.4]{ayoub_thesis_1}. In particular, \cite[1.4.2]{ayoub_thesis_1}  gives the existence of the four functors $f^*, f^!, f_*, f_!$ satisfying $(1), (2), (4)$. Purity $(5)$ follows from \cite[1.6.19]{ayoub_thesis_1} (also see \cite[2.14]{ayoubrealisation}) , while absolute purity is shown in \cite[Cor. 7.5]{ayoubrealisation}. Localization $(3)$ is \cite[4.5.47]{ayoub_thesis_2}. For separatedness $(6)$ by adjointness it is enough to show that $f^*$ is an equivalence which follows from \cite[3.9]{ayoubrealisation}. Property $(9)$ is shown in \cite[Prop. 2.30]{ayoubrealisation}. 
	Property $(7), (8)$ follow from the fact that there are equivalences $a^{tr}:DA(k)\cong DM(k):o^{tr}$, for $k$ perfect where the right hand side is the corresponding construction with transfers due to \cite{voevodsky} (see \cite[2.4]{ayoub2012relative}) for which we know embedding of Chow motives $(7)$ by \cite[2.1.4]{voevodsky} (after composing with the duality on $CHM^{eff}(k)$, $DM(k)$) and the computations of motivic cohomology $(8)$ by \cite[3.5, 4.2, 19.3]{weibel_mazza_notes} (or see \cite[2.2.6]{vaish2016motivic} for details).

	For the Beilinson motives $DM_{B,c}(X)$ all the properties of \ref{motives:6functors} except $(6), (7)$ and $(8)$ are already discussed in the introduction of \cite{cisinski2012triangulated}, section $C$ -- properties $(1)$ to $(5)$ and $(9)$ are part of Grothendieck's six functor formalism as defined in \cite[A.5.1]  {cisinski2012triangulated} (continuity is an extension of usual six functor formalism as mentioned in their list) and hold for $DM_{B,c}(X)$ by \cite[C.Theorem 3]{cisinski2012triangulated}. Separatedness $(6)$ can be deduced from \cite[Prop. 2.1.9]{cisinski2012triangulated}. Property $(7)$ follows from the fact that over a perfect field their construction reduces to Voevodsky's original construction \cite{voevodsky} (see \cite[11.1.14]{cisinski2012triangulated}), where this is known by \cite[2.1.4]{voevodsky} (after composing with the duality on $CHM^{eff}(k)$, $DM(k)$). Alternatively, Chow motives can be identified as the heart of Bondarko's weight structure due to \cite{fangzhou2016borel} giving us $(7)$. Property $(8)$ can be deduced from the corresponding computation in Voevodsky's category, see \cite[3.5, 4.2, 19.3]{weibel_mazza_notes} (this is also present in \cite[11.2.3]{cisinski2012triangulated}).
\end{proof}
\section{Morel's weight truncations in the motivic setting}\label{sec:morel}
\subsection{Cohomological Motives}\label{sec:cohMotives}
The analogue of Morel's constructions will be constructed on cohomological motives -- these are the cohomological analogues of (that is, dual of) effective motives. The key properties of the category of cohomological motives are 
\begin{itemize}
	\item It is (finitely) generated by motives of the form $\pi_*1_X$ for $\pi$ (essentially) smooth, projective.
	\item It includes all motives of the form $\pi_*1_X$ for $\pi$ arbitrary.
	\item It is stable under $f^*, f_*, f^!, f_!$ for $f$ immersions, and $f^*$ for $f$ arbitrary, hence satisfying the formalism of gluing and extended gluing.
\end{itemize}
These properties are standard and have been well covered in literature, see e.g. \cite[\S 3.1]{ayoub2012relative} or \cite[\S 3]{vaish2016motivic}. The reader who is only interested in the construction of the analogue of weight truncations $w_{\le \id}$ and $w_{\le \id+1}$ and is otherwise familiar with the cohomological motives may want to skip the section completely. 

However for construction of the intersection complex of a three-fold, we need to work with the category (finitely) generated by $\pi_*1_X$ with $\pi$ (essentially) smooth projective but restrictions on dimension of $X$. The proofs are entirely analogous to the proofs for the category for full cohomological motives, and just involves keeping track of dimensions of various schemes that occur in the construction. This is what we do below.

We throughout assume the situation of \ref{motives:6functors}. We begin with the following definitions:
\begin{definition}\label{definieCohMotives}
	Let $Y$ be a Noetherian scheme of finite type over $k$.  We define the following collections of objects of $DM(Y)$ (also identified with the corresponding full subcategories):
	\begin{align*}
		S^{dom,coh}_d(X) &= \{ p_*1_X(-r) \big| p:Y\rightarrow X\text{ projective}, Y\text{ connected regular},\hspace{-40em}\\
						&\;\;\;\;\;\;\;\;\;\;\;\;\;\;\;\;\;\;\;\;\;\; Y\text{ essentially smooth over }k, \dim Y\le d, r\ge 0, p\text{ dominant}\}\hspace{-40em}\\		
		S^{sm,coh}_d(X) &= \{ p_*1_X(-r) \big| p:Y\rightarrow X\text{ projective}, Y\text{ regular},\hspace{-40em}\\
						&\;\;\;\;\;\;\;\;\;\;\;\;\;\;\;\;\;\;\;\;\;\; Y\text{ essentially smooth over }k, \dim Y\le d, r\ge 0\}\hspace{-40em}\\
		S^{coh}_d(X)	&= \{ p_*1_X(-r) \big| p:Y\rightarrow X\text{ arbitrary}, \dim Y\le d, r\ge 0\}\hspace{-40em}\\
		S^{sm, coh}(X)	&= \{ p_*1_X \big| p:Y\rightarrow X\text{ projective}, Y\text{ regular, essentially smooth over }k\}\hspace{-40em}\\
		DM_{d,dom}^{coh}(X)	&= \vvspan{S^{dom, coh}_d(X)\cup S^{sm, coh}_{d-1}(X)}	&DM_d^{coh}(X)	&= \vvspan{S^{sm, coh}_d(X)}		\\				
		DM^{coh}(X)&=\vvspan{S^{sm, coh}(X)}
	\end{align*}
	We have inclusions:
	\begin{align*}
		S^{dom, coh}_d(X)\subset S^{sm, coh}_d(X)\subset S_d^{coh}(X)& &S^{sm, coh}(X)=\bigcup_d S^{sm, coh}_d(X)
	\end{align*}
	
	$DM^{coh}(X)$ is called as the \emph{category of cohomological motives} over $X$. 
	
	$DM_d^{coh}(X)$ and $DM^{coh}_{d,dom}(X)$ will allow us to control dimensions of varieties that occur -- while $DM_d^{coh}(X)$ restricts to schemes with dimensions $\le d$ (but allowing arbitrary, negative Tate twists), $DM_{d+1,dom}^{coh}(X)$ will be essentially the same, except that it will allow dominant varieties of one higher dimension. These categories satisfy reasonable stability properties under the four functors, and will be useful for constructing motivic intersection complexes of threefolds.
\end{definition}

We can replace projectivity in the definition of $S^{sm, coh}_d(X)$ by properness, and essential smoothness is inessential, the following lemma is \cite[3.1]{hebert2011structure}, for convenience we include a proof:
\begin{lemma}[Motivic Chow's lemma]\label{motivicChowLemma}
	Let $p:Y\rightarrow X$ be proper, with $Y$ regular. Then there is a map $q:Y'\rightarrow Y$ with $p\circ q:Y'\rightarrow X$ projective, $\dim Y'=\dim Y$, $Y'$ regular, essentially smooth over $k$ such that $p_*1_Y$ is a retract of $(p\circ q)_*1_{Y'}$.
\end{lemma}
\begin{proof}
Complete $X\hookrightarrow \bar X$ which is proper over $k$, and extend $p$ to $\bar p:\bar Y\rightarrow \bar X$ proper.
{\flushleft
\begin{minipage}{0.8\textwidth}
	Let $q_1:Y_1\rightarrow \bar Y$ be an alteration as in \ref{geometry:djong}. Then $Y_1$ is projective over $k$. Since $\bar X/k$ is separated, $\bar p\circ q_1:Y_1\rightarrow \bar X$ is projective. Restricting to $X\subset \bar X$ and defining $Y'=q_1^{-1}(Y)=(p\circ q_1)^{-1}(X)$, we see that $q=q_1|_Y:Y'\rightarrow Y$ is an alteration, $p\circ q = (p\circ q_1)|_X$ is projective and $Y'\subset Y_1$ is open hence essentially smooth over $k$ and regular. Furthermore $q$ is generically essentially etale since $q_1$ is. Also $q:Y'\rightarrow Y$ is projective since $p\circ q$ is and $p$ is separated. 
\end{minipage}%
\begin{minipage}{0.2\textwidth}
		\hfill {\begin{tikzcd}
			Y'\ar[d, "q"]\ar[r, hookrightarrow]	&Y_1\ar[d, "q_1"]\\
			Y\ar[d, "p"]\ar[r, hookrightarrow]	&\bar Y\ar[d, "\bar p"]	\\
			X\ar[r, hookrightarrow]				&\bar X
		\end{tikzcd}}
\end{minipage}%
}

	Hence $q$ factors as $q:Y'\overset{i}\hookrightarrow Y\times \mathbb P^n\xrightarrow s Y$ with $i$ closed immersion and $s$ smooth. Therefore $q^! 1_Y = i^!s^! 1_Y = i^!1_{\mathbb P^n\times Y}(c)[2c] = 1_{Y'}$ using purity and absolute purity (as $Y, Y'$ are regular).
	
	Hence there are natural maps:
	\[
		1_Y \xrightarrow{\alpha} q_*q^*1_Y = q_*1_{Y'} = q_*q^!1_{Y} \xrightarrow{\beta} 1_Y
	\]
	We can replace $X$ by it's irreducible components, and hence can assume $X, Y, Y'$ are connected. It would be enough to show that $\beta\circ\alpha = n\cdot\id$ for some $n\ne 0$. 
	
	Since $q$ is generically, essentially etale, there is a $j:U\hookrightarrow Y$ open dense, such that $q|_U$ is essentially etale. Notice that 
	$\hom(1_Y, 1_Y)\cong \hom(1_k, p_*1_Y) = CH^0(Y)$ since $Y$ is regular (we have already assumed this for $k$ perfect in \ref{motives:6functors}, more generally this computation for perfect fields implies the same for all fields, see \cite[2.2.6]{vaish2016motivic}), and we have a natural restriction map:
	\[\begin{tikzcd}[row sep=small]
		\- 		&CH^0(Y)\ar[r] \ar[d,equal]	\ar[l, phantom, "{j^*= [Y]\mapsto[Y\cap U]:\hspace{5em}}"]				& CH^0(U)\ar[d, equal]\\
				&\hom(1_Y, 1_Y)\ar[r, "j^*"] 					& \hom(1_U, 1_U)
		\end{tikzcd}\]
	obtained by restriction of cycles, and hence is an isomorphism, since $U$ is open dense. Therefore it is enough to show that $j^*\beta\circ j^*\alpha = n\cdot \id$ for some $n\ne 0$, and by smooth base change, this is the natural map:
	\[
		1_U\rightarrow q_*q^*1_U = q_!q^!1_U \rightarrow 1_U
	\]
	which is $n\cdot \id$ for some $n\ne 0$. By separatedness, it is enough to prove this for $q$ etale.
	By continuity, shrinking $U$ further if needed, it is enough to show this when $U$ is replaced by it's generic point $\eta=\Spec K$, and $Y_\eta=\Spec L$ with $q:\Spec L\rightarrow K$ a finite separable map, and restricting to each connected component, we can assume that $L/K$ is separable field extension.
	
	By Galois descent, we can replace $q$ by $\bar q: L\otimes_K \bar L\rightarrow \bar L$ where $\bar L$ is normal closure of $L$. But $L\otimes_K \bar L\cong \oplus_\tau \bar L_\tau$ with $\tau$ varying over embeddings $L\hookrightarrow \bar L$ over $K$ and $\bar q|_{\bar L_\tau}:\bar L_\tau \cong \bar L$ are Galois isomorphisms. The claim is now obvious.
\end{proof}

It is also possible to do away with assumptions of regularity and properness in $S^{sm,coh}_d(X)$. We will first need the following:
\begin{lemma}\label{sncIsCoh}
		Let $X$ be regular irreducible of $\dim X\le d$ and $i:Z\hookrightarrow X$ be a closed immersion, $Z$ not equal to $X$, such that $Z$ is a variety with normal crossing (i.e. $Z$ is reduced, irreducible components of $Z$ are regular, and for all positive integers $r$, $r$-fold intersections of the irreducible components are also regular). Let $j:U\rightarrow X$ denote the open complement.
		Then $i^! 1_X\in DM_{d-1}^{coh}(Z)$ and $j_* 1_U\in DM_{d, dom}^{coh}(X)\subset DM_{d, dom}^{coh}(X)$.
	\end{lemma}
	\begin{proof}
		We induct on $k_Z$ which is defined to be the maximum of integers $k$ such that some $k$ distinct irreducible components of $Z$ have a non-empty intersection. We can assume that $X$ is irreducible and that $Z$ is a proper subset.		

		(\emph{Base case}) If $k_Z=1$, no two irreducible components of $Z$ intersect and hence $Z$ is regular. Therefore {absolute purity} implies that $i^! 1_X \iso  1_Z[-c](-2c)$ where $Z$ is of codimension $c$. Since $Z$ is regular and $\dim Z< \dim X\le d$ and identity map is proper while $DM^{coh}_{d-1}(Z)$ is closed under shifts and negative Tate twists, this lies in $DM_{d-1}^{coh}(Z)$ as required.

		Further, since $i$ is proper, clearly $i_* DM_{d-1}^{coh}(Z) \subset DM_{d-1}^{coh}(X)\subset DM_{d,dom}^{coh}(X)$. Hence the localization triangle
		\[
			i_*i^! 1_X	\rightarrow  1_X	\rightarrow j_* 1_U \rightarrow 
		\]
		gives that the last term $j_* 1_U \in DM_{d,dom}^{coh}(X)$ since the middle term is.
		
		(\emph{Induction step}) Note that $j_* 1_U\in DM^{coh}_{d,dom}(X)$ if $i^! 1_X \in DM^{coh}_{d-1}(Z)$ using the localization triangle just as above. Hence it is enough to prove the later.
		
		Assume $k_Z\ge 2$. Let $i_1:Z_1\hookrightarrow Z$ be the union of $2-fold$ intersections of irreducible components of $Z$ and $j:U=Z-Z_1\hookrightarrow Z$ denote the open complement.Then $Z_1$ is a variety with normal crossing, with $k_{Z_1}< k_Z$. The localization triangle for $i^!\mathbf 1_X$ on $Z$ gives:
		\[
			i_{1*}i_{1}^!i^{!} 1_X \rightarrow i^! 1_X	\rightarrow j_*j^!i^! 1_X	\rightarrow 
		\]
		$i_{1}^!i^{!} 1_X = (i\circ i_1)^! 1_X \in DM_{d-1}^{coh}(Z_1)$ by induction hypothesis, and hence first term is in $DM_{d-1}^{coh}(Z)$. Hence it is enough to show that last term is.
		
		{\flushleft
		\begin{minipage}{0.6\textwidth}	
		 Now let $j':W=X-Z_1\hookrightarrow X$ be the open immersion. Thus $W\cap Z = U$. Let $i':U\hookrightarrow W$ denote the closed immersion. Then $j^!i^!  1_X= i^{'!}j^{'!} 1_X = i^{'!} 1_W$ (since $j^{'!}=j^{'*}$). Hence, it is enough to show that $j_*i^{'!} 1_W \in DM_{d-1}^{coh}(Z)$. Now since $U,W$ are regular, this is same as $j_* 1_U(-d)[-2d]$ using {absolute purity}. Hence it is enough to show that $j_* 1_U \in DM^{coh}_{d-1}(Z)$, since the latter is stable by negative Tate twists.
		 \end{minipage}\begin{minipage}{0.4\textwidth}\hfill
		 	\begin{tikzcd}[column sep = small]
									&U_i\ar[d, "r_i", hookrightarrow]\ar[rr, hookrightarrow, "\circ" description, "s_i"]	& &Z'_i\ar[d, hookrightarrow, "t_i"]\\
				Z-Z_1\ar[r, equal]	&U\ar[rr, hookrightarrow, "\circ" description, "j"]\ar[d, hookrightarrow, "i'"]			& &Z\ar[d, hookrightarrow, "i"]\ar[rd, hookleftarrow, "i_1"]\\
				X-Z_1\ar[r,equal]	&W\ar[rr, hookrightarrow, "\circ" description, "j'"]									& &X\ar[r, hookleftarrow]				&Z_1\\
			\end{tikzcd}
		\end{minipage}}

		 Now if $Z=\cup_i Z'_i$ in irreducible components, $U=\sqcup_i U_i$, a disjoint union. Let $U_i\overset{s_i}\hookrightarrow Z'_i\overset{t_i}\hookrightarrow Z$ and $r_i:U_i\hookrightarrow U$ be the natural immersions, then $1_U=\oplus r_{i*}1_{U_i}$ and $jr_i=t_i s_i$ hence $j_*1_U = \oplus_i t_{i*}s_{i*}1_{U_i}$. Since $t_i$ is proper, it is enough to show that $s_{i*}1_{U_i}\in DM_{d-1}^{coh}(Z'_i)$. But $Z'_i-U_i = \cup_{j\ne i} Z'_i\cap Z'_j\subset Z_1$ and $k_{Z'_i-U_i}\le k_{Z_1}<k_Z$, hence the induction hypothesis allows us to draw the required conclusion.
	\end{proof}

	Now it is possible to work with $f$ arbitrary for definition of $S^{coh}(X)$:
	\begin{lemma}\label{arbitraryY}
		For any morphism $f:Y\rightarrow X$, we have that $f_* 1_Y(-c)\in DM_d^{coh}(X)$, where $d=\dim Y$ and $c\ge 0$. If $Y$ is irreducible and $f$ dominant, $f_* 1_Y(-c)\in DM_{d,dom}^{coh}(X)$
	\end{lemma}
	\begin{proof}
		We can replace $X$ by scheme theoretic image $\overline{f(Y)}$, since $DM^{coh}_d$ is preserved under pushforwards for proper maps and hence we can assume $\dim X\le \dim Y=d$ and $f$ dominant. Since $DM^{coh}_d(X)$ is stable under twists by nonpositive integers, we can assume $c=0$. We induct on the dimension and number of irreducible components of $Y$. 
		
		(\emph{Base Case}): If $\dim Y=0$, $Y$ is Artinian. But then $Y_{red}$ is regular and separatedness implies that $f_* 1_Y = f_{red*} 1_{Y_{red}}$ where $f_{red}$ is the composite $f_{red}:Y_{red}\rightarrow X$.
		
		 (\emph{Induction Step}): Compactify $f$ to $\bar f:\bar Y\rightarrow X$. Let $W=\bar Y-Y$. 
		 By \ref{geometry:djong} there is an alteration $\bar p:\bar Y'\rightarrow \bar Y$ (that is $p$ is proper, generically finite, $\bar Y'$ smooth) such that $\bar p^{-1}(W)$ is simple normal crossing. Let $Y'=\bar p^{-1}Y$ and let $p=\bar p|_{Y'}$. Since $p$ is generically essentially etale, there  is a $j:U\hookrightarrow Y$ which is regular open dense such that with $U'=p^{-1}U$, $1_U$ is a summand of $p|_{U'}(1_{U'})$. Let us denote the inclusion $U'\hookrightarrow Y'$ by $j'$. Let $i:Z=Y-U\hookrightarrow Y$ and $i':Z'=Y'-U'\hookrightarrow Y'$ with reduced induced structure.
			\[\begin{tikzcd}
					U'\ar[r, "\circ" description, "j'", hookrightarrow]\ar[d,"p|_{U'}"]	&Y'\ar[d,"p"]\ar[r, "\circ" description, "j''", hookrightarrow]	&\bar Y'\ar[d, "\bar p"]\ar[r, hookleftarrow]	&W'\ar[d]	&U'\ar[r, "\circ" description, "j'", hookrightarrow]\ar[d,"p|_{U'}"]	&Y'\ar[d,"p"]\ar[r, hookleftarrow, "i'"] &Z'\ar[d]\\
					U\ar[r, "\circ" description, "j", hookrightarrow]	&Y\ar[r, "\circ" description, hookrightarrow]\ar[rd, "f" below]	&\bar Y\ar[r, hookleftarrow]\ar[d, "\bar f"]	&W &U\ar[r, "\circ" description, "j", hookrightarrow]	&Y\ar[r, hookleftarrow, "i"]	&Z\\
					&	&X
			\end{tikzcd}\]
		 Now we have a triangle:
		 \[
		 	f_* j_! 1_U\rightarrow  f_* 1_Y\rightarrow f_*i_*  1_Z \rightarrow
		 \]
		 Last term is in $DM^{coh}_{d-1}(X)$ (hence in $DM^{coh}_{d, dom}(X)$) by induction since $f$ is dominant and hence $\dim Z<\dim Y=d$, and hence it is enough to show that the first term is in $DM^{coh}_{d, dom}(X)$.
		 
		 Since  $1_U$ is a summand of $(p|_{U'})_* 1_{U'}$ it is enough to show that $f_*j_!(p|_{U'})_* 1_{U'}\in DM_d^{coh}(X)$. But $j_!(p|_{U'})_* = p_*j'_!$, and it is enough to show that $f_*p_*j'_! 1_{U'}\in DM^{coh}_d(X)$. But then we have a localization triangle:
		 \[
		 	f_* p_*j'_!  1_{U'}\rightarrow f_*p_* 1_{Y'}\rightarrow f_*p_*i'_*  1_{Z'} \rightarrow
		 \]
		 The last term is in $DM^{coh}_{d-1}(X)$ by induction hypothesis (hence in $DM^{coh}_{d, dom}(X)$) since $\dim Z'<\dim Y'=d$. Therefore it is enough to show that middle term is in $DM_{d, dom}^{coh}(X)$. But if $j'':Y'\hookrightarrow \bar Y'$ is an open immersion, then $f_*p_* 1_{Y'} = \bar f_*\bar p_* j''_* 1_{Y'}$. But since $\bar Y'- Y'$ is a simple normal crossing divisor therefore $j''_* 1_{Y'}\in DM_{d, dom}^{coh}(\bar Y')$ by \ref{sncIsCoh}. Since $\bar p, \bar f$ are proper, dominant hence $\bar f_*\bar p_* j''_* 1_{Y'}\in DM_{d,dom}^{coh}(X)\subset DM_{d,dom}^{coh}(X)$ as required.
\end{proof}

\begin{corollary}\label{t-Structure:cohomologicalMotivesHaveArbitraryY}
	We have the inclusions, for all $d\ge 0$:
	\begin{align*}
		S^{coh}_d(X) \subset DM^{coh}_d(X)	\subset DM^{coh}(X) = \bigcup_{d\ge 0}{DM^{coh}_d(X)},& &DM^{coh}_{d,dom}(X)\subset DM^{coh}_d(X)
	\end{align*}
\end{corollary}
\begin{proof}
	By \ref{arbitraryY}, we conclude that $S^{coh}_d(X)\subset DM^{coh}_d(X)$. Let $q:\mathbb P^r \times_k Y\rightarrow X$ be via second projection. Then $q_*1_{\mathbb P^r\times_k X}$ has $p_*1_X(-r)[-2r]$ as summand. In particular the second inclusion holds. In particular $\bigcup_{d\ge 0}{DM^{coh}_d(X)}\subset DM^{coh}(X)$, but $S^{sm, coh}\subset \cup_{d\ge 0}S_d^{sm, coh}$, and hence we also have the reverse inclusion. Last inclusion is obvious.
\end{proof}

This allows us to prove stability properties of $DM^{coh}_d(X)$ which are relevant for formalism of gluing:
	\begin{proposition}[Stability of cohomological motives]\label{t-Structure:stabilityOfCohomologicalMotives}
		Let $f:Y\rightarrow X$ be a morphism of schemes and fix a $d\ge 0$. Then:
		\begin{enumerate}[i.]
			\item For $f$ dominant (in particular an open dense immersion), $f_*( DM^{coh}_{d, dom}(Y)) \subset  DM^{coh}_{d, dom}(X)$.
				\\For $f$ arbitrary $f_*( DM^{coh}_d(Y)) \subset  DM^{coh}_d(X)$.
				\\In particular, for $f$ arbitrary, $f_*(DM^{coh}(Y))\subset DM^{coh}(X)$.
			\item For $f$ quasi-finite, $\overline{f(Y)}\ne X$, $X$ irreducible, $f^*( DM^{coh}_{d+1,dom}(X)) \subset  DM^{coh}_{d}(Y)$
				\\For $f$ quasi-finite, $f^*( DM^{coh}_{d+1,dom}(X)) \subset  DM^{coh}_{d+1,dom}(Y)$
				\\For $f$ quasi-finite (in particular an immersion), $f^*( DM^{coh}_d(X)) \subset  DM^{coh}_d(Y)$
				\\For $f$ arbitrary, $f^*( DM^{coh}(X)) \subset  DM^{coh}(Y)$
			\item For $f$ an open dense immersion, $f_!( DM^{coh}_{d, dom}(Y)) \subset  DM^{coh}_{d, dom}(X)$.
				\\For $f$ quasi-finite (e.g. an immersion), $f_!( DM^{coh}_d(Y)) \subset  DM^{coh}_d(X)$.
				\\In particular for $f$ quasi-finite (e.g. an immersion), $f_!( DM^{coh}(Y)) \subset  DM^{coh}(X)$.
			\item For $f$ a closed immersion $Y\ne X$, $X$ irreducible, $f^!( DM^{coh}_{d,dom}(X)) \subset  DM^{coh}_{d-1}(Y)(-1)$
				\\For $f$ an immersion, $f^!( DM^{coh}_d(X)) \subset  DM^{coh}_d(Y)$.
				\\In particular for $f$ an immersion $f^!( DM^{coh}(X)) \subset  DM^{coh}(Y)$.
		\end{enumerate}
	\end{proposition}
	\begin{proof}
		Since the four functors preserves shifts, triangles and projectors, it is enough to consider objects in the generating sets on left hand side.						

		($i.$): For $p:Z\rightarrow X$, $f_*p_* 1_Z(-c) \iso (fp)_* 1_Z(-c)$, now use \ref{t-Structure:cohomologicalMotivesHaveArbitraryY} and \ref{arbitraryY}.
		
		($ii.$): For $p:Z\rightarrow X$ proper, $f^*p_* 1_Z(-c) \iso p'_* 1_{Z_Y}(-c)$ using proper base change where $p':Z_Y\rightarrow Y$ denotes the fiber product. Now use \ref{t-Structure:cohomologicalMotivesHaveArbitraryY}, noting that $\dim Z_Y \le \dim Z$ if $f$ was quasi-finite to get the third and the fourth inclusion. Furthermore, if $Z\rightarrow X$ is dominant, $X$ irreducible and $\overline{f(Y)}\ne X$ then $\bar Z_Y\subset Z$ is closed proper, hence $\dim Z_Y \le \dim Z-1$ giving the first inclusion. In any case, if $f$ is dominant the second inclusion follows noting that $Z_Y\rightarrow Y$ is also dominant if $p$ was, and we use \ref{arbitraryY}.
		
		($iii.$): Can assume $i.,ii.$ as well. Using $i.$ and Zariski's main theorem, enough to assume that $f$ an open immersion, since for a finite map $f_!\iso f_*$. For $f$ an open immersion, let $g:X-Y \rightarrow X$ denote inclusion of the closed complement (with reduced induced structure). Hence $g$ is quasi-finite (with $\overline{g(X-Y)}\ne X$ in the first case).
		
		Let $p:Z\rightarrow Y$ be any proper (resp. proper, dominant) morphism with $\dim Z\le d$. Compactify $p$ to a proper (resp. proper, dominant) morphism $\bar p:\bar Z\rightarrow X$. 
		Then by localization property, we have a triangle:
		\[
			f_!f^*\bar p_* 1_{\bar Z}(-c) \rightarrow \bar p_* 1_{\bar Z}(-c) \rightarrow g_*g^*\bar p_* 1_{\bar Z}(-c) \rightarrow 
		\]
		The third term lies in $DM^{coh}_d(X)$ (resp. $DM^{coh}_{d,dom}(X)$) using $i.,ii.$. The second term lies in $DM^{coh}_d(X)$ (resp. $DM^{coh}_{d,dom}(X)$ since $\bar p$ is dominant) by \ref{t-Structure:cohomologicalMotivesHaveArbitraryY}. Hence so does the first term. But the first term is same as $f_!p_* 1_Z(-c)$ using base change, as required and we are done.

			{\flushleft
				\begin{minipage}{0.7\textwidth}
					($iv.$): Can assume $i., ii.$. If $f$ is an open immersion $f^!=f^*$, so assume $f$ is a closed immersion. Let $g:W=X-Y\rightarrow X$ denote the open immersion of the complement. Let $q:Z\rightarrow X$ be a proper (resp. proper, dominant) map and let $q':Z_Y\rightarrow Y$ be it's pullback. Let $s_i:S_i\hookrightarrow Z_Y$ be a stratification of $Z_Y$ into regular sub-varieties. We claim that $f^{'!}1_Z\in DM^{coh}_e(Z_Y)(e-d)$ with $e=\dim Z_Y\le \dim Z=d$. 
				\end{minipage}\hspace{0.03\textwidth}\begin{minipage}{0.3\textwidth}\hfill
					\begin{tikzcd}[row sep=small]
						S_i\ar[r, hookrightarrow, "s_i"]	&Z_Y\ar[r, "f'", hookrightarrow]\ar[dd, "q'"]	&Z\ar[dd, "q"]	& & &\\
						\-\\
															&Y\ar[r, "f", hookrightarrow]					&X				& & &\\
						S_0\ar[r, hookrightarrow, "\circ" description, "j=s_0"]	&Z_Y\ar[r, hookleftarrow, "i"]	&W'									
					\end{tikzcd}
				\end{minipage}}
		
		To see this, induct on the number of strata $S_i$. If the number is one, $Z_Y$ is regular. It follows that $f^{'!}1_Z = 1_{Z_Y}(e-d)[2e-2d]$ in this case, and hence is in $DM^{coh}_{e}(Z_Y)(e-d)$.
		
		Otherwise, let $S_0$ be the open strata, and $W'$ be the complement. Then we have a triangle:
		\[
			i_*i^!f^{'!}1_Z	\rightarrow f^{'!}1_Z	\rightarrow j_*j^!f^{'!}1_Z	\rightarrow
		\]
		First and last term live in $DM^{coh}_e(Z_Y)(e-d)$ by induction hypothesis noting that $i_*, j_*$ preserve Tate twists and $DM^{coh}_e(-)$, $\dim W'\le \dim Z_Y = e$, and the category $DM^{coh}_e(-)$ is stable under negative Tate twists.		
	
		Now the result follows noting that $DM^{coh}_{e}(Z_Y)(e-d)\subset DM^{coh}_{d}(Z_Y)$ in general (as $e\le d$, $e-d\le 0$), and $DM^{coh}_{e}(Z_Y)(e-d)\subset DM^{coh}_{d-1}(Z_Y)(-1)$ if $q$ is dominant and $Y\ne X$ (as $e\le d-1$, $e-d\le -1$).
	\end{proof}
	
\begin{proposition}[Formalism of gluing, continuity]\label{coh:continuity}
	Let $d\ge 0$. Then extended formalism of gluing \ref{gluing:extended4functors} and continuity \ref{gluing:continuity} for any $X$, for either of the following three cases:
		\begin{itemize}
			\item $D_Y=DM^{coh}(Y)$ and $D(K)=DM^{coh}(K)\cong DM^{coh}(K^{perf})$
			\item $D_Y=DM_d^{coh}(Y)$ and $D(K) = DM^{coh}_{d-n}(K^{perf})$.% where $n$ is transcendence degree of $K/k$.
			\item Assume $X$ irreducible. Define:
				\begin{align*}
					D_Y&=\begin{cases} DM^{coh}_{d+1,dom}(X)	&Y\subset X\text{ open}\\
										DM^{coh}_d(Y)			&\text{ otherwise}.
						\end{cases}	\\ D(K)&=	\begin{cases}
													DM^{coh}_{d+1-n}(K^{perf})	&\Spec K\hookrightarrow X\text{ generic point}\\
													DM^{coh}_{d-n}(K^{perf})	&\text{otherwise}	
												\end{cases}
				\end{align*}
		\end{itemize}
		where $Y\in Sub(X)$ and $K/k$ a field extension of transcendence degree $n$.

%
%	Let $d\ge 0$. Define $D(Y):=DM^{coh}_d(Y)$ (resp. $DM^{coh}(Y)$) for $Y/k$ any scheme of finite type, and $D(K)=DM^{coh}_{d-n}(K)$ (resp. $DM^{coh}(K)$) for $K$ a field of transcendence degree $n$ over $k$. Then this arrangement satisfies .
%	
%	Alternatively, fix $X$ irreducible of dimension $d_X$. Define $D(Y):=DM^{coh}_{d+1,dom}(Y)$ for $Y\subset X$ open, otherwise $D(Y):=DM^{coh}_d(Y)$. Define $D(K) = DM^{coh}_{d-n}(K)$ for $K$ a field of transcendence degree $n<d_X$ over $k$ and $DM^{coh}_{d+1-d_X}(K)$ if $K=K(X)$, the field of rational functions. Then this arrangement satisfies extended formalism of gluing \ref{gluing:extended4functors} and continuity \ref{gluing:continuity}.
\end{proposition}
\begin{proof}
		Note that $DM^{coh}_d(K)\xrightarrow{\cong}DM^{coh}_d(K^{perf})$ by using separatedness and continuity, and noting the pullback preserves dimensions (conversely, any variety over $K^{perf}$ is defined over a finite purely inseparable extension $L/K$).

		{(Formalism of gluing)}: The above stability properties imply the stability of $DM^{coh}(-)$ resp. $DM_d^{coh}(-)$ under Grothendieck's four functors for immersions \ref{gluing:4functors}, and hence the first two cases satisfies formalism of gluing \ref{gluing:openClosed} since the categories $DM(Y)$ do. The last case is the same, noting that (pullback under) closed immersions takes $DM^{coh}_{d+1,dom}(-)$ to $DM^{coh}_{d}(-)$ and pullback and pushforwards for open immersions preserves $DM^{coh}_{d+1,dom}(-)$.

	(Extended pullbacks): The functor $f^*:DM(X)\rightarrow DM(Y)$ is defined for maps $f:Y\rightarrow X$ of not necessarily finite type, we need to show it restricts to the categories $D(-)$. Let $a\in S^{sm, coh}_d(Y)$. Therefore $a=p_*1_Z$ for $Z\rightarrow Y$ of $\dim Z\le d$ and we assume $p$ to be proper. Let $\epsilon=\Spec K\in Y$ be of transcendence degree $n$. Then if $\epsilon\in p(Z)$, and the fiber over $\epsilon$, $Z_\epsilon$, has dimension $r$, $\dim Z\ge r+n$. Therefore $r\le d-n$, in particular $\epsilon^*(a) = p_*1_{Z_\epsilon}\in DM^{coh}_{d-n}(K)$ by proper base change. If $\epsilon\ne p(Z)$, $\epsilon^*a=0$. Hence in either case we are done. 
	
	For the last bit, one notices that if $Z\rightarrow Y$ was proper dominant with $Z$ irreducible and $\dim Y=d+1$ then $Z_\epsilon\subset Z_W\ne Z$ for $W$ being the closure $\bar \epsilon$. Hence $\dim Z_\epsilon \le \dim Z_W-n \le \dim Y-1-n=d-n$ for $K$ not generic, as required.
	
	(Essentially surjective): Let $a\in S^{coh}_{d-n}(K)$ for some $\epsilon=\Spec K\in X$. Let $Y=\bar \epsilon$. Hence $\dim Y=n$ is the transcendence degree $n$ of $K/k$. Now $a=p_*1_X(-r)$ for some $p:X\rightarrow \Spec K$ proper, with $\dim X\le d-n$. Then, by spreading out, we can find a $U\subset Y$ open dense and a proper map $\bar p: \bar X\rightarrow U$ such that fiber dimensions are $d-n$, and $\epsilon^*\bar p = p$ by proper base change. Define $\bar a = \bar p_*1_{\bar X}$. %We can find a $U\subset X$ open dense such that $U\cap Y= W$, and hence if $i:W\rightarrow U$ is the closed immersion, $\bar a:=i_*\bar a_1\in DM^{coh}_d(U)$. 
	Then $\epsilon^*\bar a= a$ and $\bar a\in DM^{coh}_d(Y)$ by \ref{t-Structure:cohomologicalMotivesHaveArbitraryY}, so we are done. 
	
	Now if $a\in DM^{coh}_{d-n}(K)$ is arbitrary, it is obtained in finitely many steps by taking shifts and cones. Taking shifts easily extends to a neighbourhood. Taking cones also extends because for $c=Cone(f:a\rightarrow b)$, we can work on the intersection of neighborhoods $U$ where $a$ and $b$ extend to, and then, restricting $U$ still further if needed, $f$ extends to $U$ as well by continuity on $DM(X)$.
	
	The same argument holds in the last case, where, when $\Spec K$ denotes the generic point, we work with varieties of $d+1-n$ instead of $d$ and note that the variety we get by spreading out is necessarily dominant.
	
	(Full, Faithful): These properties directly follow from continuity of $DM(-)$.
\end{proof}
\subsection{Weight truncations over a field} \label{sec:overField}
\emph{In this subsection $k$ will denote a \emph{perfect} field.} 

\begin{para}\label{fields:defineCats} Fix a perfect field $k$. Let $DM(k)$ denote the rigid symmetric monoidal triangulated category of motives over a field $k$ with $\Q$ coefficients as in \S \ref{sec:motives}. Recall that we have natural functors:
\[
	(SmProj/k)^{op}\xrightarrow{h} CHM^{eff}(k) \overset M\hookrightarrow DM(k)
\]
with $M$ being a fully faithful embedding. By abuse of notation we will often let $M(A)$ be denoted by $A$ for any Chow motive $A$. (Hence, in particular, for $X$ smooth projective over $\Spec k$, $h(X)$ denotes $M(X)$ when there is no confusion).

For $i=0, 1$, we make the following definitions:
\begin{align*}
	S^{'\le i}(k)	:=& \{h^{\le i}(X)\in DM(k)\big| X\text{ is smooth, projective, connected over }k\} \\
	S^{\le i}(k)	:=& \{h^{\le i}(X)\in DM(k)\big| X\text{ is smooth, projective, connected over }k, \dim X\le i\} \\
	S^{> i}(k)		:=& \{h^{> i}(X)\in DM(k)\big| X\text{ is smooth, projective, connected over }k\}
\end{align*}
and for $i=2$, we define:
\begin{align*}
	S_a^{\le 2}(k):=&	\{h^{\le 2}(X)\in DM(k)\big| X\text{ is smooth, projective, connected over }k, \dim X\le 2\} \\
	S_b^{\le 2}(k):=&	\{h^{\le 0}(X)(-1)\in DM(k)\big| X\text{ is smooth, projective, connected over }k, \dim X\le 2\}  \\
	S_a^{> 2}(k):=&	\{h^{> 2}(X)\in DM(k)\big| X\text{ is smooth, projective, connected over }k, \dim X\le 2\}  \\
	S_b^{> 2}(k):=&	\{h^{> 0}(X)(-1)\in DM(k)\big| X\text{ is smooth, projective, connected over }k, \dim X\le 2\}  \\
	S_c^{> 2}(k):=&	\{h(X)(-r)\in DM(k)\big| X\text{ is smooth, projective, connected over }k, \dim X\le 2,r\ge 2\}  \\
	S^{\le 2}(k):=&S_a^{\le 2}(k)\cup S_b^{\le 2}(k) \;\;\;\;\;\;\;\;\;\;\;\;\;\;\; S^{> 2}(k):=S_a^{> 2}(k)\cup S_b^{> 2}(k)\cup S_c^{> 2}(k)
\end{align*}
and finally define:
\begin{align*}
	{^wDM^{\le i}}(k) :=& \vvspan{S^{\le i}(k)} & 	{^wDM^{> i}}(k) :=& \vvspan{S^{> i}(k)}	& \text{for }i=0,1\\
	{^wDM_2^{\le 2}}(k) :=& \vvspan{S^{\le 2}(k)} & 	{^wDM_2^{> 2}}(k) :=& \vvspan{S^{> 2}(k)}	
\end{align*}
where the generation is inside the triangulated category $DM(k)$.
\end{para}

\begin{lemma}\label{fields:dimLessThaniIsGoodEnough}
	For $i=0, 1$ we have $\vvspan{S^{'\le i}} = \vvspan{S^{\le i}}$
\end{lemma}
\begin{proof}
	It is enough to show that $h^{\le i}(X)\in \vvspan{S^{'\le i}}$, for $i=0, 1$. Since $h^0(X)=h^{\le 0}(X)$ is a summand of $h^{\le 1}(X)$, $\vvspan{S^{'\le 0}}\subset \vvspan{S^{'\le i}}$ it is enough to show that $h^i(X)\in \vvspan{S^{'\le i}}$, for $i=0,1$. For $i=0$, note that $h^0(X) = h(\Spec L)$ if $X\rightarrow \Spec L\rightarrow \Spec k$ is the Stein factorization and hence is in $\vvspan{S^{'\le 0}}$. For $h^1(X)$, the claim follows from \ref{h1isInCurve}, using Galois descent to go to a field extension.
\end{proof}

It is on the category $DM^{coh}(k)$ (resp. $DM_2^{coh}(k)$), we will be able to extend the functor $h^{\le i}$ for $i\in\{0,1\}$ (resp. $i=2$). The extended functor will be the motivic analogue of the Morel's weight truncations \cite{morelThesis} over a field and hence will be denoted as $w_{\le i}$. We first prove the key proposition in case $i\in\{0, 1\}$:
\begin{proposition}\label{orthogonality}
	Let $A\in S^{\le i}$, $B\in S^{>i}$ for $i=0,1$. Then 
		\[	
			\hom(A,B[m])=0\text{ for all }m\in \Z.
		\]
\end{proposition}
\begin{proof}	
	\emph{Case $i=0$}: We have to show that $\hom(h^0(X), h^{>0}(Y)[m]) = 0$. It is enough to show that the inclusion 
		\[
			\hom(h^0(X), h^0(Y)[m])\overset{p}\hookrightarrow \hom(h^0(X), h(Y)[m])
		\]
	is an isomorphism. Notice that $h^0(X) = \Spec k'$ for some finite extension $k'/k$ by construction of $h^0(X)$, and hence $h^0(X)^\vee\cong h^0(X)=h(\Spec k')$. Therefore we have:
	\[
		\hom(h^0(X),h(Z)[m])=\hom( 1_k, h(Z\otimes_k k')[m])=H^{m,0}(Z\otimes_k k')
	\]
which vanishes for $m\ne 0$. Hence $p$ is an isomorphism for $m\ne 0$. For $m=0$, let $Y\rightarrow \Spec k''\rightarrow \Spec k$ be the Stein factorization. Then the morphism reduces to:
	\[
		CH^0(\Spec k''\otimes_k k')\hookrightarrow CH^0(Y\otimes_k k')
	\]
	which is obtained by pullback via projection $Y\otimes_k k'\rightarrow \Spec k''\otimes_k k'$ using \ref{chow:h0ispullback}. This is an isomorphism as $Y$ is geometrically connected over $k''$.
	
	\emph{Case $i=1$}: It is enough to show that $\hom(h^1(X), h^{>1}(Y)[m]) = 0$, since the case for $i=0$ already gives vanishing of $\hom(h^0(X), h^{>1}(Y)[m])$.
	
	By \ref{h1isInCurve} $h^1(X)$ is a summand of $h^1(C)$ for some $C$ smooth projective curve, and hence we can work with $X=C$. We have a (split) distinguished triangle:
	\begin{align*}
		h^{\le 1}C\rightarrow C \rightarrow h^2(C)\overset 0 \rightarrow \text{ with }h^2(C)\cong h^0(C)(-1)[-2]
	\end{align*}
	
	Hence we get a split long exact sequence:
	\begin{align}\label{eq:long}
		\overset 0\rightarrow\hom(h^{2}(C),h(Y)[m])\rightarrow\hom(h(C),h(Y)[m])
				\rightarrow\hom(h^{\le 1}(C),h(Y)[m])\overset 0 \rightarrow%\hom(h^{2}(C)[-1],h(Y)[m])\rightarrow	
	\end{align}
	
	Let $C\rightarrow \Spec k'\rightarrow \Spec k$ be the Stein factorization. Therefore \FIXME{internal reference}:
	\begin{align*}
		\hom(h^2(C),h(Y)[m]) 	&= \hom(h^0(C)(-1)[-2], h(Y)[m]) \\
								&= \hom( 1_k, h(Y\otimes_k k')[m+2](1)) = H^{m+2,1}(Y\otimes_k k') \\
		\hom(h(C),h(Y)[m]) &= \hom( 1_k, h(Y\times_k C)[m+2](1)) = H^{m+2,1}(Y\times_k C)
	\end{align*}
	By computations of motivic cohomology, the only non-vanishing terms occur in the cases $m=0, m=-1$. Hence $\hom(h^{\le 1}(C), h(Y)[m])=0$ for $m\ne 0, -1$. For $m=-1$ note that $H^{1,1}(Z)=\mathcal O^*(Z)\otimes\Q$. Therefore the corresponding short exact sequence becomes:
	\[
		0\rightarrow \mathcal O^*(Y\otimes_k k')\otimes\Q\rightarrow \mathcal O^*(Y\times_k C)\otimes\Q\rightarrow \hom(h^{\le 1}(C), h(Y)[-1])\rightarrow 0
	\]
	But the first map is an isomorphism: since $C$ is projective with $\Spec k'$ as the Stein factorization, $p:Y\times_k C\rightarrow Y\otimes_k k'$ is smooth with geometrically connected fibers, and hence $Rp_*\mathcal O(Y\times_k C) = \mathcal O(Y\otimes_k k')$. Hence $\hom(h^{\le 1}(C), h(Y)[-1])=0$. Since $h^{>1}(Y)$ is a summand of $h(Y)$, we conclude that $\hom(h^{\le 1}(C), h^{>1}(Y)[m]) = 0, m\ne 0$. 
	
	For the case $m=0$, let $Y\xrightarrow p\Spec k''\rightarrow \Spec k$ be the Stein factorization. We fix a $l\supset k'', k'$, and let $s:\Spec l\rightarrow \Spec k$ denote the natural map. Let subscript $l, k$ denote the field over which we are working (thus $h_k(-) = h(-)$, for example). We have $h_l(C\otimes_k l) = s^*h(C)$ and therefore by \ref{chow:pullbacks}, we can let $h^{\le 1}_l(C\otimes_k l) =s^*(h^{\le 1}(C))$. We fix a choice of $h^{>1}(Y)$. Then, using \ref{chow:pullbacks} and \ref{chow:pushforwards} we have $h^{>1}_k(Y\otimes_k l) = s_*h^{>1}_l(Y\otimes_k l) = s_*s^*h^{>1}(Y)$. Therefore:
	\begin{align*}
		\hom_{DM(l)}(h_l^{\le 1}(C\otimes_k l), h_l^{>1}(Y\otimes_k l) 	&= \hom_{DM(l)}(s^*h^{\le 1}(C), h_l^{>1}(Y\otimes_k l)) \\
																		&= \hom_{DM(k)}(h^{\le 1}(C), s_*s^*h^{>1}(Y))\supset \hom_{DM(k)}(h^{\le 1}(C), h^{>1}(Y))
	\end{align*}
	since $h^{>1}(Y)$ is a summand of $s_*s^*h^{>1}(Y)$ as $s$ is etale. Now connected components of $C\otimes_k l$, resp. $Y\otimes_k l$ are geometrically connected over $l$, and replacing $k$ by $l$ and $C$ (resp. $Y$) by a connected component of $C\otimes_k l$ (resp. $Y\otimes_k l$) we can assume that $C, Y$ are geometrically connected.
		
	Then we have a natural diagram:
	\[\begin{tikzcd}
				&0\ar[d]											&0\ar[d]											&0\ar[d]\\
		0\ar[r]	&\hom(h^2(C),h^{\le 1}(Y))\ar[r, ]\ar[d]			&\hom(h(C),h^{\le 1}(Y))\ar[r]\ar[d, "\beta"]		&\hom(h^{\le 1} (C), h^{\le 1}(Y))\ar[d]\ar[r]&0\\				
		0\ar[r]	&\hom(h^2(C),h(Y))\ar[r, "\alpha"]					&\hom(h(C),h(Y))\ar[r]								&\hom(h^{\le 1} (C), h(Y))\ar[d]\ar[r]&0\\
				&													&													&\hom(h^{\le 1} (C), h^{>1}(Y))\ar[d]\\
				&													&													&0
	\end{tikzcd}\]
	We want to show that the term on the bottom right is $0$. It is enough to show that $\alpha\oplus \beta$ is surjective, using snake lemma. Let $d=\dim Y$. This reduces to \ref{chow:h1ispullback} noting that:
	\begin{align*}
		h(Y) &= h(Y)^\vee\otimes\L^{\otimes d}	&	
		h(C) &= h(C)^\vee\otimes \L &
		h^2(C)	&= h^0(C)\otimes \L	\\
		h^2(C)^\vee &= h^0(C)\otimes \L^{-1} &
		h^{\le 1}(C)^\vee &= h^{>1}(C)\otimes \L^{-1} &
		h^{\le 1}(Y)^\vee 	&= h^{>2d-1}(Y)\otimes \L^{-d}
	\end{align*}
	and using \ref{chow:duality}.
\end{proof}

We now prove some observations which will be useful later, in particular to prove orthogonality in the case $i=2$, as well as later:

\begin{lemma}\label{tateTwistsIncreaseWeight}
	Let $M\in DM^{coh}(k)$ and $A=h^{\le 1}(C)\in S^{' \le 1}(k)$ with $\dim C\le 1$. Then 
	\[
		A\perp M(-r)[m], \forall r\ge 1, \forall m.
	\]
\end{lemma} 
\begin{proof}
	If $\dim C=0$, $A=h^0(C)$, and the claim follows using \ref{orthogonality} and noting that $\L=1_k(-1)[-2]=h^2(\mathbb P^1)$. Let $f:C\rightarrow \Spec k$ be any curve with Stein factorization $C\xrightarrow{g} \Spec L\xrightarrow{h} \Spec k$. Since $DM^{coh}(k) = \vvspan{S^{sm, coh}(k)}$, it will be enough to assume $M=p_*1_X = h(X)$ with $X$ smooth, projective over $k$. Noting that $h(C)^\vee = h(C)(1)[2]$, we get an isomorphism:
	\[
		\hom(h(C), h(X)(-r)[m]) = \hom(1_k, h(C\times_k X)(-r+1)[m+2]) = H^{m+2, -r+1}(X)
	\]
	If $r>1$, this is $0$, and if $r=1$, this is $0$ if $m\ne -2$. Hence $\hom(h^{\le 1}(C), M(-r)[m]) = 0$, being a summand, except possibly for $r=1, m=-2$. In this case, we have% a long exact sequence:
	\[\begin{tikzcd}[row sep=small]
		\hom(h^{2}(C), h(X)(-1)[-2])\ar[r, hookrightarrow]\ar[d, equal] 	&\hom(h(C), h(X)(-1)[-2])\ar[r, twoheadrightarrow]\ar[d, equal] &\hom(h^{\le 1}(C), h(X)(-1)[-2]) \\
		CH^0(X\otimes_k L)									&CH^0(X\times_k C)
	\end{tikzcd}\]
	a long exact sequence and it is enough to show that the first map is an isomorphism. But the first map is induced via the natural projection: $g:C\rightarrow \Spec L$ by \ref{chow:h0ispullback}, and the pullback map $CH^0(X\otimes_k L)\rightarrow CH^0(X\times_k C)$ is an isomorphism since $C\rightarrow \Spec L$ is smooth with geometrically connected fibers.
\end{proof}

\begin{proposition}
	Let $A\in S^{\le 2}(k), B\in S^{>2}(k)$. Then $\hom(A, B[m])=0$ for all $m\in \Z$.
\end{proposition}
\begin{proof}
	By \ref{chow:h0ispullback}, and \ref{chow:h1ispullback}, we can conclude:
	\begin{align*}
		A=h^{\le 2}(X) \in S_a^{\le 2}(k)			&\Rightarrow 	\begin{cases}	A = A^\vee = h(X) = h^0(X)	 						&\dim X = 0\\
																					A = h(X), A^\vee = h(X)(1)[2]						&\dim X = 1\\
																					A^\vee = h^{> 1}(X)(2)[4]							&\dim X = 2
																	\end{cases}\\
		A=h^{\le 0}(X)(-1)\in S_b^{\le 2}(k)		&\Rightarrow 					A=h^0(X)(-1), A^\vee = h^{0}(X)(1)\\
		B=h^{> 2}(Y)\in S_a^{> 2}(k)				&\Rightarrow 	\begin{cases} 	B = 0											&\dim Y < 2\\
																					B^\vee = h^{\le 1}(Y)(2)[4]						&\dim Y = 2
																	\end{cases}\\
		B=h^{>0}(Y)(-1)\in S_b^{> 2}(k)				&\Rightarrow	\begin{cases}	B = 0											&\dim Y = 0\\
																					B^\vee = h^{\le 1}(Y)(2)[2]						&\dim Y = 1%\\
%																																	&\dim Y = 2
																	\end{cases}\\
		B=h(Y)(-r)\in S_c^{> 2}(k), r\ge 2			&\Rightarrow					B^\vee = h(Y)(r+\dim Y)[2\dim Y]	
	\end{align*}
	By duality in $DM(k)$, we have the equalities:
	\[
		\hom(A, B[m]) = \hom(1, A^\vee\otimes B[m]) = \hom(B^\vee, A^\vee[m])
	\]
	Then, the vanishing is proved through a case by case analysis, and the reasons are summarised in the table below. 
	
	\emph{Notation:} We let $A\in S^{\le 2}_p(k)$ for $p\in \{a,b\}$ be a summand of $h(X)$ and $B\in S^{>2}_q(k)$ for $q\in \{a, b, c\}$ be a summand of $h(Y)$ as above. We assume $\dim X=d_X$ and $\dim Y=d_Y$. The place holder $\ast$ will stand for all acceptable values:\\
		\begin{tabular}{r|c|c|p{32em}}
		\hline
		$(p,q)$ 	&$d_X$	&$d_Y$ 	&Reason of vanishing of $\hom(A,B[\ast])$ for $A\in S_p^{\le 2}(k), B\in S_q^{>2}(k)$\\
		\hline
		$(*, a)$	&*			&0, 1		&$B=0$\\
		$(a, a)$	&0			&2			&$\hom(h^0(X), h^{>2}(Y)[\ast])\subset \hom(h^0(X), h^{>0}(Y)[\ast])=0$ by \ref{orthogonality}.\\
		$(a, a)$	&1			&2			&$\hom(B^\vee, A^\vee[\ast])= \hom(h^{\le 1}(Y), h(X)(-1)[\ast])=0$ by \ref{tateTwistsIncreaseWeight}.\\
		$(a, a)$	&2			&2			&$\hom(B^\vee, A^\vee[\ast])= \hom(h^{\le 1}(Y), h^{>1}(X)[\ast])=0$ by \ref{orthogonality}.\\
		$(b, a)$	&*			&2			&$\hom(B^\vee, A^\vee[\ast])= \hom(h^{\le 1}(Y), h^0(X)(-1)[\ast])=0$ by \ref{tateTwistsIncreaseWeight}.\\
		\hline		
		$(*, b)$	&*			&0			&$B=0$\\
		$(a, b)$	&0			&1			&$\hom(B^\vee, A^\vee[\ast])=\hom(h^{\le 1}(Y), h^0(X)(-2)[\ast])=0$ by \ref{tateTwistsIncreaseWeight}.\\
		$(a, b)$	&1			&1			&$\hom(B^\vee, A^\vee[\ast])= \hom(h^{\le 1}(Y), h(X)(-1)[\ast])=0$ by \ref{tateTwistsIncreaseWeight}.\\
		$(a, b)$	&2			&1			&$\hom(B^\vee, A^\vee[\ast])= \hom(h^{\le 1}(Y), h^{>1}(X)[\ast])=0$ by \ref{orthogonality}.\\
		$(a, b)$ 	&0			&2			&$\hom(1_k, B\otimes A^\vee[\ast])\subset \hom(1_k, h(Y\times_k X)(-1)[\ast])=0$ by \ref{tateTwistsIncreaseWeight}.\\
		$(a, b)$ 	&1			&2			&$\hom(1_k, B\otimes A^\vee[\ast])\subset \hom(1_k, h^{>0}(Y)\otimes h(X)(-r+2)[\ast])=0$ by below.\\		
		$(a, b)$	&2			&2			&$\hom(h^{\le 2}(X), h^{>0}(Y)(-1))=0$ by below.\\
		$(b, b)$	&*			&*			&$\hom(h^0(X)(-1), h^{>0}(Y)(-1)) = 0$ by \ref{orthogonality}\\
		\hline
		$(a, c)$	&0, 1		&*			&$\hom(1_k, B\otimes A^\vee[\ast]) \subset \hom(1_k, h(Y\times_k X)(-r+d_X)[\ast])= 0$ as $r>d_X$\\
		$(a, c)$	&2			&*			&$\hom(1_k, B\otimes A^\vee[\ast]) \subset \hom(1_k, h^{>0}(X)\otimes h(Y)(-r+2)[\ast]) =0$ by below.\\
		$(b, c)$	&*			&*			&$\hom(1_k, B\otimes A^\vee[m]) \subset \hom(1_k, h(Y\times X)(-r+1)[\ast])= 0$ as $r>1$.\\
		\hline
	\end{tabular}\\
	
	Here the only non-trivial cases are $(p,q,d_X, d_Y)\in \{(a,b,1,2), (a,b,2,2), (a, c, 2, \ast)\}$. For the cases $(p,q,d_X, d_Y) = (a,b,1,2)$ or $(a, c, 2, \ast)$ it is enough to show that 
	\[
		\hom(1_k, h^{>0}(X)\otimes h(Y)(-s)[m]) = 0\text{ for all }s \ge 0, \text{ for all }m
	\]
	This is a summand of $H^{m,-s}(X\times_k Y)$ and hence vanishes unless $s=m=0$. Let $X\xrightarrow g \Spec L\rightarrow \Spec k$ be the Stein factorization. In this case we have a split exact sequence:
	\[\begin{tikzcd}[row sep=small]
		\hom(1_k, h^0(X)\otimes h(Y))\ar[r, hookrightarrow]\ar[d, equal] 	&\hom(1_k, h(X)\otimes h(Y))\ar[r, twoheadrightarrow]\ar[d, equal] &\hom(1, h^{>0}(X)\otimes h(Y)) \\
		CH^0(Y\otimes_k L)									&CH^0(Y\times_k X)
	\end{tikzcd}\]
	and it is enough to show that the first map is an isomorphism. But the first map is induced via the natural projection: $g:X\rightarrow \Spec L$ by \ref{chow:h0ispullback}, and the pullback map $CH^0(Y\otimes_k L)\rightarrow CH^0(Y\times_k X)$ is an isomorphism since $X\rightarrow \Spec L$ has geometrically connected fibers. Note that $d_Y$ plays no role here.
	
	For $(p,q,d_X,d_Y)=(a,b,2,2)$ we have to show that $\hom(h^{\le 2}(Y), h^{>0}(Y)(-1)[m]) = 0$. But we have:
	\begin{align*}
		\hom&(h^{\le 2}(X), h^{>0}(Y)(-1)[m]) = \hom(1_k, h^{>1}(X)\otimes h^{>0}(Y)(1)[m+4]) \\
								 &\subset \hom(1_k, h(X)\otimes h^{>0}(Y)(1)[m+4]) \subset \hom(1_k, h(X)\otimes h(Y)(1)[m+4]) = H^{m+4,1}(X\times_k Y)
	\end{align*}
	Therefore, this vanishes for $m\ne -3, -2$. For $m=-3$, we have a split exact sequence:
	\[\begin{tikzcd}[row sep=small, column sep=small]
		\hom(1_k, h(X)\otimes h^0(Y)(1)[1])\ar[r, hookrightarrow]\ar[d, equal] 	&\hom(1_k, h(X)\otimes h(Y)(1)[1])\ar[r, twoheadrightarrow]\ar[d, equal] &\hom(1_k, h(X)\otimes h^{>0}(Y)(1)[1]) \\
		\mathcal O^*(X\otimes_k L')\otimes\Q												&\mathcal O^*(X\times_k Y)\otimes\Q
	\end{tikzcd}\]
	where $Y\xrightarrow g\Spec L'\rightarrow \Spec k$ is the Stein factorization, and the first map, induced by pullback along $g$ by \ref{chow:h0ispullback}, is an isomorphism since $X\times_k Y\rightarrow X\otimes_k L'$ is smooth with geometrically connected fibers. 
	
	Now assume $m=-2$. Assume $X\rightarrow \Spec L\rightarrow \Spec k$, resp. $X\rightarrow \Spec L'\rightarrow \Spec k$ are the Stein factorizations and let $M\supset L', L$. 
Let $s:\Spec M\rightarrow \Spec k$ denote the natural map. Let subscript $M, k$ denote the field over which we are working (thus $h_k(-) = h(-)$, for example). We have $h_M(X\otimes_k M) = s^*h(X)$ and therefore by \ref{chow:pullbacks}, we can let $h^{\le 2}_M(X\otimes_k M) =s^*(h^{\le 2}(M))$. We also fix a choice of $h^{>0}(Y)$. Then, using propositions \ref{chow:pullbacks} and \ref{chow:pushforwards} we have the equalities $h^{>0}_k(Y\otimes_k M) = s_*h^{>0}_M(Y\otimes_k M) = s_*s^*h^{>0}(Y)$. Therefore:
	\begin{align*}
		\hom_{DM(M)}&(h_M^{\le 2}(X\otimes_k M), h_M^{>0}(Y\otimes_k M)(-1)[-2]) \\
										&= \hom_{DM(M)}(s^*h^{\le 2}(X), h_M^{>0}(Y\otimes_k M)(-1)[-2]) \\
																		&= \hom_{DM(k)}(h^{\le 2}(X), s_*s^*h^{>0}(Y)(-1)[-2])
																		\supset \hom_{DM(k)}(h^{\le 2}(X), h^{>0}(Y)(-1)[-2])
	\end{align*}
	since $h^{>0}(Y)$ is a summand of $s_*s^*h^{0}(Y)$ as $s$ is etale. Now connected components of $X\otimes_k M$, resp. $Y\otimes_k M$ are geometrically connected, and replacing $k$ by $M$ and $X$ (resp. $Y$) by a connected component of $X\otimes_k M$ (resp. $Y\otimes_k M$) we can assume that $X, Y$ are geometrically connected.	
	
	Then we have a diagram:
	\[\begin{tikzcd}[column sep=small]
		\hom(h^{> 2}(X), h^{0}(Y)(-1)[-2])\ar[r, hookrightarrow]\ar[d, hookrightarrow] &\hom(h^{> 2}(X), h(Y)(-1)[-2])\ar[r, twoheadrightarrow]\ar[d, hookrightarrow, "\beta"] &\hom(h^{> 2}(X), h^{>0}(Y)(-1)[-2])\ar[d, hookrightarrow]\\
		\hom(h(X), h^{0}(Y)(-1)[-2])\ar[r, hookrightarrow, "\alpha"]			&\hom(h(X), h(Y)(-1)[-2])\ar[r, twoheadrightarrow]				&\hom(h(X), h^{>0}(Y)(-1)[-2])\ar[d, twoheadrightarrow]	\\
			&	&\hom(h^{\le 2}(X), h^{>0}(Y)(-1)[-2])
	\end{tikzcd}\]
	where the two rows and the last column are split exact sequences. We want to show that the term in bottom right vanishes, hence it is enough to show that $\alpha\oplus\beta$ is surjective. This follows from \ref{chow:h1ispullback} since $\dim X=2$.
\end{proof}

\begin{remark}\label{field:twistOfG1IsG3}
	We emphasize that in the proof above for $(p,q,d_X,d_Y)=(a,b,1,2)$ or $(a,c,1,*)$, $d_Y$ is irrelevant, and we can even have $d_Y>2$. This will be useful later.
\end{remark}

\begin{corollary}\label{field:isTstructure}
	$({^wDM}^{\le i}(k), {^wDM}^{> i}(k))$ is a $t$-structure on $DM^{coh}(k)$ for $i=0, 1$.
	Also, $({^wDM_2}^{\le 2}(k), {^wDM_2}^{> 2}(k))$ is a $t$-structure on $DM_2^{coh}(k)$.
\end{corollary}
\begin{proof}
	Use \ref{tFromGenerators}.
\end{proof}

\begin{definition}
	For $i\in \{0,1,2\}$, we let $(w_{\le i}, w_{>i})$ denote the truncations for the corresponding $t$-structure above.
\end{definition}

\begin{corollary}\label{field:tateTwistsIncreaseWeights}
	$DM^{coh}(k)(-r)\subset {^wDM^{>i}}(k)$ for $i=0,1$ and $r\ge 1$.
\end{corollary}
\begin{proof}
	It is enough to prove that ${^wDM^{\le 0}}(k)\subset {^wDM^{\le 1}}(k)\perp DM^{coh}(k)(-r)$ for $r\ge 1$. By \ref{gluing:decomposition}, it is enough to prove the inclusion $S^{\le 0}(k)\subset {^wDM^{\le 1}}(k)$ and the orthogonality $S^{'\le 1}(k)\perp DM^{coh}(k)(-r)$. But the inclusion is clear since $h^{\le 0}(X)$ is a retract of $h^{\le 1}(X)$ and the orthogonality follows from \ref{tateTwistsIncreaseWeight}. 
\end{proof}

\subsection{Weight truncations over a base}\label{sec:overBase} In the previous section analogue of S.~Morel's $t$-structures were constructed on appropriate subcategories of $DM(k)$ for $k$ a perfect field. In this section we use punctual gluing to relativize the situation using \ref{gluing:mainresult} and construct the corresponding $t$-structures over an arbitrary base $S$.

%We want to glue $t$-structures constructed in \ref{field:isTstructure} using \ref{gluing:mainresult}.
We begin by defining a situation which will figure repeatedly in the arguments below:
\begin{para}\label{complicatedDiagram} Fix a scheme $S$, a point $\epsilon:\Spec K\rightarrow S$ in $S$, and a point $\delta:\Spec L\rightarrow \bar \epsilon$ (where $\bar \epsilon$ denotes the closure of $\epsilon$ with reduced induced structure). 

We will be interested in objects of $DM(U)$ (below, cohomological motive of $\bar X$) on some neighborhood $U$ of $\Spec K$, which restricts to a given object in $DM(K^{perf})$ (below, $h(X')$) and to understand the restriction of the same to $DM(L^{perf})$ whenever $\Spec L$ is a point in $U$. 

In particular we have the following diagram, once we fix $X'/\Spec K^{perf}$ smooth, projective:

	
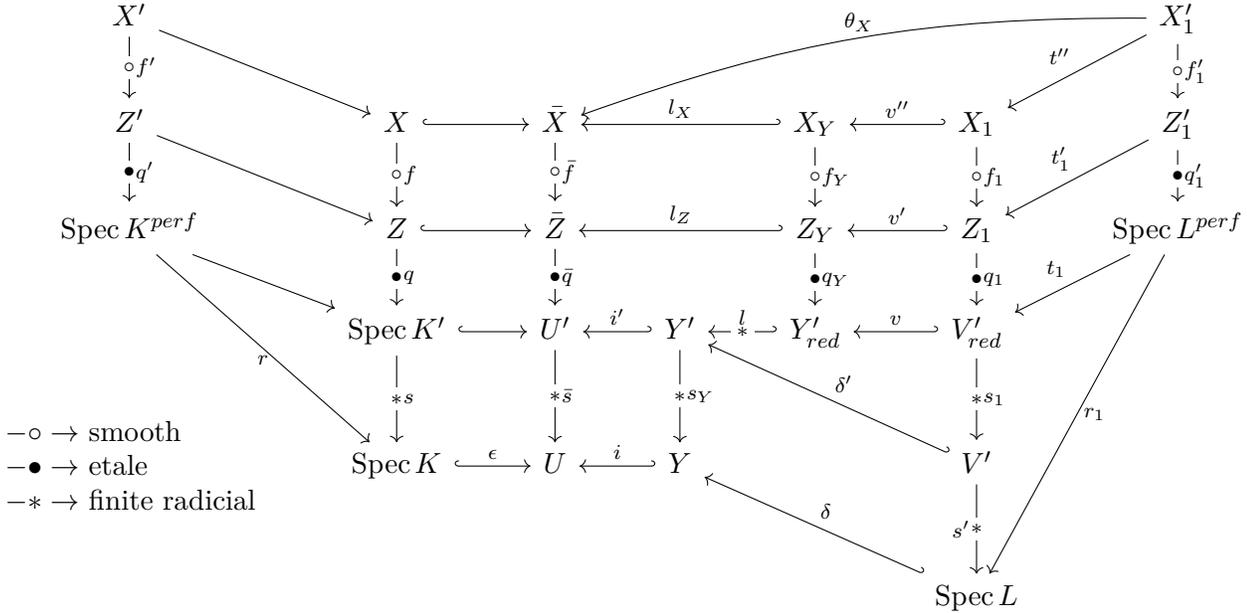
\[\hspace{-0.1\linewidth}\begin{tikzcd}
		X'\ar[d, "\circ" description, "f'"]\ar[rd,rightarrow]  &	 & & & & &X'_1\ar[d, "\circ" description, "f'_1"]\ar[dllll, "\theta_X" above, bend right=10]\\ 
		Z'\ar[d, "\bullet" description, "q'"]\ar[rd, rightarrow] &X\ar[d, "\circ" description, "f"]\ar[r, hookrightarrow] &\bar X\ar[d, "\circ" description, "\bar f"]\ar[rr, hookleftarrow, "l_X"]& &X_Y\ar[d, "\circ" description, "f_Y"]\ar[r, hookleftarrow, "v''"]&X_1\ar[d, "\circ" description, "f_1"]\ar[ru, leftarrow, "t''"]&Z'_1\ar[d, "\bullet" description, "q'_1"]\\
		\Spec K^{perf}\ar[rd]\ar[ddr, "r" below]		 &Z\ar[d, "\bullet" description, "q"]\ar[r, hookrightarrow]	   &\bar Z\ar[d, "\bullet" description, "\bar q"]\ar[rr, hookleftarrow, "l_Z"]& &Z_Y\ar[d, "\bullet" description, "q_Y"]\ar[r, hookleftarrow, "v'"] &Z_1\ar[d, "\bullet" description, "q_1"]\ar[ru, leftarrow, "t'_1"]&\Spec L^{perf}\\
			&\Spec K'\ar[d, "\ast" description, "s"]\ar[r, hookrightarrow]&U'\ar[r, "i'", hookleftarrow]\ar[d, "\ast" description, "\bar s"] &Y'\ar[rrd, hookleftarrow, "\delta'"]\ar[d, "\ast" description, "s_Y"]\ar[r, "\ast" description, "l", hookleftarrow]	   &Y'_{red}\ar[r,hookleftarrow, "v"] &V'_{red}\ar[d, "\ast" description, "s_1"]\ar[ru, leftarrow, "t_1"]\\
	\begin{minipage}{0.4\linewidth}
		$-\circ\rightarrow$ smooth\\
		$-\bullet\rightarrow$ etale\\
		$-\ast\rightarrow$ finite radicial
	\end{minipage}\hspace{-0.2\linewidth}
			&\Spec K\ar[r, hookrightarrow, "\epsilon"]&U\ar[r, "i", hookleftarrow]	&Y\ar[rrd, hookleftarrow, "\delta"]	&	   &V'\ar[d, "\ast" description, "s'" left]\\
			& &\- & & &\Spec L\ar[ruuu, leftarrow, "r_1" right]	
	\end{tikzcd}\]
	
		The diagram is obtained as follows (beginning with $X'/\Spec K^{perf}$ smooth, projective):
		\begin{itemize}
			\item $X'\xrightarrow{f'}Z'\xrightarrow{q'} \Spec K^{perf}$ is the Stein factorization. Hence $f'$ is smooth $q'$ etale.
			\item The situation descends to a finite extension $K'/K$ giving $X, Z, f, q$ with $f, q$ smooth.
			\item $\bar X, \bar Z, U' \bar f, \bar q, \bar s$ obtained by spreading out in the closure $\bar \epsilon$. Hence can assume $\bar f, \bar q, \bar s$ proper. Restricting $U$ can assume it is regular. Restricting $U$ further, can assume $U'$ is regular as well since $\bar s$ is finite. Hence $\bar X$, $\bar Z$ are regular.
			\item $Y\subset \bar \delta$ is regular, open, dense neighborhood of $\Spec L$ inside the closure $\bar \delta$ with the reduced induced sub-structure.
			\item $Y'$, resp. $V'$ is the pullback of $U'$ to $Y$ resp. $\Spec L$. 
			\item $V'_{red}$ and $Y'_{red}$ are the schemes with reduced structure. Shrinking $Y$ can assume $Y'_{red}$ is regular since $s_Y$ is finite.
			\item $V'_{red}\rightarrow \Spec L$ is finite radicial, hence $V'_{red}=\Spec L'$ with $L'/L$ finite purely inseparable. Hence $r_1$ factors through $s'\circ s_1$.
			\item Remaining objects are obtained via base change to $Y'_{red}, V'_{red}$ and $\Spec L^{perf}$. By construction, all of them are regular.
		\end{itemize}
		
		We also use the notation $p=q\circ f$, $\bar p = \bar q\circ \bar f$, etc. We let $\bar g:=\bar s\circ \bar p$, and $\bar h=\bar s\circ \bar q$.
\end{para}
\begin{lemma}\label{complicatedDiagram:lemma}
	Assume we are in the situation of \ref{complicatedDiagram}. Let $d=\dim Y'-\dim U'$, and $e=\dim X$. Then we have a natural isomorphism of functors:
	\begin{align*}
		\alpha(\bar g):r_1^*\delta^*i^*\bar g_*1_{\bar X} &\xrightarrow\cong p'_{1*}1_{X'_1} & \beta(\bar g):r_1^*\delta^*i^!\bar g_*1_{\bar X} &\xrightarrow\cong p'_{1*}1_{X'_1}(-d)[-2d]
	\end{align*}
	such the the following diagrams, where the vertical arrows are obtained using natural maps of adjunction $1_{\bar Z}\rightarrow \bar f_*1_{\bar Z}$ (resp. for $Z'_1, X'_1$) and $\bar f_*\bar f^!1_{\bar Z}\rightarrow 1_{\bar Z}$ (resp. for $Z'_1, X'_1)$, commutes:
	\[\begin{tikzcd}[column sep=tiny]
		r_1^*\delta^*i^*\bar g_*1_{\bar X}\ar[rr, "\alpha(\bar g)"] & &p'_{1*}1_{X'_1} 	 
		&r_1^*\delta^*i^*\bar g_*1_{\bar X}\ar[rr, "\alpha(\bar g)"] & &p'_{1*}1_{X'_1} 
		&r_1^*\delta^*i^!\bar g_*1_{\bar X}\ar[rr, "\beta(\bar g)"] & &p'_{1*}1_{X'_1}(-d)[-2d]\\
		r_1^*\delta^*i^*\bar h_*1_{\bar Z}\ar[u]\ar[rr, "\alpha(\bar h)"] & &q'_{1*}1_{Z'_1}\ar[u] 	 
		&r_1^*\delta^*i^*\bar h_*1_{\bar Z}\ar[u, leftarrow](-e)[-2e]\ar[rr, "\alpha(\bar h)"] & &q'_{1*}1_{Z'_1}(-e)[-2e]\ar[u, leftarrow] 
		&r_1^*\delta^*i^!\bar h_*1_{\bar Z}\ar[u]\ar[rr, "\beta(\bar h)"] & &q'_{1*}1_{Z'_1}(-d)[-2d]\ar[u]
	\end{tikzcd}\]
	(where we use the fact that $\bar f^!1_{\bar Z} = \bar f^*1_{\bar Z}(e)[2e]=1_{\bar X}(e)[2e]$ to get the middle diagram).
\begin{proof}
	We have natural isomorphism of functors:
	\begin{align*}
		&r_1^*\delta^*i^*\bar g_* \cong r_1^*\delta^*i^*\bar s_* \bar p_* \cong r_1^*s'_*{\delta'}^*{i'}^*\bar p_* \cong r_1^*s'_*s_{1*}s_1^*{\delta'}^*{i'}^*\bar p_* 
			\cong r_1^*s'_*s_{1*}v^*l^*{i'}^*\bar p_*
			\cong r_1^*s'_*s_{1*} q_{1*}f_{1*}{v''}^*l_X^*
	\end{align*}
	and also an isomorphism of functors:
	\begin{align}\label{eq:internaliso}
		&r_1^*s'_*s_{1*} q_{1*}f_{1*}{v''}^*l_X^* \cong t_1^*s_1^*{s'}^*s'_*s_{1*}q_{1*}f_{1*}{v''}^*l_X^* \cong t_1^*q_{1*}f_{1*}{v''}^*l_X^* \cong q'_{1*}f'_{1*}{t''}^*{v''}^*l_X^* \cong p'_{1*} \theta_X^*
	\end{align}
	by using separatedness and proper base change. We call the composite of the two as $\alpha_{X}$. 
	
	Applying $\alpha_X$ to $1_{\bar X}$, we get the map $\alpha(\bar g)$. Applying $\alpha_Z$ (obtained similarly, by replacing $X$s by $Z$s in the diagram and letting $f$s to be identity) to the diagram $1_{\bar Z}\rightarrow \bar p_*1_{\bar X}$ and using proper base change gives the first commutative diagram, and applying $\alpha_Z$ to the diagram $\bar p_*1_{\bar X}\rightarrow 1_{\bar Z}(-e)[-2e]$ gives the second commutative diagram in the same fashion.
	
	To obtain $\beta(\bar p)$ note that we have a natural isomorphism of functors:
	\begin{align*}
		r_1^*\delta^*i^!\bar g_* &\cong r_1^*\delta^*i^!\bar s_* \bar p_* \cong r_1^*\delta^*s_{Y*}{i'}^! \bar p_* 
			\cong r_1^*s'_*{\delta'}^*{i'}^!\bar p_* \cong r_1^*s'_*s_{1*}s_1^*{\delta'}^*{i'}^!\bar p_* \\
			&\cong r_1^*s'_*s_{1*}v^*l^*{i'}^!\bar p_* \cong r_1^*s'_*s_{1*}v^*l^!{i'}^!\bar p_* 
			\cong r_1^*s'_*s_{1*}v^*q_{Y*}f_{Y*} l_X^!
			\cong r_1^*s'_*s_{1*} q_{1*}f_{1*}{v''}^*l_X^!
	\end{align*}	
	Applying to $1_{\bar X}$ we get an isomorphism:
	\begin{align*}
		r_1^*\delta^*i^!\bar g_*1_{\bar X} \cong r_1^*s'_*s_{1*} q_{1*}f_{1*}{v''}^*l_X^! 1_{\bar X} \cong r_1^*s'_*s_{1*} q_{1*}f_{1*}{v''}^*l_X^* 1_{\bar X}(-d)[-2d]
	\end{align*}
	where we use that $X_1, \bar X$ are regular for the last isomorphism (purity). Combining with the isomorphism \eqref{eq:internaliso} applied to $1_{\bar X}(-d)[-2d]$ we get the isomorphism:
	\[
		r_1^*\delta^*i^!\bar g_*1_X \cong p'_{1*} \theta_X^*1_{\bar X} (-d)[-2d]\cong p'_{1*}1_{X'_1}(-d)[-2d]
	\]
	and this is what we call as $\beta(\bar p)$. The commutative diagram follows by noting that the isomorphisms here are natural -- this is standard for all isomorphisms, except possibly for the purity isomorphism, however in that case the compatibility is a consequence of \cite[Prop. 1.7]{lehalleur2015motivic}.
\end{proof}
\end{lemma}

Now we are in the position to construct the analogue of certain S.\@ Morel's $t$-structure:
\begin{proposition}\label{main:requisitesFor01}
	Fix a scheme $S$. For any $Y\subset S$ a subscheme, define $D_Y := D^{coh}(Y)$, and for any $x=\Spec K\in S$ define $D(K):= DM^{coh}(\Spec K^{perf})$, where $K^{perf}$ denotes the perfect closure of $K$. Then this setup satisfies extended formalism of gluing \ref{gluing:extended4functors} and continuity \ref{gluing:continuity}.
	
	Fix $i\in \{0, 1\}$. Define $D^{\le}(K):= {^wDM^{\le i}}(K^{perf})$ and $D^{>}(K):= {^wDM^{> i}}(K^{perf})$ (see \ref{field:isTstructure} for notation). Then this setup satisfies the formalism for continuity of $t$-structures \ref{gluing:continuityForT}. 
\end{proposition}
\begin{proof}
	For $\delta:\Spec L\rightarrow Y$, and $r_1:\Spec L^{perf}\rightarrow \Spec L$ the pullback map $D_Y\rightarrow D(K)$ is by definition the  $r_1^*\epsilon^*:DM^{coh}(Y)\rightarrow DM^{coh}(K^{perf})$ (this would be the functor ``$\delta^*$'' of \ref{gluing:extended4functors}, but to avoid confusion here $\delta^*$ will always denote the pullback in motivic sheaves). Let $\epsilon:\Spec K\rightarrow S$ be fixed.

	\emph{Extended formalism of gluing and continuity:} This is \ref{coh:continuity}.% That $D_Y=DM^{coh}(Y)$, $D(K) =  satisfies formalism of gluing and continuity is in \ref{coh:continuity}. It satisfies extended formalism of gluing along with continuity also from \ref{coh:continuity}, noting that the pullback $r^*:DM(K)\rightarrow DM(K^{perf})$ is an equivalence of categories preserving cohomological motives, by separatedness and continuity for $DM(-)$.
	
	\emph{Continuity for $D^{\le}(-)$:}

	\emph{Case $i=0$}. Let $a\in D^{\le}(K)={^wDM^{\le 0}}(K^{perf})$. By \ref{fields:dimLessThaniIsGoodEnough} we can assume that $a\in S^{'\le 0}(K^{perf})$, therefore $a = h^{\le 0}(Z')=h(Z')$ for $p':Z'\rightarrow K^{perf}$ finite, smooth. Let the diagram be as in \ref{complicatedDiagram} with $X'=Z'$. Let $\bar a = \bar s_*\bar q_{1*}$. Then:
	\[
		r_1^*\delta^*\bar a = p'_{1*}1_{Z'_1} \in {^wDM^{\le 0}(L^{perf})}
	\]
	using \ref{complicatedDiagram:lemma}, since $Z_1'/L^{perf}$ is finite. For $\delta = \epsilon$ we also get $r^*\epsilon^*\bar a = p'_*1_{Z'} = a$ hence this verifies the condition for $i=0$.
	
	\emph{Case $i=1$}. Let $a\in D^{\le}(K)={^wDM^{\le 0}}(K^{perf})$. By \ref{fields:dimLessThaniIsGoodEnough} we can assume that $a\in S^{'\le i}(K^{perf})$, therefore $a = h^{\le 1}(X')$ for $p':X'\rightarrow K^{perf}$ a smooth projective curve. Let the situation be extended as in \ref{complicatedDiagram}.
	
	Now notice that by \ref{chow:h0ispullback} the map $h(X) \rightarrow h^2(X)$ is induced by the adjunction $h(X')\xrightarrow{h(f')} h(Z')(-1)[-2]$, and hence $a$ can be regarded as third term in the triangle:
	\begin{align*}
		a\rightarrow h(X')=p'_*1_{X'} \rightarrow h(Z')=q'_*1_{Z'}(-1)[-2]\rightarrow.
	\end{align*}
%	and hence $\bar a$ can be regarded as the third term in the triangle:

	Now let $\bar a$ be defined as the third term of the triangle:
	\[
		\bar a \rightarrow \bar g_*1_{\bar X}	\rightarrow \bar h_*1_{\bar Z}(-1)[-2]\rightarrow
	\]
	where the last morphism is obtained via the adjunction $\bar f_*\bar f^!1_{\bar Z}\rightarrow 1_{\bar Z}$ and noting that $f^!1_{\bar Z}=1_{\bar X}(1)[2]$ since $f$ is smooth of relative dimension $1$. 

	Applying $r_1^*\delta^*$, and using \ref{complicatedDiagram:lemma} we get a morphism of triangles:
	\[\begin{tikzcd}
		r_1^*\delta^* \bar a\ar[r]\ar[d]	&r_1^*\delta^*\bar g_*1_{\bar X}\ar[d,"\alpha(\bar g)", "\cong" left] \ar[r] 	&r_1^*\alpha^*\bar h_*1_{\bar Z}(-1)[-2]\ar[r]\ar[d, "\alpha(\bar h)", "\cong" left]	&\-\\ 
		w\ar[r]								&p'_{1*}1_{X'_1}\ar[r]											&q'_{1*}1_{Z'_1}(-1)[-2]\ar[r]	&\-
	\end{tikzcd}\]
	and hence the first map is an isomorphism. It follows that the $r_1^*\delta^*\bar a\cong w$, but using \ref{chow:h0ispullback}, it follows that $w=h^{\le 1}(X'_1)$. In particular, for $\delta = \epsilon$, we get that $r^*\epsilon^*\bar a = h^{\le 1}(X') = a$ and we are done.
	
%	
%	but there are functorial isomorphisms:
%	\begin{align*}
%		r_1^*\delta^* \bar p_*\bar f_*1_{\bar C} \iso r_1^* p_{1*} f_{1*}1_{C_1} = t_1^*s_1^*s_{1*}q_{1*} f_*1_{\bar C}\cong t_1^*q_{1*}f_*1_{\bar C} \cong p'_{1*}f'_{1*}1_{C'_1}
%	\end{align*}
%	using proper base change and separatedness. Same calculation gives, $r_1^*\delta^*p_*1_{\bar Z}(-1)[-2] \cong p'_{1*}1_{Z'_1}(-1)[-2]$, whence $r^*_1\delta^* H'$ can be identified with the natural map
%	\[
%		r^*_1\delta^* \bar H' = H'_1 : p'_{1*}f'_{1*}1_{C'_1} \cong p'_{1*}f'_{1!}f_1^{'!}1_{Z'_1}(-1)[-2]\rightarrow p'_{1*}1_{Z'_1}(-1)[-2]
%	\]
%	obtained by using $f'$ is smooth and proper of relative dimension $1$. Using \ref{chow:h0ispullback} again, we see that this can be identified with the natural map $h(C'_1)\rightarrow  h^2(C'_1)$, hence the third term in the triangle, $r^*_1\delta^*(\bar a)$, can be identified with $h^{\le 1}(C'_1)$ as required. Furthermore, for $\delta = \epsilon$, we see that $r_1^*\delta^*\bar a = h^{\le 1}(C') = a$, and we are done.
%	
	\emph{Continuity for $D^{>}(-)$:} Let $a\in D^{>}(K)={^wDM^{>i}}(K^{perf})$. 

We can assume that $a\in S^{>i}(K^{perf})$, in particular $a=h^{>i}(Z')$, a retract of $b=h(X')=p'_*1_{X'}$. Let $\delta^!:=r_1^*\delta^*i^!$. Now, the closure $\bar \epsilon$ is an irreducible subvariety, and hence for $\bar \delta\ne \bar \epsilon$ we have that $\dim (U)-\dim (Y)=d>0$. Hence in this case, it follows that if $\bar b = \bar p_*1_{\bar X}$:
\begin{align*}
	r_1^*\delta^*\bar b &\cong b &
	\delta^!(\bar b) &\cong p'_{1*}1_{X'_1}(-d)[-2d]
\end{align*}
	using \ref{complicatedDiagram:lemma}. Since $a$ is a retract of $b$, by continuity, restricting $U$ if needed, we can find an $\bar a$ a retract of $\bar b$ such that $r^*\epsilon^* \bar a = a$. But then $\delta^!(\bar a)$ is a retract of $\delta^!(\bar b)$ which is in $DM^{coh}(L^{perf})(-d)\subset {^wDM^{>1}}(L^{perf})\subset {^wDM^{>0}}(L^{perf})$ for $\delta \ne \epsilon$, as required.

\end{proof}
\begin{corollary}\label{tStructure:01}
	Let $i\in \{0, 1\}$. Define
	\begin{align*}
		^wDM^{\le \id+i}(X) &= \{M\in DM^{coh}(X)| s^*(M) \in DM^{\le i}(K)\text{ for all }s:\Spec K\rightarrow X, K\text{ perfect }\} \\
		^wDM^{> \id+i}(X) &= \{M\in DM^{coh}(X)| s^!(M) \in DM^{>i}(K)\text{ for all }s:\Spec K\rightarrow X, K\text{ perfect }\}
	\end{align*}
	Then $({^wDM^{\le \id+i}}(X), {^wDM^{> \id+i}}(X))$ forms a $t$-structure on $DM^{coh}(X)$.
\end{corollary}
\begin{proof}
	This follows immediately from previous proposition and \ref{gluing:mainresult}.
\end{proof}

\begin{proposition}\label{main:requisitesFor2}
	Fix an irreducible scheme $X$ of dimension $3$ and finite type over $k$. For any $Y\subset X$ a subscheme, define $D_Y := D^{coh}_{2}(Y)$ if $\bar Y\ne X$ and $D_Y=DM^{coh}_{3, dom}(Y)$ for $Y$ open. 
	
	For any $x=\Spec K\in X$ and $x$ not the generic point of $X$, if $n_K$ is the transcendence degree of $K/k$, define :%$D(K)$ to be $DM^{coh}_{2-n_K}(\Spec K^{perf})$, where $K^{perf}$ denotes the perfect closure of $K$. Define $D(K)=DM^{coh}_0(\Spec K^{perf})$ for $x=\Spec K$ the generic point. 
%	
%	
%	For $K/k$ of transcendence degree $n_K$, define: 
	\begin{align*}
		D(K)&:=\begin{cases}
			{DM_0^{coh}}(K^{perf})	&n_K = 3\\
			{DM_1^{coh}}(K^{perf})	&n_K = 2,1\\
			{DM_2^{coh}}(K^{perf})	&n_K = 0\\
		\end{cases}	& D^{\le}(K)&:= \begin{cases}
				{^wDM^{\le 0}}(K^{perf})	&n_K = 3\\
				{^wDM^{\le 1}}(K^{perf})	&n_K = 2, 1\\
				{^wDM^{\le 2}_2}(K^{perf})	&n_K = 0\\
				\end{cases} 										\\
		D^{>}(K) &:= D^{\le}(K)^\perp := \{b\in D(K) \text{ such that } \hom(a,b)=0, \forall a\in D^{\le}(K)\}\hspace{-40em}
	\end{align*}
	Then this setup satisfies extended formalism of gluing \ref{gluing:extended4functors}, continuity \ref{gluing:continuity}, and continuity of $t$-structures \ref{gluing:continuityForT}. 
\end{proposition}
\begin{proof}
In below, let $x=\Spec K\in X$ and let $r:\Spec K^{perf}\rightarrow \Spec K$ denote the natural map. Then, for $\epsilon:\Spec K\rightarrow Y$, the pullback map $D_Y\rightarrow D(K)$ is by definition the  $r^*\epsilon^*:DM^{coh}(Y)\rightarrow DM^{coh}(K^{perf})$. For clarity of notation we will also denote $D^{>}(K)$ as $^wDM_2^{>i}(K^{perf})$ (thereby bringing $i$, apriori computable from $n_K$, in the notation).

	\emph{Extended formalism of gluing and continuity:} That this satisfies formalism of gluing is in \ref{coh:continuity}. It satisfies extended formalism of gluing along with continuity also from \ref{coh:continuity}, noting that the pullback $r^*:DM(K)\rightarrow DM(K^{perf})$ is an equivalence of categories preserving preserving $DM_d^{coh}(-)$. 
	
	\emph{$(D^{\le}(K), D^{>}(K))$ form a $t$-structure on $D(K)$:} Note that ${^wDM^{\le i}}(K^{perf})$ is the same as the category ${^wDM^{\le i}_i}(K^{perf})$ for $i=0,1$ and hence $D^{\le}(K)\subset D(K)$. Orthogonality and invariance is a given, one only needs to verify that every object $z$ fits in a triangle $a\rightarrow z\rightarrow b\rightarrow $ in $D(K)$, $a\in D^{\le}(K)$, $b\in D^{>}(K)$, and it is enough to show this for generators, that is motives of the form $h(X)$, with $X$ smooth projective over $K^{perf}$. Then the triangle defining the decomposition for the corresponding $t$-structure in $DM^{coh}(K^{perf})$ (resp. in $DM_2^{coh}(K^{perf})$) is a split triangle with $a\in D^{\le}(K)$. Hence, $b$ being a summand of $z$, also lies in $D(K)$ and being orthogonal to $D^{\le}(K)$ lies in $D^{>}(K)$. 
	
	\emph{Continuity for $D^{\le}$:} Let $a\in D^{\le}(K)$. 
	
	If transcendence degree of $K/k$ is $0$, then $\Spec K$ is a closed point in $X$, and let $i:\Spec K\rightarrow X$ denote the immersion. Then we let $\bar a = i_*\bar a$. 
	
	If the transcendence degree of $K/k$ is $1, 2$, $a\in {^w}DM^{\le 1}(K^{perf})$, then by \ref{main:requisitesFor01}, we can find an $\bar a$ on $U$ such that for any point  $\epsilon:\Spec L\rightarrow U$, $\epsilon^*a\in {^w}DM^{\le 1}(L^{perf})$. The claim follows noting that ${^w}DM^{\le 1}(L^{perf})\subset {^w}DM_2^{\le 2}(L^{perf})$.
	
	If the transcendence degree of $K/k$ is $3$, the same argument holds, noting in addition that ${^w}DM^{\le 0}(L^{perf})\subset {^w}DM^{\le 1}(L^{perf})$.
	
	\emph{Continuity for $D^{>}$:} We have that $\delta^!= r_1^*\delta^*i^!$. %$a\in D^{>}(K)$. %Consider the diagram as 

	Let $b\in {^wDM^{> 0}_2}(K^{perf})$ be given by $h^{> 0}(X')$. Therefore, by \ref{chow:h0ispullback} we have a triangle:
		\[
			q'_*1_{Z'} \rightarrow p'_*1_{X'}\rightarrow b
		\]
	Let the situation be extended as in \ref{complicatedDiagram}. Define $\bar b$ by the triangle:
		\[
			 \bar s_*\bar q_*1_{\bar Z}	\rightarrow \bar s_* \bar p_*1_{\bar X}\rightarrow \bar b\rightarrow
		\]
%	We have a natural isomorphism of functors:
%		\begin{align*}
%				r_1^*\delta^*i^!\bar s_*\bar q_* &= r_1^*\delta^*s_{Y*} i^{'!}q_* = r_1^*\delta^*s_{Y*}l_*l^!i^{'!}q_* = r_1^*\delta^*s_{Y*}l_*q_{Y*}l_Z^! = r_1^*s'_*\delta^{'*}l_*q_{Y*}l_Z^! \\
%					&= r_1^*s'_*s_{1*}s_1^*\delta^{'*}l_*q_{Y*}l_Z^!=r_1^*s'_*s_{1*}v^*l^*l_*q_{Y*}l_Z^! = r_1^*s'_*s_{1*}v^*q_{Y*}l_Z^! = r_1^*s'_*s_{1*}q_{1*}v^{'*}l_Z^!\\
%					&= t_1^*s_1^*s^{'*}s'_*s_{1*}q_{1*}v^{'*}l_Z^! = t_1^*q_{1*}v^{'*}l_Z^! = p'_{1*}t^{'*}_1v^{'*}l_Z^!
%		\end{align*}
	Therefore, applying $\delta^!=r_1^*\delta^*i^!$ to the defining triangle for $\bar b$, and using \ref{complicatedDiagram:lemma}, we get:
	\[
		q'_{1*}1_{Z'_1}(-d)[-2d] \rightarrow p'_{1*}1_{X'_1}(-d)[-2d]\rightarrow \delta^!(\bar b)\rightarrow
	\]
	where $d=\dim U-\dim Y$. Hence $\delta^!(\bar b) = h^{>0}(X'_1)(-d)[-2d]$ by \ref{chow:h0ispullback}. It follows that if $\delta=\epsilon$, then $\delta^!(\bar b) = h^{>0}(X'_1)=b$ as $d=0$. In other cases $d\ge 1$ and $\delta^!(\bar b)[2d] = h^{>0}(X'_1)(-d)\in S_b^{>2}\cup S_c^{>2}[-2d]$, and hence, it is in $^wDM_2^{>2}(L^{perf})\subset {^wDM_2^{>i}(L^{perf})}$ for $i\in\{0,1,2\}$. Hence we conclude that $\bar b$ satisfies the requisite properties for $b$. 
	
	More generally, a generator $a=h^{>i}(X')\in {^wDM^{>i}_2}(K^{perf})$, $i\in\{0,1,2\}$, is a retract of such a $b=h^{>0}(X')$, and hence, restricting $U$ if needed, we can find $\bar a$ a retract of $\bar b$ as above and $\epsilon^!(\bar a) = r^*\epsilon^*(\bar a) = a$. Thus $\delta^!(\bar a)$ is a retract of $\delta^!(\bar b)$, and we are done.
	\end{proof}	

\begin{remark}\label{base:twistOfG1IsG3}
	Proof for continuity of ${^wDM^{> 0}_2}(K^{perf})$ actually shows that for any object $a\in {^wDM^{> 0}}(K^{perf})$ (no restriction of dimensions), we can find an $\bar a$ defined on some neighborhood $U$ such that $\epsilon^!(\bar a) =a$ for $\epsilon:\Spec K\rightarrow U$ generic, and $DM_2^{\le 2}(L^{perf}) \perp \delta^!\bar a$ for any $\delta:\Spec L\rightarrow U$ for $\delta$ not generic (use \ref{field:twistOfG1IsG3}). 
	
	In particular if $b=\pi_*1_{Z}$ with $\pi:Z\rightarrow \Spec L^{perf}$ a proper curve then $b\in DM_2^{\le 2}(L^{perf})$, and hence $b\perp \delta^!\bar a$. This is useful in constructing the relative Picard motive of any morphism.
\end{remark}

\begin{notation}
	We borrow the notations from \cite{arvindVaish}. For any monotone step function on $F$ on non-negative integers ($\le \dim X$), one expects a $t$-structure $({^wDM^{\le F}}(X), {^wDM^{> F}}(X))$ on $DM^{coh}(X)$. These are lifts of corresponding $t$-structures on mixed sheaves as constructed in \cite[3.1.7]{arvindVaish}. The $t$-structures in \ref{tStructure:01} correspond to the functions $\id=n\mapsto n$ and $\id+1=n\mapsto n+1$. 
	
	Let $F$ denote the monotone step function on non-negative integers:
	\[
		F(n)=	\begin{cases}
					2		&n=0,1\\
					3		&n=2,3
				\end{cases}
	\]
\end{notation}
	The previous proposition gives an analogue of $t$-structure for this particular $F$ on an appropriate sub-category of $DM^{coh}(X)$ for a $3$-fold $X$.
\begin{corollary}\label{base:tStructureF}
	Let $X$ be irreducible of dimension $3$. Let $F$ be as above. For any point $\Spec K\in  X$, define $D(K)$, resp. $D^{\le }(K)$, resp. $D^>(K)$ as in previous proposition. In particular $D^{\le}(K) = DM^{\le F(t_K)-t_K}_2(K^{perf})$ where $t_K$ is the transcendence degree of $K/k$ and $D^{>}(K) = D^{\le }(K)^\perp $ inside $D(K)$. Define:
	\begin{align*}
		^wDM^{\le F}(X) &= \{M\in DM_{3,dom}^{coh}(X)| s^*(M) \in D^{\le}(K)\text{ for all }s:\Spec K\in  X\} \\
		^wDM^{> F}(X) &= \{M\in DM_{3,dom}^{coh}(X)| s^!(M) \in D^{>}(K)\text{ for all }s:\Spec K\in X\}.
	\end{align*}
	Then $({^wDM^{\le F}(X)}, {^wDM^{> F}(X)})$ gives a $t$ structure on $DM_{3,dom}^{coh}(X)$.
\end{corollary}
\begin{proof}
	This follows immediately from previous proposition and \ref{gluing:mainresult}.
\end{proof}

\section{Applications}\label{sec:applications}
\subsection{Invariants of singularity}\label{sec:invariants} In \ref{tStructure:01} we constructed the analogue of the $t$-structures due to S.\@ Morel \cite{morelThesis} (or rather the limiting version as exposited in \cite{arvindVaish}), which were denoted there as  $({^wDM^{\le\id+i}(X)},{^wDM^{>\id+i}(X)})$ on $DM^{coh}(X)$ for $i\in\{0,1\}$. One of the motivations of constructing these $t$-structures is that they can be used to construct weight truncated cohomology groups \cite{arvindVaish} which capture certain invariants of singularities on one hand, and relate to the intersection complex on the other. We explore the corresponding relations in motivic world below.

Let $w_{\le\id}$ resp. $w_{\le\id+1}$ denote the corresponding truncation to the negative part of the $t$-structure when $i=0$ resp. $i=1$.

For $i=0$, this can be identified with the $t$-structure $({^wDM^{\le\id}(X)},{^wDM^{>\id}(X)})$ of \cite{vaish2016motivic}. In particular we have the identities \cite[3.2.13, 3.3.4, 4.1.2]{vaish2016motivic}:
\[
	EM_X := w_{\le\id}j_*1_U \cong w_{\le\id}\pi_*1_{\tilde X} \cong w_{\le\id}IM_X
\]
is dependent only on $X$, where $X$ is any scheme, $j:U\hookrightarrow X$ is is an immersion onto an open dense regular subset, $\pi:\tilde X\rightarrow X$ is any resolution of singularities, and $IM_X$ is the motivic intersection complex of $X$ (if it exists) in the sense of \cite{wildeshaus_ic} or the weaker sense of \cite{wildeshaus_shimura_2012}. In particular, $EM_X$ should be thought of as an invariant of singularity. We show below that the same continues to hold for $i=1$ (see \ref{application:invariants}). In fact, we work in a more general framework allowing us to recover several other invariants of singularities motivically. 

Consider any scheme $X$ and consider the subcategories
	\begin{align*}
		DM^{coh}(X)(-1) &:= \{ A(-1)\big| A\in DM^{coh}(X)\} \subset DM^{coh}(X) \\DM^{bir}(X) &:= DM^{coh}(X)/DM^{coh}(X)(-1)
	\end{align*}
	where the quotient is the Verdier quotient of triangulated categories. The category $DM^{bir}(X)$ will be called here as the category of \emph{birational motives} over $X$ (motivation being \cite[4.2.5]{Kahn06102016}) due to the following:
	
	\begin{lemma}\label{application:lemma:birationalIsBirational}
		Let $Q_X:DM^{coh}(X) \rightarrow DM^{bir}(X)$ denote the quotient map. Let $p_j:Y_j\rightarrow X$ for $j\in \{1, 2\}$ be schemes such that $Y_j$ are regular, and there is a birational morphism $q:Y_1\dashrightarrow Y_2$. Then there is an isomorphism $Q_X(p_{1*}1_{Y_1})\cong Q_X(p_{2*}1_{Y_2})$.
%		
%		Let $p:Y\rightarrow X$ be any scheme with $Y$ regular and $j:U\hookrightarrow Y$ be open dense. Then the induced map:
%		\[
%			Q_X(p_*1_Y) \overset\cong\longrightarrow Q_X(p_*j_*1_U)
%		\]
%		is an isomorphism. In particular if $Y-\rightarrow X$ is birational
	\end{lemma}
	\begin{proof}
		It is enough to consider $j:U\subset Y_j$ regular, and show that the induced map:
		\[
			Q_X(p_{j*}1_{Y_j}) \overset\cong\longrightarrow Q_X(p_{j*}j_*1_U)
		\]
		is an isomorphism. It will be simpler to drop the subscript $_j$ from notation below. Let $i:Z\hookrightarrow Y$ denote the closed complement. Then applying $p_{*}$ to the localization triangle for $(U, Z)$ we have the localization triangle:
		\[
			p_*i_*i^! 1_{Y} \longrightarrow p_*1_{Y} \longrightarrow p_*j_* 1_U \rightarrow 
		\]
		By \ref{t-Structure:stabilityOfCohomologicalMotives}(iv) $i^!1_{Y}\in DM^{coh}(Z)(-1)$, and hence by \ref{t-Structure:stabilityOfCohomologicalMotives}(i) $p_*i_*i^! 1_Y\in DM^{coh}(X)(-1)$. Hence applying $Q_X$ the first term goes to $0$, and therefore the second map becomes an isomorphism, as required.
	\end{proof}

	Then, the functors $w_{\le\id+i}$ are actually birational invariants for $i\in \{0,1\}$:
	\begin{proposition}\label{application:wIsBirational}
		For $i\in\{0,1\}$, the functor $w_{\le\id+i}$ factors through $DM^{bir}(-)$. That is there exists a natural diagram of functors:
		\[\begin{tikzcd}
			DM^{coh}(X)\ar[d, "Q_X" left]\ar[r, "w_{\le\id+i}"]	&	DM^{coh}(X)\\
			DM^{bir}(X)\ar[ru, dashed, "w^{bir}_{\le\id+1}" below]
		\end{tikzcd}\]
	\end{proposition}
	\begin{proof}
		Since $w_{\le\id+i}$ preserve triangles, we only need to verify that for any $A\in DM^{coh}(X)(-1)$, we have that $w_{\le\id+i}(A)=0$, equivalently that $A\in {^wDM^{>\id+i}}(X)$. By definition of the glued $t$-structure, it is enough to show that $\epsilon^!(A)\in {^wDM^{>i}}(k)$, for any $\epsilon:\Spec k\hookrightarrow X$. But $\epsilon^!$ commutes with Tate twists and we can then use \ref{field:tateTwistsIncreaseWeights}.
	\end{proof}	

 	The constructions so far can be used to construct various invariants of singularities motivically. But to do so, we will have to understand the relation between realisations and weight truncations:
	
	\begin{para}
		Let $k$ be a field of characteristic zero with a fixed embedding $\sigma:k\hookrightarrow \C$. Then given $X$ a variety over $k$, by the underlying analytic space we mean the analytic space underlying the $\C$-variety $X\times_{k,\sigma}\Spec \C$ and we denote it by $X^{an}$. Then the cohomology of $X^{an}$ carries a mixed Hodge structure. We let $MHS(k)$ denote the category of mixed Hodge Structures over $k$. 
		For such a $k$ we also have realization functors:
		\[
			real: DM_{gm}(k)\rightarrow D^{b}(MHS(k))
		\]
		since we are working with $\Q$-coefficients, due to \cite{ivorra2016perverse}. Then for any variety $p:X\rightarrow \Spec k$ the cohomology $H^{j}(X^{an})$ can be identified with the underlying vector space of $H^{j}(real(p_{*}1_{X}))$.
		
		Also recall that, as observed by S.Morel \cite[\S 3]{morelThesis} (she works with $k$ finite, but the same argument works in this context), the full subcategories
		\begin{align*}
			{^{w}}D^{\le i}(k)&:=\{A\in D^{b}(MHS(k))\mid Gr^{r}_{W}(H^{j}(A))=0, \forall r>i, \forall j\in \Z\}\\
			{^{w}}D^{> i}(k)&:=\{B\in D^{b}(MHS(k))\mid Gr^{r}_{W}(H^{j}(B))=0, \forall r\le i, \forall j\in \Z\}
		\end{align*}
		(where $Gr^{r}_{W}$ denotes the graded for the weight filtration) form a $t$-structure on $D^{b}(MHS(k))$. We let the truncation for this $t$-structure be denoted by $w_{\le i}$ and $w_{>i}$ respectively.
	\end{para}
	Then we have the following proposition:
	\begin{proposition}
		We use the notation of the previous paragraph. Let $A\in DM^{coh}(k)$ and let $i\in \{0,1\}$. Then:
		\begin{align*}
			real(w_{\le i}(A))\cong w_{\le i}(real(A))& &real(w_{> i}(A))\cong w_{> i}(real(A))
		\end{align*}
	
	\end{proposition}
	\begin{proof}
		By construction $H^{j}(real({^w}DM^{\le i})(k))$ is weight $\le i$ (resp. $H^{j}(real({^w}DM^{> i})(k))$ has weights $>i$) -- we can test this on generators of $DM^{\le i}(k)$ (resp. $DM^{> i}(k)$) as specified in \ref{fields:defineCats}. But it is easy to see that $H^{j}(real(h^{\le i}(X)))$ (resp. $H^{j}(real(h^{> i}(X)))$) is $H^{j}(X^{an})$ (resp. vanishes) for $j\le i$ while, it vanishes (resp. is $H^{j}(X^{an})$) for $j>i$, while the underlying $MHS$ on $H^{j}(X^{an})$ for $X$ smooth proper is pure of weight $j$.
			
		In other words, it follows that 
		\begin{align*}
			real({^w}DM^{\le i}(k)) \subset {^{w}D^{\le i}(k)}& &real({^w}DM^{> i}(k)) \subset {^{w}D^{> i}(k)}
		\end{align*}
		where $({^{w}D^{\le i}(k)},{^{w}D^{> i}(k)})$ is S. Morel's $t$-structure \cite[\S 3]{morelThesis}. Applying $real(-)$ to the triangle:
		\[
			w_{\le i}A\rightarrow A \rightarrow w_{>i}A \rightarrow 
		\]
		for some $A\in DM^{{coh}}(k)$, we immediately get a triangle:
		\begin{align*}
			real(w_{\le i}A)\rightarrow real(A) \rightarrow real(w_{>i}A) \rightarrow .
		\end{align*}
		Since the first term is in $^{w}D^{\le i}$ while the last term in $^{w}D^{>i}$, the result follows.
	\end{proof}
	
	Now we can give a couple of examples which are related to constructions in literature:
	\begin{example}[Boundary Complex]\label{example:boundary}
		Fix a scheme $p:X\rightarrow \Spec k$ and $i:Z\subset X$ closed such that $X-Z$ is regular. Then for any resolution of singularities $\pi:\tilde X\rightarrow X$ such that $\pi^{-1}(Z)$ is a simple normal crossing variety, one can associate a CW-complex by taking the vertices $I$ to correspond to the set of irreducible components $\{Z_i|i\in I\}$ of $\pi^{-1}(Z)$ and $J\subset I$ has $r$ $\#J$-faces if $\cap_{j\in J}Z_j$ has $r$ components, the so called \emph{boundary complex} of $(X,Z)$. The cohomology of this complex is independent of chosen resolution -- this can be shown using Hodge theory, see e.g. \cite[4.4]{payne2013boundary}. This can be recovered using the birational motives as follows:
		
		\begin{proposition} The motive $w_{\le 0}p_*i_*i^*\pi_*1_{\tilde X}$ is independent of the choice of the resolution $\tilde X$. If $k\hookrightarrow \C$, $real(w_{\le 0}p_*i_*i^*\pi_*1_{\tilde X})$ computes the cohomology of the Boundary complex. 
		\end{proposition}
		
		\begin{proof}  This is independent of chosen resolution since $p_*, i_*$ etc. preserve Tate twists and cohomological motives, and using \ref{application:wIsBirational}. By proper base change $i_{*}i^{*}\pi_{*}1_{\bar X} = \pi_{*}1_{\bar Z}$. Therefore the motive $p_*i_*i^*\pi_*1_{\tilde X} = q_{*}1_{\bar Z}$ where we let $q: \bar Z\rightarrow \Spec k$ denote the structure morphism. Therefore, it follows by the preceding proposition (and the fact that weight gradation is exact):
		\[
			H^{i}(real(w_{\le 0}p_*i_*i^*\pi_*1_{\tilde X})) \cong H^{i}(w_{\le 0}real(q_{*}1_{\bar Z})) \cong W_{0}H^{i}(\bar Z)
		\]
		where $W_{0}$ is the $0$-th piece of the weight filtration. Since $\bar Z$ is a simple normal crossing variety by assumption, the $0$-th weight piece computes the cohomology of the dual complex (see, for example, \cite[3.13]{arapura2013weights}). 
		\end{proof}
	\end{example}
	\begin{remark}
		Even though the invariant here is constructed using a resolution of singularities, it is possible to work with Galois alterations and hence over $k$ finite. The corresponding analogue of Boundary complex in characteristic $p$ was constructed by \cite{thuillier2007geometrie}. 
	\end{remark}
	\begin{example}[First cohomology of a fiber]\label{example:h1}
		Fix a scheme $p:X\rightarrow \Spec k$ and $i:Z\subset X$ closed. Let $\pi:\tilde X\rightarrow X$ be any resolution. Then $H^1(\pi^{-1}(Z)^{an})$ is independent of the chosen resolution -- this can be shown using Hodge theory, see e.g. \cite[6.2]{arapura2013weights} for case of $Z$ being a point, or more generally it follows from \cite[7.1]{payne2013boundary} for $Z$ smooth (since $H^{1}(-)$ has weights $\le 1$, it is enough to show that $Gr^{i}_{W}$ for $i=0,1$ is an invariant, which is the content of  \cite[7.1]{payne2013boundary}). This can be recovered from the birational motives as follows:
		
		\begin{proposition}
			 The motive $\mathbb H^1(w_{\le 1}p_*i_*i^*1_{\tilde X})$ is independent of the choice of the resolution $\tilde X$, where  the cohomology is computed with respect to the motivic $t$-structure on $1$-motives (due to \cite{MR2102056} on compact motives or \cite{ayoub2011nmotivic} more generally -- see \cite[\S 5.3]{vaish2017punctual} for a discussion).  If $k\hookrightarrow \C$, we have identifications:
			 	\[
					real(\mathbb H^1(w_{\le 1}p_*i_*i^*1_{\tilde X})) \cong H^1(\pi^{-1}(Z)^{an})
				\]
		\end{proposition}
		\begin{proof}
			First note that $^{w}DM^{\le 1}(k)\subset DM^{coh}_{1-mot}(k)$ so that we can apply $\mathbb H^{1}(-)$ and left hand side makes sense. As before, the result is independent of chosen resolution since $p_*, i_*$ etc. preserve Tate twists and cohomological motives, and using \ref{application:wIsBirational}.
			
			%Let $\eta:D^{b}(MHS)\rightarrow D^{b}(k-v.s.)$ denote the forgetful functor from mixed Hodge structures to $k$-vector spaces. 
			Recall that the triangulated category of $1$-motives $DA^{coh}_{1-mot}(k)$ is equivalent to the derived category $D^{b}(\mathcal M_{1}(k))$, where we let $\mathcal M_{1}(k)$ denote the Deligne's category of $1$-motives with rational coefficients (\cite{MR2102056}). This equivalence takes the $t$-structure $t^{1}_{MM}$ on $DA^{coh}_{1-mot}(k)$ to the standard $t$-structure on $D^{b}(\mathcal M_{1}(k))$. It follows that the realization $real$ restricted to $DA^{coh}_{1-mot}(k)$ is $t$-exact for the $t^{1}_{MM}$ on one side and the standard $t$-structure on $D^{b}(MHS(k))$ -- this is a special case of \cite[4.1(v)]{lehalleur2017constructible} where the argument is made for Betti realization, but the key ingredient \cite[7.2.2]{ancona2015motive} also works for the Hodge realization.
			
		As in the previous example, we let $\bar Z := \pi^{-1}Z$ and $q:\bar Z\rightarrow \Spec k$ denote the structure morphism. It follows that we have identifications: 
		\begin{align*}
			real(\mathbb H^1(w_{\le 1}p_*i_*i^*1_{\tilde X})) \cong H^{1}w_{\le 1}real(q_{*}1_{\bar Z})\cong w_{\le 1}H^1(\bar Z^{an}) \cong H^1(\bar Z^{an})
		\end{align*}
		since $w_{\le 1}$ on $D^{b}(MHS(k))$ commutes with taking cohomology, and $H^{1}(-)$ has weights $\le 1$. 
		\end{proof}
%		Motivically this is captured by the birational invariant $\mathbb H^1(w_{\le 1}p_*i_*i^*1_{\tilde X})$ where the cohomology is computed with respect to the motivic $t$-structure $t_{MM}^1$ of \cite[\S 5.3]{vaish2017punctual} on $DA^{coh}_{1-mot}(k)$. 
	\end{example}
	As another application we now generalize the motivic weightless cohomology groups of \cite{vaish2016motivic} to the situation of $w_{\le\id+1}$:
	\begin{proposition}\label{application:EMX:01}
		Let $X$ be any scheme and let $j:U\hookrightarrow X$ be an immersion onto an open dense subset with $U$ regular. Then the motive
		$
			w_{\le\id+i}j_*1_U
		$
		 ($\in \{0,1\}$) is independent of the choice of $U\subset X$ if $U$ is regular. Further, if $p:\tilde X\rightarrow X$ is a resolution of singularity of a reduced scheme $X$ then  
		\[
			w_{\le\id+i}j_*1_U \cong w_{\le\id+i}p_*1_{\tilde X}.
		\]
	\end{proposition}
	\begin{proof}
		Independence from $U$ follows since $w_{\le Id+i}$ factors through $DM^{bir}(X)$. The second claim follows using \ref{application:lemma:birationalIsBirational}. 
	\end{proof}
	\begin{definition}\label{definition:EMX:01}
		We define 
		\[
			EM^{Id+i}_X:=w_{\le\id+i}j_*1_U
		\]
		 for $U\hookrightarrow X$ an immersion onto an open dense subset with $U$ regular. By the result above, it depends only on $X$ and not on $U$.
	\end{definition}
%	\begin{corollary}[Corollary/Definition]\label{application:EMX:01}
%		Let $X$ be any scheme and let $j:U\hookrightarrow X$ be an immersion onto an open dense subset with $U$ regular. Then we define:
%		\[
%			EM^{Id+i}_X:=w_{\le\id+i}j_*1_U
%		\]
%		This, being a birational invariant of $U$ regular, is independent of the choice of $U$. In fact, by \ref{application:lemma:birationalIsBirational} it also follows that if $p:\tilde X\rightarrow X$ denote any resolution of singularity of $X$ for $X$ reduced, then:
%		\[
%			EM^{Id+i}_X\cong w_{\le\id+i}p_*1_{\tilde X}.
%		\]
%	\end{corollary}

Finally, we prove that the constructions here have the expected relationship with motivic intersection complex whenever it exists:	
\begin{para}
	Let $X$ be an arbitrary variety over $k$, and $j:U\hookrightarrow X$ dense immersion with $U$ regular. For any object $N\in DM(U)$ such that $N$ is of weight $0$ (in the sense of \cite{bondarko_weights}), consider the extension $j_{!*}N\in DM(X)$, as defined in \cite[Definition 2.10]{wildeshaus_shimura_2012} (which is a slightly weaker notion than an unconditional definition in \cite{wildeshaus_ic} but exists unconditionally in more cases, for example, the Baily-Borel compactification of an arbitrary Shimura variety). The main interest in the motivic intersection complex comes from the fact that its realization corresponds to the ordinary intersection complex $j_{!*}\mathcal L$ with a local system $\mathcal L$ whenever $N$ realizes to $\mathcal L$.
 
 	The object $j_{!*}N$ is unique up to isomorphism, and we will denote $j_{!*}1_U$ by $IM_X$. It is shown in \cite[2.12]{wildeshaus_shimura_2012} that we have the following properties:
 	\begin{enumerate}
		\item There is no summand of $j_{!*}N$ of the form $i_*L_Z$ with $L_Z$ of weight $0$ in $DM(Z)$, where $i:Z\hookrightarrow X$ denotes the closed complement of $U$.
		\item There is an isomorphism $j^*j_{!*}N \iso N$. The induced map $End(j_{!*}N)\rightarrow End(N)$ has a nilpotent kernel (in case of \cite{wildeshaus_shimura_2012}, it was assumed to be injective in \cite{wildeshaus_ic}). 
		\item $(1)$ and $(2)$ characterize $j_{!*}N$ up to isomorphism.
		\item The following analogue of the decomposition theorem holds: let $M \in DM(X)$ be of weight $0$ in the sense of \cite{bondarko_weights} such that $j^*M \iso N$ for some smooth open dense $j:U\hookrightarrow X$. Then (if $j_{!*}N$ exists) there is a non-canonical isomorphism:
	\[
		M\overset \iso \longrightarrow j_{!*}N \oplus i_* L_Z
	\]
	where $i:Z\hookrightarrow X$ is a proper closed immersion and $L_Z\in DM(Z)$ is of weight $0$.
	\end{enumerate}
\end{para}

	\begin{theorem}\label{application:invariants}
		Let $X$ be any scheme and assume that the motivic intersection complex exist. Then we have:
		\[
			EM^{\id+i}_X \cong w_{\le\id+i}IM_X
		\]
	\end{theorem}
	\begin{proof} The proof is a copy of the argument in \cite[4.1.2]{vaish2016motivic}.

	Let $\pi:Y\rightarrow X$ be an alteration \ref{geometry:djong}, in particular $Y$ is regular  and $\pi$ is proper, generically finite. %and there is $U\subset X$ open dense with $U_{red}$ regular, such that if $W=\pi^{-1}U$, $q=\pi|_W$ is {\color{red} essentially} \'etale. 
	
	Notice that if $A\oplus B = C$, then $j_{!*}A\oplus j_{!*}B$ satisfies $(1)$ and $(2)$ and hence we have by $(3)$ that $j_{!*}C \iso j_{!*}A\oplus j_{!*}B$. Now adjunction induces a morphism, $j_{!*}A \overset{\alpha_A}\longrightarrow j_*A$ and clearly $\alpha_{C}=\alpha_A\oplus\alpha_B$ where $\alpha_B,\alpha_C$ are defined similarly. Therefore if $w_{\le\id+i}\alpha_C$ is an isomorphism, it will follow that $w_{\le\id+i}\alpha_A$ and $w_{\le\id+i}\alpha_B$ are isomorphism. 
	
	By construction \ref{geometry:djong} there is a $U\subset X$ open dense with $U_{red}$ regular, such that if $W=\pi^{-1}U$, $q|_W$ is essentially etale and hence $1_U$ is a retract of $q_*1_W$. Therefore it is enough to show that 
	\[
	 	w_{\le\id+i}j_{!*} 1_W \longrightarrow w_{\le\id+i}j_*q_* 1_W
	\]
	is an isomorphism since $j_{!*} 1_U \iso IM_X$ and $w_{\le\id+i}j_* 1_U \iso EM_X$.
			
	By \cite[1.1]{wildeshaus_ic} $\pi_* 1_Y$ is of weight $0$. Then proper base change tells us that $j^*\pi_* 1_Y \iso q_* 1_W$, and hence by $(4)$ we have a non-canonical isomorphism:
	\[
		\pi_* 1_Y \iso j_{!*}q_* 1_W \oplus i_*L_Z
	\]
	with $i:Z\hookrightarrow X$ a proper closed immersion. 
	
	This implies in particular $j_{!*}q_* 1_W$ and hence also $IM_X$ lie in $DM^{coh}(X)$ (since $1_U$ is a summand of $q_* 1_W$). By contracting $U$ if necessary, we can assume that $Z$ is the complement of $j:U\hookrightarrow X$. Then consider the localization triangle:
	\[
		i_*i^!\pi_* 1_Y \longrightarrow \pi_* 1_Y \longrightarrow j_*q_* 1_W \rightarrow .
	\]
	We apply the functor $w_{\le\id+i}$ to this triangle. Then by \ref{application:invariants} $w_{\le\id+i}i_*i^!\pi_* 1_Y =0$. Since $j^*i_* = 0$, the decomposition  $\pi_* 1_Y \iso j_{!*}q_* 1_W \oplus i_*L_Z$ tells us that $i_*L_Z$ is a summand of the first term in the triangle, and hence $w_{\le\id+i} i_*L_Z = 0$. Hence $w_{\le\id+i} j_{!*}q_* 1_W\iso w_{\le\id+i}\pi_* 1_Y$. Hence also, $w_{\le\id+i}j_{!*}q_* 1_W \iso w_{\le\id}j_*q_* 1_W$ as required.				
	\end{proof}

\subsection{Motivic intersection complex}\label{sec:motivicIC}
We show that if appropriate realization functors exist then the truncations for $t$-structures of \ref{tStructure:01} and \ref{base:tStructureF} indeed realize to the corresponding S.\@ Morel's truncations in the setting of mixed sheaves. We then construct the motivic intersection complex of an arbitrary $3$-fold.

\begin{para}\label{para:421} Assume that $k$ be complex numbers or a finite field and fix a scheme $X$ of finite type over $k$. Under these assumptions there are triangulated categories of mixed sheaves in the sense of \cite{saito_formalisme_2006} enriching the derived category of constructible sheaves over $X$.

Let $D(X)$ be $D^bMHM(X)$ the bounded category of mixed Hodge modules of Saito \cite{saito_introduction_1989,saito_mhm_1990} for $k=\C$ or $D^b_m(X,\Q_l)$ the bounded derived category of mixed complexes of Deligne. We let $\Q_{X}$ denote the constant sheaf on $X$ (that is $\Q_{X} = p^{*}(\Q_{k}$) where $p:X\rightarrow \Spec k$ is the structure map and $\Q_{k}$ is the trivial mixed hodge module for $k=\C$ or the trivial sheaf $\Q_{l}$ over $\Spec k$ for $k$ finite). 

In this setting, given a monotone step function $D$ (e.g. $D=\id := n\mapsto n$, $D=\id+1:= n\mapsto n+1$ or $D=F:=(3\mapsto 3, 2\mapsto 3, 1\mapsto 2, 0\mapsto 2)$) defined on the interval $[0, \dim X]$ we are given a $t$-structure $({^wD}^{\le D}(X), {^wD}^{>D})(X))$ on $D(X)$. We let $w_{\le D}$ denote the truncation to the negative part for this $t$-structure.
	
	For a good monotone step functions $D$, e.g. those satisfying $0\le D\le \dim X$, one defines the weight truncated complex of $X$ \cite[3.1.15]{arvindVaish} as follows:
	\[
		EC^D_X:= w_{\le D}Rj_*\Q_U
	\]
	for any immersion $j:U\hookrightarrow X$ onto an open dense sub-scheme, with $U$ regular. While this is independent of $U$ for any $D$, it was shown \cite[4.1.2]{arvindVaish} that if $D=\id$:
	\[
			EC_X: = EC^{\id}_X = w_{\le \id}R\pi_*\tilde X
	\] 
	for any resolution of singularity $\pi:\tilde X\rightarrow X$. This continues to hold for $D=\id+1$ by similar arguments, and is the motivation for the motivic version \ref{application:wIsBirational} for us. It fails to hold for higher $D$.
\end{para}

	We have constructed the motivic analogues $EM^D_X$ for $D=\id, \id+1$ and $X$ arbitrary. For any arbitrary $3$-fold, it is possible to construct $EC^D_X$ for $D=F$ as well, and we do this next:
\begin{proposition}\label{application:EMX:F}
		Let $X$ be irreducible with $\dim X=3$. For $F=d\mapsto \begin{cases}3 &d\in\{3,2\} \\ 2 &d\in\{1,0\}\end{cases}$, and $j:U\hookrightarrow X$ an immersion onto an open dense subset with $U$ regular. Then we define:
		\[
			EM^F_X := w_{\le F}j_*1_U.
		\]
		where the truncation is with respect to the $t$-structure constructed in \ref{base:tStructureF}. 
		
		Then $EM^F_X$ is independent of the choice of $U$. 
\end{proposition}
\begin{proof}
	Note that $j_*1_U\in DM^{coh}_{dom, 3}(X)$ and hence the truncation makes sense.
	
	Let $h:V\hookrightarrow U$ be an open dense immersion and $g:Z\hookrightarrow U$ be the closed complement. Then $j_*$ applied to the localization triangle gives:
	\[
		j_*g_*g^! 1_U \rightarrow j_*1_U \rightarrow j_*h_*1_V \rightarrow 
	\]
	and hence it is enough to show that $j_*g_*g^!1_U\in {^wDM^{>F}}(X)$. By definition of the glued $t$-structure, it is enough to show that $g^!1_U\in {^wDM^{>F}}(Z)$. 

	To see this, let $\epsilon:\Spec K\rightarrow Z$ be any point. Let $Y$ be closure of $\Spec K$ in $Z$ and $U_1\subset Y$ be open dense regular. Let $r:U_1\rightarrow Z$ denote the immersion, and let $\epsilon_1:\Spec K\hookrightarrow U_1$ denote the natural map. Then:
	\[
		\epsilon^!g^!1_U = r_K^*\epsilon_1^*r^!g^! 1_U =r_K^*\epsilon_1^* 1_{U_1}(-i)[-2i]\in DM^{coh}(K^{perf})(-i)
	\]
	where $i = \dim U-\dim U_1 = 3-t_K$ where $t_K=\text{transcendence degree of }K/k$ and we have $r_K:\Spec K^{perf}\rightarrow \Spec K$ denote the natural map. Let $D^{\le}(K)$ and $D^{>}(K)$ be as in \ref{base:tStructureF}.
	
	Therefore, in particular, using \ref{field:tateTwistsIncreaseWeights} for $n=2,1$, and the definition of $S_c^{>2}$ \ref{fields:defineCats} for $n=0$, we see that $D^{\le}(K) \perp \epsilon^!g^!1_U$ and hence we conclude that $g^!1_U \in {^wDM^{>F}(U)}$ as required. 
	\end{proof}

\begin{para}
	Let $X$ be any irreducible scheme. A motivation for the construction of of $EC^D_X$ was that they approximate the intersection complex $IC_X$, and in special cases, are equal to it. For $D$ a monotone step function $\id\le D\le \dim X$, one always has morphisms (\cite[3.2.2]{arvindVaish}):
	\[
		\Q_X \rightarrow EC_X \rightarrow EC^D_X \rightarrow IC_X.
	\]
	For $D$ large enough, the last map is an isomorphism. Define:
	\[
		D_-:=\begin{cases}d &d=\dim X\\ d-1&d<\dim X\end{cases}.
	\] Then we have:
	\begin{lemma}\label{application:lemma:truncatesToIC}
		Let $X$ be irreducible. For any monotone step function $D_-\le D\le \dim X$, $EC^D_X\cong IC_X$.
	\end{lemma}
	\begin{proof}
		By \cite[3.1.13]{arvindVaish} $EC^{D_-}_X := w_{\le D_-}IC_X =IC_X$ using irreducibility of $IC_X$ and \cite[2.2.12]{arvindVaish}. In particular, we have that $IC_X \in {^wD^{\le D_-}}\subset {^wD^{\le D}}$ (the inclusion since $D_-\le D$). Since by \cite[3.1.13]{arvindVaish} we also have $EC^D_X = w_{\le D}IC_X$ for any $D\le \dim X$, we conclude that $EC^D_X=IC_X$ as required.
	\end{proof}
\end{para}

	We have the following compatibility result between constructions in the motivic setting and that in mixed sheaves \cite{arvindVaish}:	
\begin{proposition}\label{application:weightTruncationsArePreserved}
	Let $k=\C$ or $k$ finite and we assume the notation of \ref{para:421}. Assume that for all schemes $X$ over $k$ we have a triangulated realization functor ${real}:DM(X) \rightarrow D(X)$ commuting with the Grothendieck's four functors (including pullbacks of not necessarily finite type) such that $real( 1_X)=\Q_X$. Then, for $D\in \{\id,\id+1\}$ and $X$ arbitrary or $D=F$ and $\dim X\le 3$, $X$ irreducible, we have:
		\[
			real(EM^{D}_X)=EC^{D}_X.
		\]
\end{proposition}
\begin{remark*}
		For the result to hold, it is important to have realization functors to a mixed category (that is with a notion of weights). For example, Betti realization of Ayoub--Zucker \cite{ayoub2010note} is not good enough for our purpose. 
		
		Such realizations are known for $k$ finite due to \cite{ivorra2007realisation}. Even though mixed realizations for $k\hookrightarrow \C$ are also known due to \cite{ivorra2016perverse} but in that case the compatibility with the four functors of Grothendieck is not known.

\end{remark*}
%\begin{remark*}
%	By assumption \ref{para:421} $D(X)$ is a mixed category of sheaves, that is, has a notion of weights. This is clearly needed for constructions in \cite{arvindVaish} to go through. Thus, for example, Betti realization of \cite{ayoub2012relative} will not be good enough for us. However, when $k$ is finite, such functors do exists due to \cite{ivorra2007realisation}. 
% \end{remark*}
\begin{proof}
	Consider the defining triangle for $EM_X^D$:
	\[	
		EM^D_X	\rightarrow j_* 1_U	\rightarrow w_{>D} j_*1_U \rightarrow .
	\]
	Then applying realization, using commutativity with $j_*$, we get:
	\[
		real(EM^D_X)	\rightarrow Rj_*\Q_U \rightarrow real(w_{>D}j_*1_U) \rightarrow
	\]
	and it will be enough to show that the first term lies in ${^wD^{\le D}}(X)$ and the last term in ${^wD^{>D}}(X)$, whence $real(EM^D_X) = w_{\le D}Rj_*\Q_U$ which is $EC^D_X$ by definition. 
	
	But $real(w_{>D}j_*1_U)\in real({^wDM^{> D}})$, and it will be enough to show that we have an inclusion $real({^wDM^{> D}}(X))\subset {^wD^{> D}}(X)$. The case for $real(EM^D_X)$ is similar.
	
	Let $A\in {^wDM^{> D}}(X)$. It is enough to show that $j:U\hookrightarrow X$ open such that $real(j^*A)\in {^wD^{>D}}(U)$ -- then it follows by Noetherian induction that $real(i^!A)\in {^wD^{>D}}(Z)$ where $i:Z\hookrightarrow X$ is the closed complement (and the base case also holds for then, $U=X$). The result would then follow since $real$ commutes with $j^*$ and $i^!$ and $({^wD^{\le D}}(X), {^wD^{>D}}(X))$ is a glued $t$-structure (use \cite[2.1.20]{vaish2016motivic}).
	
	 Let $\dim X=d$. Let $\epsilon:\Spec K\rightarrow X$ be a generic point and let $q:\Spec K^{perf}\rightarrow \Spec K$ be the natural map. Any object $A\in {^wDM^{> D}}(X)$, then 
	\begin{align*}
		q^*\epsilon^*(A) \in \begin{cases}
				{^wDM^{> 0}}(K^{perf}) &(d=3, D=F)\text{ or }D=\id \\
				{^wDM^{> 1}}(K^{perf}) &(d=2,1, D=F)\text{ or }D=\id+1 \\
				{^wDM^{> 2}}(K^{perf}) &(d=0, D=F)
			\end{cases}
	\end{align*}
		by definition of the glued $t$-structure. Generators of ${^wDM^{> i}}(K^{perf})$ are of the form $h^{>i}(Y)$ for $Y$ proper smooth over $K^{perf}$ (by definition \ref{fields:defineCats}). Note that $\mathbb H^j(real(h^{>i}(Y))) = H^j(Y)$ for $j> i$ and $0$ otherwise by definition of Chow-K\"unneth projectors.
		
		Now $Y$ must be defined over a finite purely inseparable extension $L$ of $K$, say $Y_L/L$ smooth proper, and the situation descends to the following:
		\[\begin{tikzcd}
			Y\ar[d, "p'"]\ar[rd]						&\\
			\Spec K^{perf} \ar[rd]\ar[rdd, "q" left] 	&Y_L\ar[d, "p"]\ar[r, hookrightarrow]		&\bar Y\ar[d, "\bar p"] \\
														&\Spec L\ar[d, "s"]\ar[r, hookrightarrow]	&U'\ar[d, "\bar s"] \\
														&\Spec K\ar[r, hookrightarrow, "\epsilon_U"]		&U\ar[r, hookrightarrow, "\circ" description]	& X
		\end{tikzcd}\]
		where $\bar Y$ and $U'$ are obtained by spreading out $s$ and $p$. In particular we can choose $U$ small enough that $U, U'$ are regular, and hence so is $\bar Y$. Since $a_j:=h^{>j}(Y)$ is a summand of $a:=h(Y)$, by continuity, shrinking $U$ if needed, there is $\bar a_j$, a summand of $\bar a:=(\bar s\bar p)_*1_{\bar Y}$, such that $q^*\epsilon_U^*\bar a_j = a_j$. Now, by decomposition theorem \FIXME{Reference?}:
		\begin{align*}
			real(\bar a_j)\hookrightarrow real(\bar a) = R(\bar s\bar p)_*\Q_{\bar Y} = \oplus_iR^i(\bar s\bar p)_*\Q_{\bar Y}[-i]\twoheadrightarrow \oplus_{j>i}R^{i}(\bar s\bar p)_*\Q_{\bar Y}[-i]
		\end{align*}
		where the last map is the truncation $x\rightarrow \tau_{>j}x$ for the standard $t$-structure on $D(X)$.
		The composite becomes an isomorphism after applying $q^*\epsilon^*$ by the way $a_j$ is defined. Consider the triangle in $D(U)$, defining $C$:
		\[
			real(\bar a_j)\rightarrow \oplus_{i>j}R^{i}(\bar s\bar p)_*\Q_{\bar Y}[-i] \rightarrow C \rightarrow .
		\]
		Then $q^*\epsilon^* C = 0$. Hence $\epsilon^*C=0$, and hence $C$ must be trivial on a neighborhood of $\Spec K$. It follows that, by shrinking $U$ if needed, we have:
		\[
			real(\bar a_j) \cong \oplus_{j>i}R^{i}(\bar s\bar p)_*\Q_{\bar Y}[-i]
		\]
		
		Now $q^*\epsilon^*A$ is obtained by applying finitely many shifts, and extensions to objects of the form $a_j$. In particular, using continuity, shrinking $U$ if necessary, $j^*A$ is obtained by applying finitely many shifts and extensions to objects of the form $\bar a_j$. Hence it is enough to show that $real(\bar a_j) \in {^wD^{>D}}(U)$, in particular enough to show that $R^{i}(\bar s\bar p)_*\Q_{\bar Y}\in {^wD^{>D}}(U)$ for $j>i$. By definition, this is the case if the perverse sheaf $R^i(\bar s\bar p)_*\Q_{\bar Y}[d]$ is of weights greater than $D(d)$, that is the local system $R^i\bar p_*\Q_{\bar Y}$ is of weights greater than $D(d)-d$.
		
		A direct computation shows that $i = D(d) - d$ in all the three cases (i.e. $D=\id, \id+1$ or $F$) and the claim follows since for a proper morphism $\bar p$ with $\bar Y$ regular, $R^j\bar p_* \Q_{\bar Y}$ is pure of weight $j$, but $j>i=D(d)-d$ as required.
		
		A ditto calculation shows that $real({^wDM^{\le D}}(X))\subset {^wD^{\le D}}(X)$, and we are done.
\end{proof}
\begin{remark}
	Although it simplified notations for us, the use of decomposition theorem in the above proof is not necessary, it is enough to observe that $Rf_*\Q_Y$ is obtained by finitely many extensions and shifts of $R^jf_*\Q_Y$ for any $f:Y\rightarrow U$ of finite type.
\end{remark}
It is now immediate that $EM^F_X$ is the motivic intersection complex for an arbitrary $F$:
\begin{corollary}\label{application:ICfor3folds}
	Assume that for all schemes $X$ we have triangulated realization functors $real:DM(X) \rightarrow D(X)$ commuting with the Grothendieck's four functors (including pullbacks of not necessary finite type) such that $real( 1_X)=\Q_X$. Assume that $X$ is any irreducible $3$-fold. Then we have:
	\[
		real(EM^F_X) \cong IC_X
	\]
	where right hand side denotes the intersection complex of $X$.
\end{corollary}
\begin{proof}
	Follows immediately from \ref{application:lemma:truncatesToIC} and \ref{application:weightTruncationsArePreserved}.
\end{proof}
\bibliographystyle{alpha}
\bibliography{ehp1}{}	
\end{document}